\documentclass[11pt]{amsart}

\usepackage{amsmath,amssymb,amsthm}
\usepackage{graphicx}
\usepackage[all]{xy}
\usepackage{overpic}

\newtheorem{theorem}{Theorem}[section]
\newtheorem{proposition}[theorem]{Proposition}
\newtheorem{lemma}[theorem]{Lemma}
\newtheorem{corollary}[theorem]{Corollary}

\theoremstyle{definition}
\newtheorem{definition}[theorem]{Definition}
\newtheorem{example}[theorem]{Example}
\newtheorem{remark}[theorem]{Remark}

\newtheorem{introtheorem}{Theorem}

\numberwithin{equation}{section}
\numberwithin{figure}{section}
\numberwithin{table}{section}

\newcommand{\Z}{\mathbb{Z}}
\newcommand{\Q}{\mathbb{Q}}
\newcommand{\R}{\mathbb{R}}
\newcommand{\A}{\mathcal{A}}
\newcommand{\TT}{\mathsf{T}}
\newcommand{\M}{\mathcal{M}}
\newcommand{\G}{\mathcal{G}}

\newcommand{\F}{\mathcal{F}}
\newcommand{\D}{\mathrm{D}}
\newcommand{\C}{\mathrm{C}}
\newcommand{\CC}{\mathcal{C}}
\renewcommand{\SS}{\mathfrak{S}}
\newcommand{\SSS}{\mathcal{S}}
\newcommand{\T}{\mathcal{T}}
\newcommand{\LL}{\mathcal{L}}
\newcommand{\OO}{\mathcal{O}}
\newcommand{\K}{\mathrm{K}}

\newcommand{\Ker}{\operatorname{Ker}}
\renewcommand{\Im}{\operatorname{Im}}
\newcommand{\Int}{\operatorname{Int}}
\newcommand{\Mon}{\mathrm{Mon}}
\newcommand{\Mag}{\mathrm{Mag}}
\newcommand{\ifil}{\operatorname{i-filter}}
\newcommand{\ideg}{\operatorname{i-deg}}
\newcommand{\edeg}{\operatorname{e-deg}}
\newcommand{\Lk}{\operatorname{Lk}}
\newcommand{\Cr}{\operatorname{Cr}}
\newcommand{\Sym}{\operatorname{Sym}}

\newcommand{\Aut}{\operatorname{Aut}}

\newcommand{\Homeo}{\operatorname{Homeo}}
\newcommand{\MJ}{\operatorname{MJ}}
\newcommand{\Lie}{\operatorname{Lie}}
\newcommand{\sgn}{\operatorname{sgn}}
\newcommand{\comm}{\operatorname{comm}}
\newcommand{\Gr}{\operatorname{Gr}}

\newcommand{\Zt}{\widetilde{Z}}
\newcommand{\Cob}{\mathcal{C}\mathit{ob}}
\newcommand{\LCob}{\mathcal{LC}\mathit{ob}}

\newcommand{\ZK}{Z^{\mathrm{K}}}

\newcommand{\ZKLMO}{Z^{\mathrm{K\text{-}LMO}}}
\newcommand{\tsA}{{}^{\mathit{ts}}\!\mathcal{A}}
\newcommand{\btT}{{}_{\mathit{b}}^{\mathit{t}}\mathcal{T}}

\newcommand{\Cub}{\mathcal{C}\mathit{ub}}

\newcommand{\ang}[1]{\left\langle{#1}\right\rangle}
\newcommand{\angg}[1]{\langle\!\langle{#1}\rangle\!\rangle}
\newcommand{\fc}[1]{\lfloor{#1}\rceil}
\newcommand{\curvearrowru}{\rotatebox[origin=c]{180}{$\curvearrowleft$}}
\newcommand{\op}{\mathrm{op}}
\newcommand{\Id}{\mathrm{Id}}

\newcommand{\raisegraph}[3]{\raisebox{#1}{\includegraphics[height=#2]{#3}}}

\newcommand{\strutgraph}[2]{\hspace{-0.2em}\raisebox{-0.7em}{ 
\begin{overpic}[height=2em]{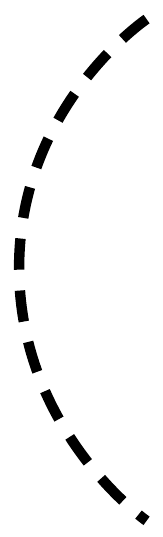}\scriptsize
 \put (35,80){$#1$}
 \put (35,0){$#2$}
\end{overpic}}\hspace{1em}}

\newcommand{\dstrutgraph}[2]{\raisebox{-0.7em}{
\begin{overpic}[height=2em]{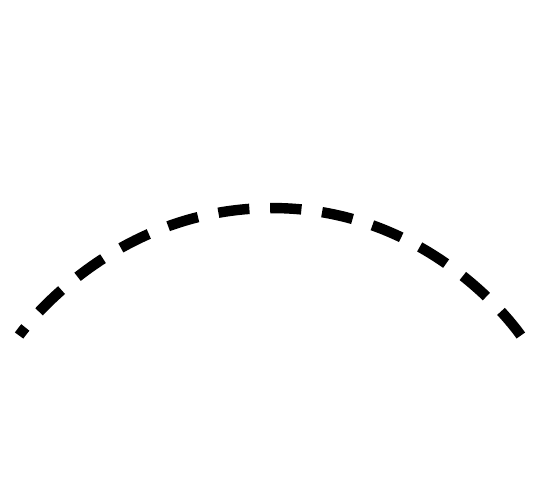}\scriptsize
 \put (-10,-5){$#1$} 
 \put (85,-5){$#2$} 
\end{overpic}}\hspace{0.5em}}

\newcommand{\ldstrutgraph}[2]{\raisebox{-0.7em}{
\begin{overpic}[height=2em]{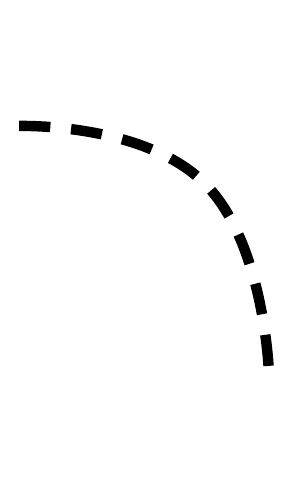}\scriptsize
 \put (35,-5){$#1$}
 \put (-5,90){$#2$}
\end{overpic}}}

\newcommand{\lustrutgraph}[2]{\raisebox{-0.7em}{
\begin{overpic}[height=2em]{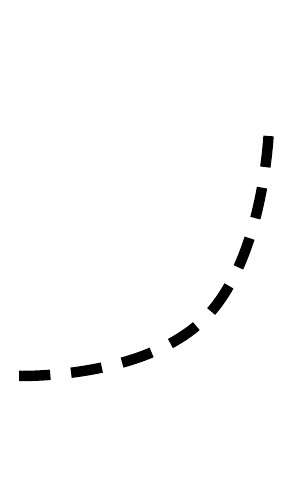}\scriptsize
 \put (35,80){$#1$}
 \put (-5,-15){$#2$}
\end{overpic}}}

\newcommand{\Ygraph}[3]{\raisebox{-0.7em}{
\begin{overpic}[height=2em]{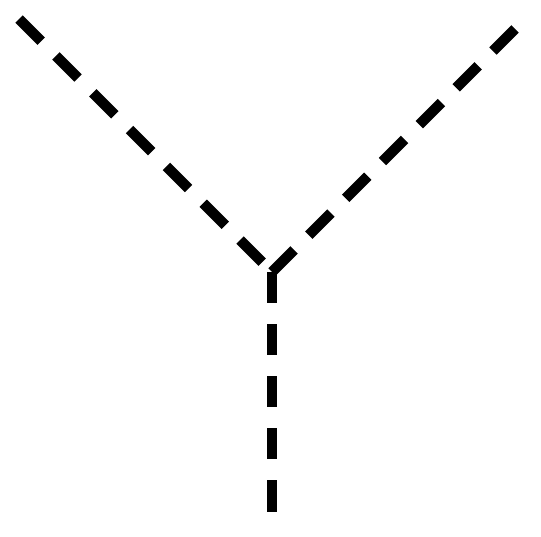}\scriptsize
 \put (0,105){$#1$}
 \put (90,105){$#2$}
 \put (60,0){$#3$}
\end{overpic}}}

\newcommand{\dYgraph}[3]{\raisebox{-0.7em}{
\begin{overpic}[height=2em]{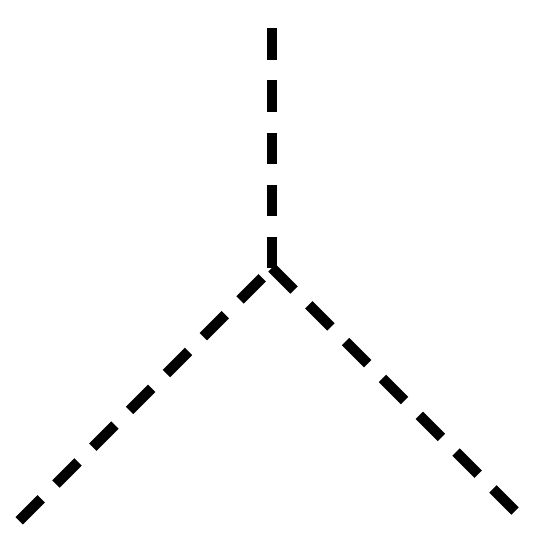}\scriptsize
 \put (60,80){$#1$}
 \put (0,-30){$#2$}
 \put (90,-30){$#3$}
\end{overpic}}}

\newcommand{\uYgraph}[3]{\hspace{0.2em}\raisebox{-0.7em}{
\begin{overpic}[height=2em]{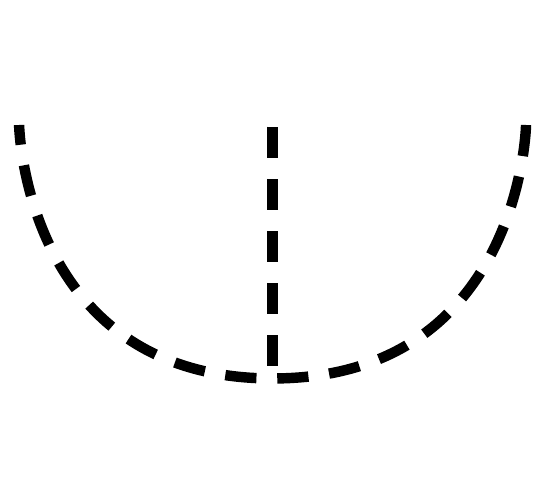}\scriptsize
 \put (-10,80){$#1$}
 \put (40,80){$#2$}
 \put (90,80){$#3$}
\end{overpic}}}

\newcommand{\Phigraph}[2]{\hspace{-0.2em}\raisebox{-0.7em}{
\begin{overpic}[height=2em]{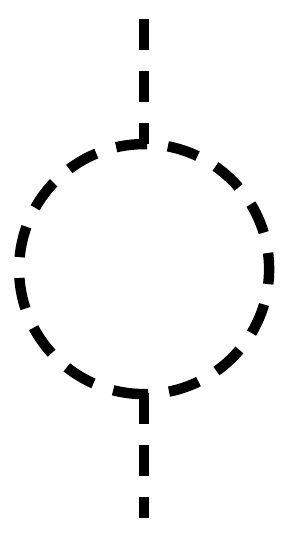}\scriptsize
 \put (40,80){$#1$}
 \put (40,0){$#2$}
\end{overpic}}\hspace{0.5em}}

\newcommand{\dPhigraph}[2]{\raisebox{-0.7em}{
\begin{overpic}[height=2em]{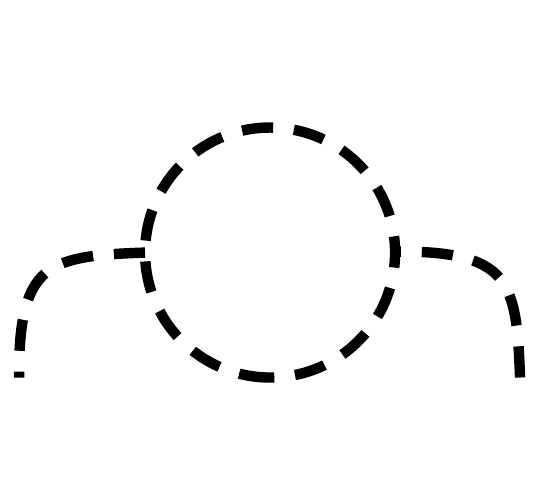}\scriptsize
 \put (-10,-5){$#1$}
 \put (80,-5){$#2$}
\end{overpic}}\hspace{0.3em}}

\newcommand{\lrPhigraph}[2]{\hspace{-0.1em}\raisebox{-0.7em}{
\begin{overpic}[height=2em]{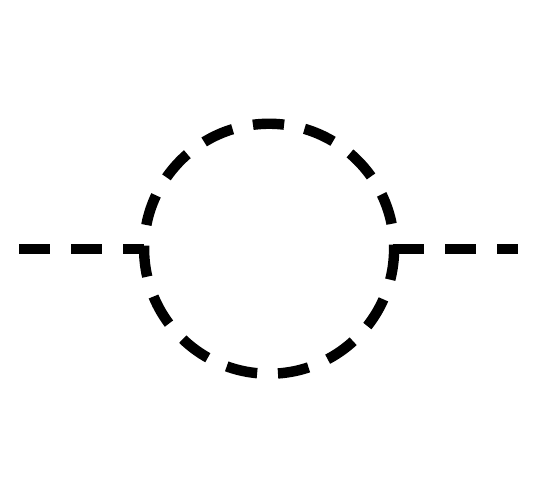}\scriptsize
 \put (-5,55){$#1$}
 \put (85,55){$#2$}
\end{overpic}}\hspace{0.3em}}

\newcommand{\ldPhigraph}[2]{\hspace{-0.2em}\raisebox{-0.7em}{
\begin{overpic}[height=2em]{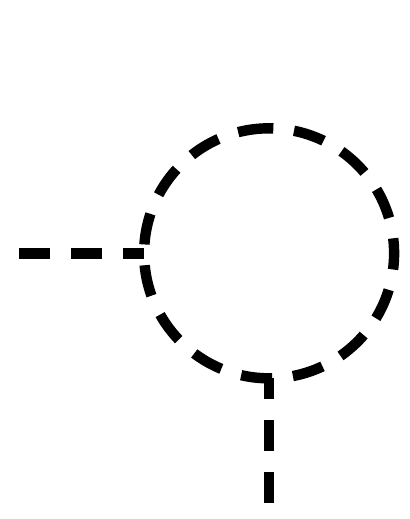}\scriptsize
 \put (60,-10){$#1$}
 \put (-5,65){$#2$}
\end{overpic}}}

\newcommand{\luPhigraph}[2]{\hspace{-0.1em}\raisebox{-0.7em}{
\begin{overpic}[height=2em]{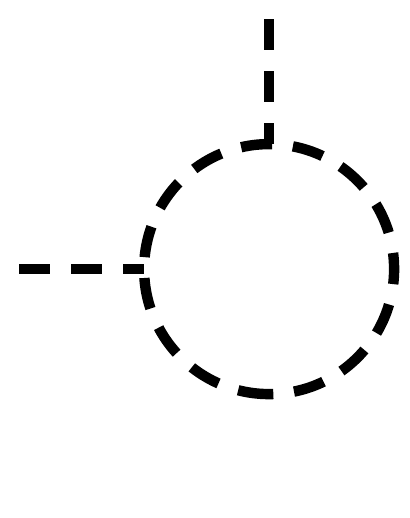}\scriptsize
 \put (60,90){$#1$}
 \put (-5,60){$#2$}
\end{overpic}}}

\newcommand{\rTgraph}[3]{\raisebox{-0.7em}{
\begin{overpic}[height=2em]{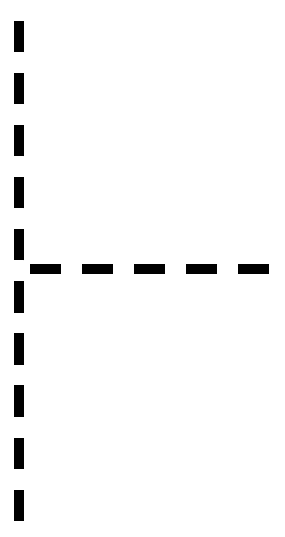}\scriptsize
 \put (10,85){$#1$}
 \put (10,0){$#2$}
 \put (45,60){$#3$}
\end{overpic}}\hspace{0.4em}}

\newcommand{\lTgraph}[3]{\raisebox{-0.7em}{
\begin{overpic}[height=2em]{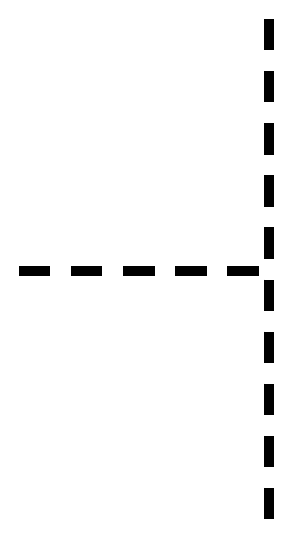}\scriptsize
 \put (55,80){$#1$}
 \put (55,0){$#2$}
 \put (0,65){$#3$}
\end{overpic}}\hspace{0.4em}}

\newcommand{\rPigraph}[4]{\raisebox{-0.7em}{
\begin{overpic}[height=2em]{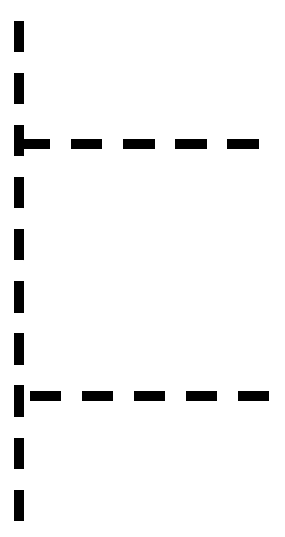}\scriptsize
 \put (5,100){$#1$}
 \put (5,-20){$#2$}
 \put (55,70){$#3$}
 \put (55,10){$#4$}
\end{overpic}}\hspace{0.7em}}

\newcommand{\lPigraph}[4]{\raisebox{-0.7em}{
\begin{overpic}[height=2em]{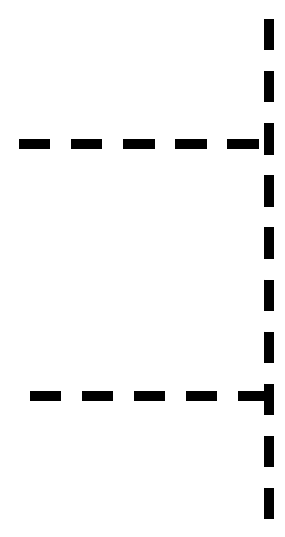}\scriptsize
 \put (60,100){$#1$}
 \put (60,-25){$#2$}
 \put (0,85){$#3$}
 \put (0,-10){$#4$}
\end{overpic}}\hspace{0.8em}}

\newcommand{\Hgraph}[4]{\hspace{0.8em}\raisebox{-0.7em}{
\begin{overpic}[height=2em]{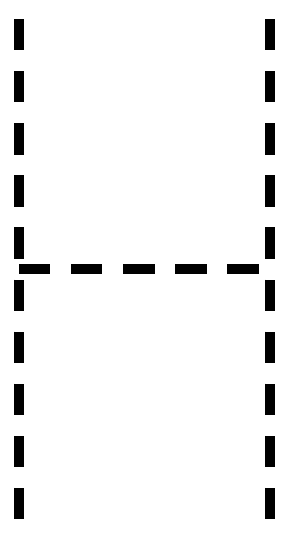}\scriptsize
 \put (-50,80){$#1$}
 \put (60,80){$#2$}
 \put (-50,0){$#3$}
 \put (60,0){$#4$}
\end{overpic}}\hspace{0.6em}}

\newcommand{\lrHgraph}[4]{\raisebox{-0.7em}{
\begin{overpic}[height=2em]{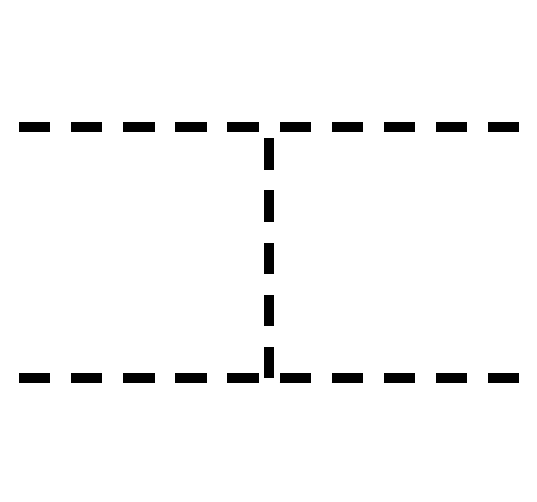}\scriptsize
 \put (-5,80){$#1$}
 \put (80,80){$#2$}
 \put (-5,-15){$#3$}
 \put (80,-15){$#4$}
\end{overpic}}}

\newcommand{\uHgraph}[4]{\raisebox{-0.9em}{
\begin{overpic}[height=2em]{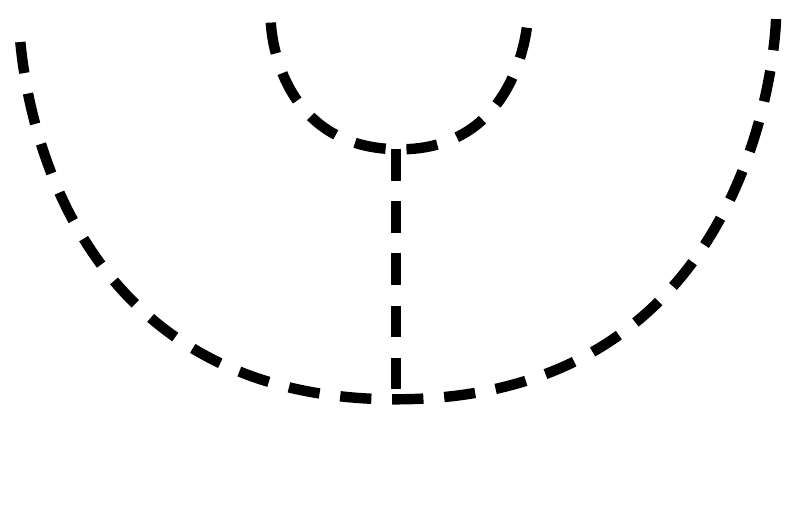}\scriptsize
 \put (-5,70){$#1$}
 \put (25,70){$#2$}
 \put (60,70){$#3$}
 \put (90,70){$#4$}
\end{overpic}}}

\newcommand{\ldHgraph}[4]{\hspace{0.4em}\raisebox{-0.7em}{
\begin{overpic}[height=2em]{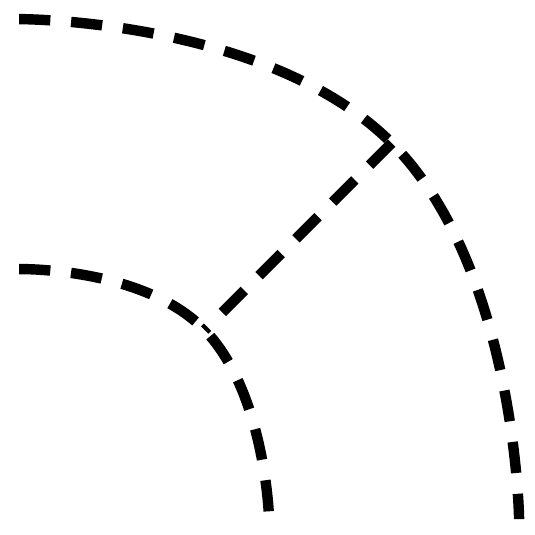}\scriptsize
 \put (30,-30){$#1$}
 \put (80,-30){$#2$}
 \put (-30,90){$#3$}
 \put (-30,45){$#4$}
\end{overpic}}}

\newcommand{\luHgraph}[4]{\hspace{0.4em}\raisebox{-0.7em}{
\begin{overpic}[height=2em]{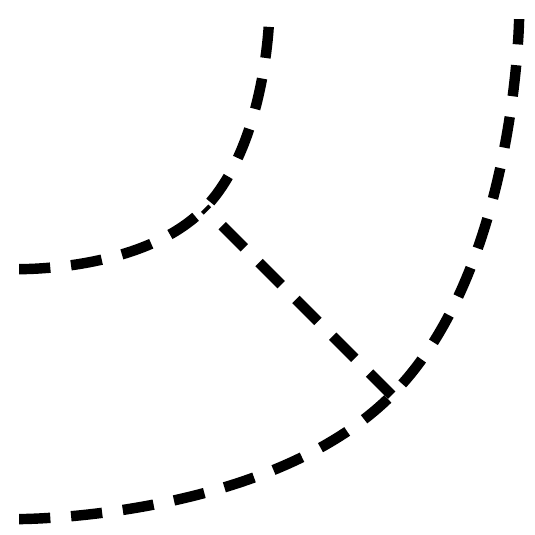}\scriptsize
 \put (30,105){$#1$}
 \put (80,105){$#2$}
 \put (-30,45){$#3$}
 \put (-30,0){$#4$}
\end{overpic}}}

\newcommand{\lambdagraph}[4]{\raisebox{-0.7em}{
\begin{overpic}[height=2em]{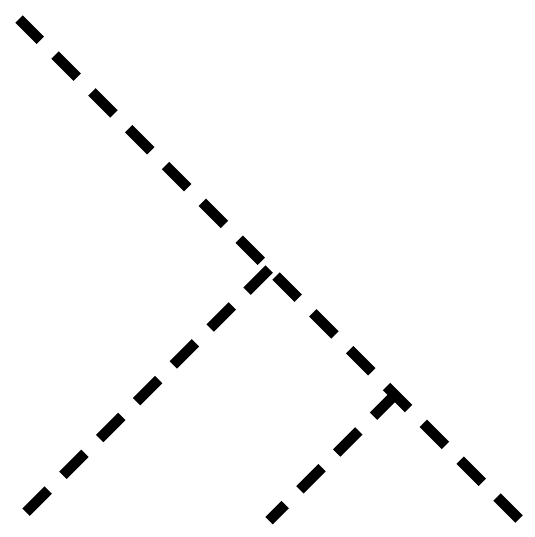}\scriptsize
 \put (-10,105){$#1$}
 \put (-10,-30){$#2$}
 \put (40,-30){$#3$}
 \put (90,-30){$#4$}
\end{overpic}}}

\newcommand{\mlambdagraph}[4]{\raisebox{-0.7em}{
\begin{overpic}[height=2em]{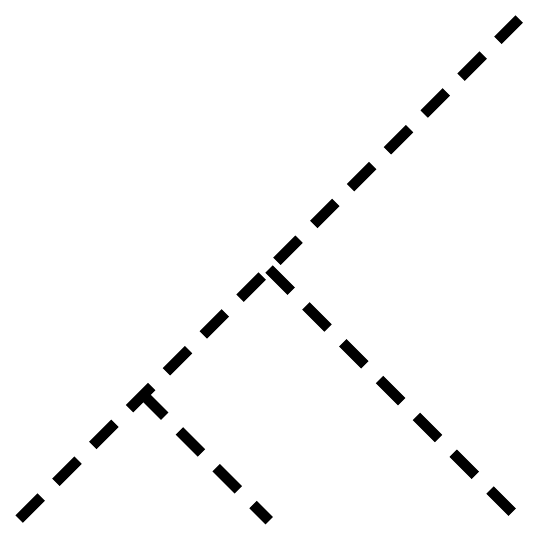}\scriptsize
 \put (90,105){$#1$}
 \put (-10,-30){$#2$}
 \put (40,-30){$#3$}
 \put (90,-30){$#4$}
\end{overpic}}}

\newcommand{\yengraph}[4]{\raisebox{-0.7em}{
\begin{overpic}[height=2em]{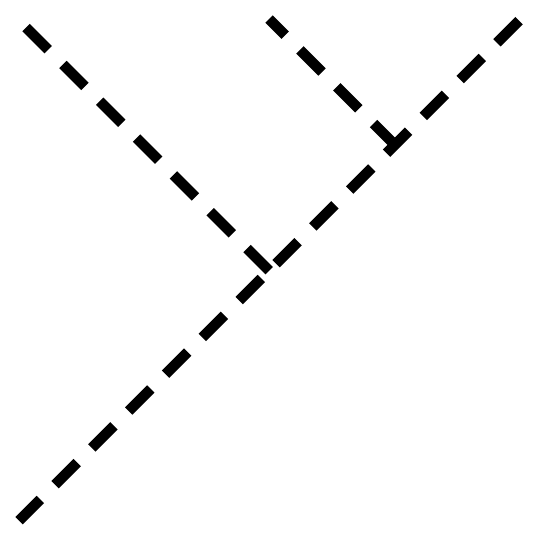}\scriptsize
 \put (-5,105){$#1$}
 \put (45,105){$#2$}
 \put (90,105){$#3$}
 \put (-5,-30){$#4$}
\end{overpic}}}

\newcommand{\myengraph}[4]{\raisebox{-0.7em}{
\begin{overpic}[height=2em]{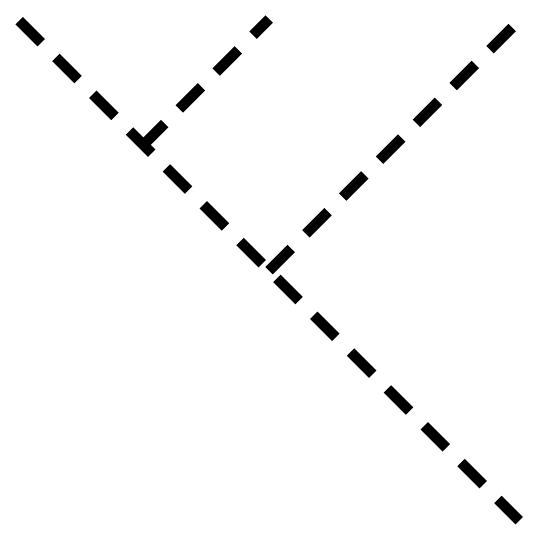}\scriptsize
 \put (-5,105){$#1$}
 \put (45,105){$#2$}
 \put (90,105){$#3$}
 \put (90,-30){$#4$}
\end{overpic}}}

\newcommand{\chairgraph}[4]{\raisebox{-0.7em}{
\begin{overpic}[height=2em]{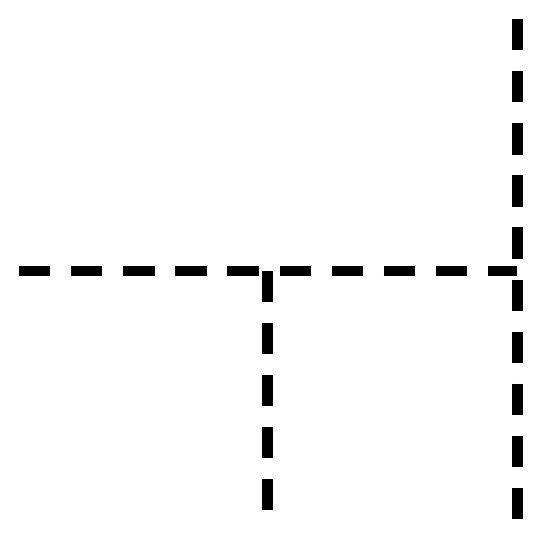}\scriptsize
 \put (80,105){$#1$}
 \put (30,-25){$#2$}
 \put (80,-25){$#3$}
 \put (0,60){$#4$}
\end{overpic}}}

\newcommand{\uchairgraph}[4]{\raisebox{-0.7em}{
\begin{overpic}[height=2em]{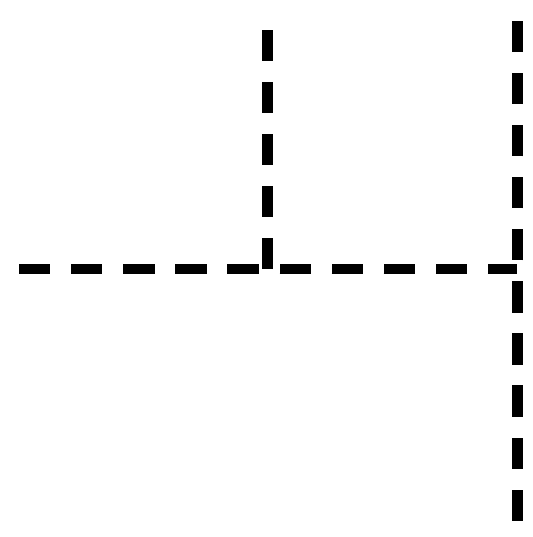}\scriptsize
 \put (40,105){$#1$}
 \put (90,105){$#2$}
 \put (90,-25){$#3$}
 \put (0,60){$#4$}
\end{overpic}}\hspace{0.7em}}

\newcommand{\BCHgraph}[3]{\hspace{-0.2em}\raisebox{-1.7em}{
\begin{overpic}[height=3.8em]{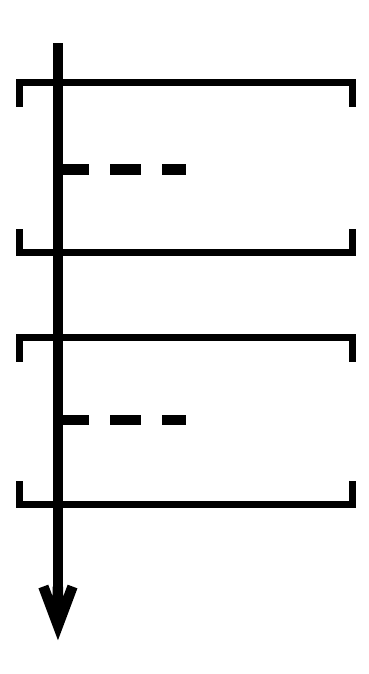}\scriptsize
 \put (33,69){$#1$}
 \put (33,31){$#2$}
 \put (18,3){$#3$}
\end{overpic}}\hspace{0.2em}}

\begin{document}

\title[An extension of the LMO functor]{An extension of the LMO functor}
\author[Y. Nozaki]{Yuta Nozaki}
\address{Graduate School of Mathematical Sciences, The University of Tokyo \endgraf
3-8-1 Komaba, Meguro-ku, Tokyo, 153-8914 \endgraf
Japan}
\email{nozaki@ms.u-tokyo.ac.jp}

\maketitle

\begin{abstract}
 Cheptea, Habiro and Massuyeau constructed the LMO functor, which is defined on a certain category of cobordisms between two surfaces with at most one boundary component.
 In this paper, we extend the LMO functor to the case of any number of boundary components, and our functor reflects relations among the parts corresponding to the genera and boundary components of surfaces.
 We also discuss a relationship with finite-type invariants and Milnor invariants.
\end{abstract}

\setcounter{tocdepth}{1} 
\tableofcontents

\section{Introduction}\label{sec:Intro}
In the early 1990s, Kontsevich \cite{Kon93} defined the Kontsevich invariant (the universal finite-type invariant) of knots by the integral on the configuration space of finite distinct points in $\mathbb{C}$.
All rational-valued Vassiliev invariants are recovered from the Kontsevich invariant through weight systems.

In the late 1990s, Ohtsuki \cite{Oht96b} showed that one can consider an arithmetic expansion of the quantum $SO(3)$-invariant of rational homology spheres.
The result of this expansion is called the perturbative $SO(3)$-invariant.
Ohtsuki \cite{Oht96a} also introduced integer-valued finite-type invariants of integral homology spheres.
Kricker and Spence \cite{KrSp97} proved that the coefficients of the perturbative $SO(3)$-invariant are of finite-type. 
On the other hand, the perturbative $SO(3)$-invariant was extended to the perturbative $PG$-invariant for any simply connected simple Lie group $G$, where $PG$ is the quotient Lie group of $G$ by its center.
Moreover, using the Kontsevich invariant, Le, Murakami and Ohtsuki \cite{LMO98} introduced the LMO invariant of connected oriented closed 3-manifolds.
It is known that the LMO invariant is universal among perturbative invariants of rational homology spheres.

Bar-Natan, Garoufalidis, Rozansky and Thurston \cite{BGRT02a,BGRT02b,BGRT04} gave an alternative construction of the LMO invariant of rational homology spheres by introducing the \r{A}rhus integral that is also called the formal Gaussian integral.
In these papers, it is suggested that the \r{A}rhus integral can be extended to an invariant of tangles in a rational homology sphere, which is called the Kontsevich-LMO invariant in \cite{HabN00} and \cite{CHM08}.
Using the Kontsevich-LMO invariant, Cheptea, Habiro and Massuyeau \cite{CHM08} defined the LMO functor as a functorial extension of the LMO invariant.
In fact, the value for a rational homology cube $M$ (in which case the boundary of $M$ is $S^2$) coincides with the LMO invariant of the closed 3-manifold obtained from $S^3\setminus\Int[-1,1]^3$ and $M$ by gluing their boundaries, see \cite[Section~3.5]{CHM08}.
One of the advantage of the LMO functor is that we can use its functoriality to calculate its values and to prove its properties.

\begin{figure}[h]
 \centering
\begin{minipage}{0.45\columnwidth}
 \centering
 \includegraphics[height=6em]{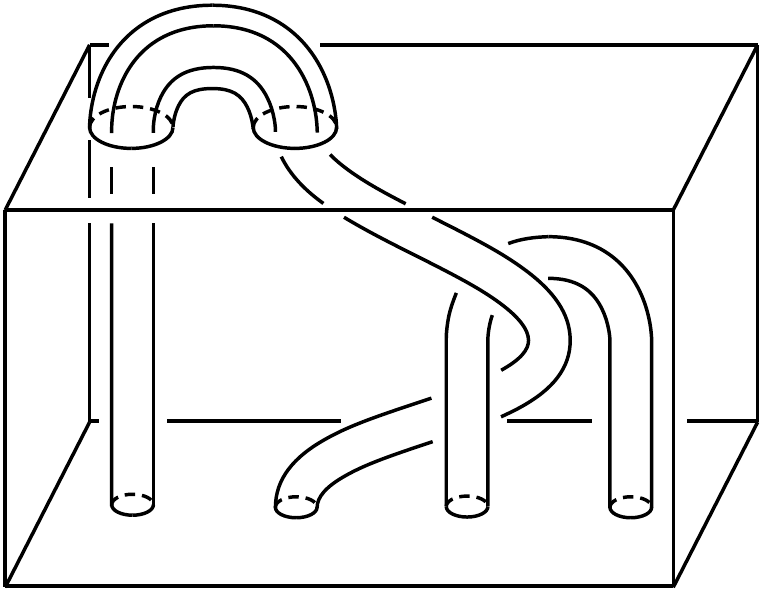}
\end{minipage}
\begin{minipage}{0.45\columnwidth}
 \centering
 \includegraphics[height=6em]{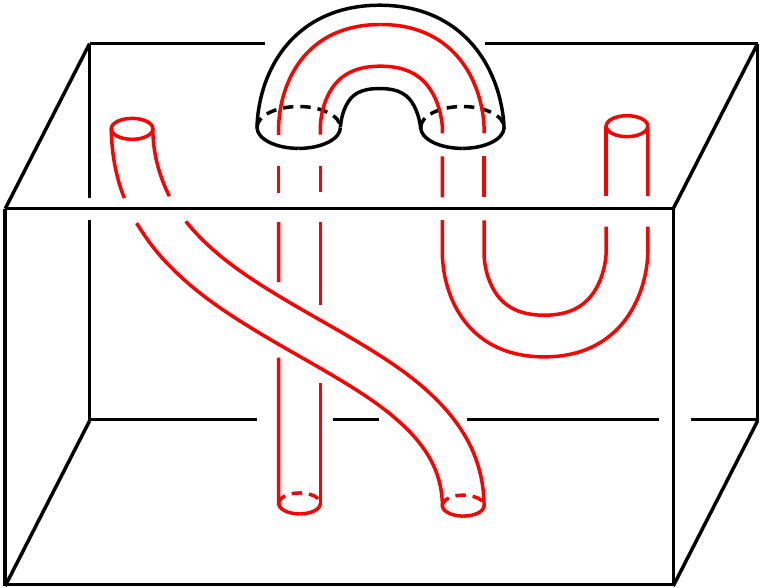}
\end{minipage}
 \caption{A cobordism introduced in \cite{CHM08} (left) and a new one (right) both with parametrizations on their boundaries}
 \label{fig:OldNewCob}
\end{figure}%

The LMO functor is defined for a connected, oriented, compact 3-manifold regarded as a certain cobordism between two surfaces.
Here, these surfaces are assumed to be with at most one boundary component.
The purpose of this paper is to construct an extension of the LMO functor to the case of \emph{any} number of boundary components (compare two figures in Figure~\ref{fig:OldNewCob}).
It is expected that this extension enables us to introduce many \emph{categorical} operations on cobordisms, for instance, which corresponds to the pairing or shelling product defined in \cite{ABMP10}.

We define the braided non-strict monoidal category $\LCob_q$ to be Lagrangian $q$-cobordisms extended as above (see Definition~\ref{subsec:qCob} and Remark~\ref{rem:braided}).
The main result is the following theorem.

\begin{introtheorem}[Theorem~\ref{thm:Zt}]
 There is a tensor-preserving functor $\Zt\colon \LCob_q \to \tsA$ between two monoidal categories, which is an extension of the LMO functor.
\end{introtheorem}

A generating set of $\LCob_q$ as a monoidal category is determined in Proposition~\ref{prop:Gen} and the values on them are listed in Table~\ref{tab:Value}.
Therefore, the functoriality and tensor-preservingness of $\Zt$ enable us to compute the value on a Lagrangian $q$-cobordism by decomposing it into the generators.
It should be emphasized that there are diagrams colored by both $+$ (or $-$) and $0$ in Table~\ref{tab:Value}, which implies that our extension is non-trivial.

Habiro \cite{HabK00} and Goussarov \cite{GGP01} introduced claspers and clovers respectively that play a crucial role in a theory of finite-type invariants of arbitrary 3-manifolds.
In \cite{CHM08}, using claspers, it was shown that the LMO functor is universal among rational-valued finite-type invariants.
We prove that our LMO functor $\Zt$ has the same property.

\begin{introtheorem}[Theorem~\ref{thm:universal}]
 $\Zt$ is universal among rational-valued finite-type invariants.
\end{introtheorem}

In \cite{Mil54}, Milnor defined an invariant of links, which is called the Milnor $\overline{\mu}$-invariant.
Habegger and Lin \cite{HaLi90} showed that this invariant is essentially an invariant of string links.
In \cite{CHM08}, it was proven that the tree reduction of the LMO functor is related to Milnor invariants of string links.
By extending the Milnor-Johnson correspondence (see, for example, \cite{HabN00}) suitably, we show that the same is true for $\Zt$. 

\begin{introtheorem}[Theorem~\ref{thm:MilnorInv}]
 Let $(B,\sigma)$ be a string link in an integral homology sphere $B$.
 Then the first non-trivial term of $\Zt^{Y,t}(\MJ_w^{-1}(B,\sigma))$ is determined by the first non-vanishing Milnor invariant of $(B,\sigma)$, and vice versa.
\end{introtheorem}

Finally, in \cite{ChLe06}, \cite{ABMP10} and \cite{Kat14}, one can find related researches from different points of view from this paper.

\subsection*{Organization of this paper}
In Section~\ref{sec:CobBTtangle}, we define cobordisms and bottom-top tangles that are main object in this paper.
Section~\ref{sec:JacobiDiagOp} is devoted to reviewing Jacobi diagrams and the formal Gaussian integral.
In Section~\ref{sec:KontsevichLMOinv}, the Kontsevich-LMO invariant of tangles in a homology cube is explained, which plays a key role in the subsequent sections.
The main part of this paper is Section~\ref{sec:ExtLMOfunc}, where we construct an extension of the LMO functor.

In Section~\ref{sec:GenValue}, we shall give a generating set of the category $\LCob$ and calculate the values on them.
These values will be used later.
Section~\ref{sec:Universal} is devoted to reviewing clasper calculus and proving the universality among finite-type invariants.
Finally, in Section~\ref{sec:Knot}, we apply our LMO functor $\Zt$ to some cobordisms arising from knots or string links.
In particular, the relationship between $\Zt$ and Milnor invariants of string links will be discussed.

\subsection*{Notation}
Throughout this paper, we denote by $H_\ast(X)$ the homology groups of a topological space $X$ with coefficients in $\Z$.
(However, arguments in Sections~\ref{sec:CobBTtangle}--\ref{sec:ExtLMOfunc} are valid for the case of coefficients in $\Q$, see Remark~\ref{rem:Rational}.)
Let $\fc{n}^\ast$ denote the set $\{1^\ast,\dots,n^\ast\}$, where $\ast$ is $+,\ -,\ 0$ or we omit $\ast$.
Let $X$ and $Y$ be geometric objects equipped with the ``top'' and ``bottom'', for example, cobordisms or tangles.
Then, the composition of $X$ and $Y$ is always defined by stacking the bottom of $Y$ on the top of $X$.

We use almost the same notation and terminology as in \cite{CHM08}.
However, their definitions are suitably extended.

\subsection*{Acknowledgements}
The author would like to thank Takuya Sakasai, who provided him with helpful comments and suggestions.
He also feels thankful to Kazuo Habiro, who helped him to prove Proposition~\ref{prop:Gen}.
He would also like to thank Gw\'ena\"el Massuyeau for pointing out many crucial mistakes in his first draft and recommending him to introduce the functor $\K$ in Remark~\ref{rem:Kill}.
Finally, this work was supported by the Program for Leading Graduate Schools, MEXT, Japan.

\section{Cobordisms and bottom-top tangles}\label{sec:CobBTtangle}
In this section, we give the definition of cobordisms and a way to express them as certain tangles.

\subsection{Cobordisms}\label{subsec:Cob}
We first introduce some notation.
Let $\Mon(x_1,\dots,x_k)$ denote the free monoid generated by letters $x_1,\dots,x_k$.
Similarly, we denote by $\Mag(x_1,\dots,x_k)$ the free magma generated by $x_1,\dots,x_k$.
We call an element of $\Mon(x_1,\dots,x_k)$ or $\Mag(x_1,\dots,x_k)$ a \emph{word}.
Let $w$ be a word.
We denote by $w^{x_i}$ the element obtained from $w$ by replacing all letters except for $x_i$ with the empty word $e$.
Let $|w|$ denote the word length of $w$.
For example, if $w=\bullet((\circ\bullet)\bullet) \in \Mag(\bullet,\circ)$, then $|w^\bullet|=|\bullet(\bullet\bullet)|=3$.

Next, we prepare two kinds of surfaces.
Let $w \in \Mon(\bullet,\circ)$ and let $g=|w^\bullet|$, $n=|w^\circ|$.
We define the oriented compact surface $F_w$ in $\R^3$ as illustrated in Figure~\ref{fig:F}, which has handles and inner boundary components corresponding to letters $\bullet$ and $\circ$ in $w$, respectively.
Moreover, let $\alpha_1,\dots,\alpha_g$, $\beta_1,\dots,\beta_g$ and $\delta_1,\dots,\delta_n$ denote the oriented simple closed curves based at $\ast$ as drawn in Figure~\ref{fig:F}.
We often regard the above closed curves as free loops.

\begin{figure}[h]
 \centering
 \includegraphics[width=0.6\columnwidth]{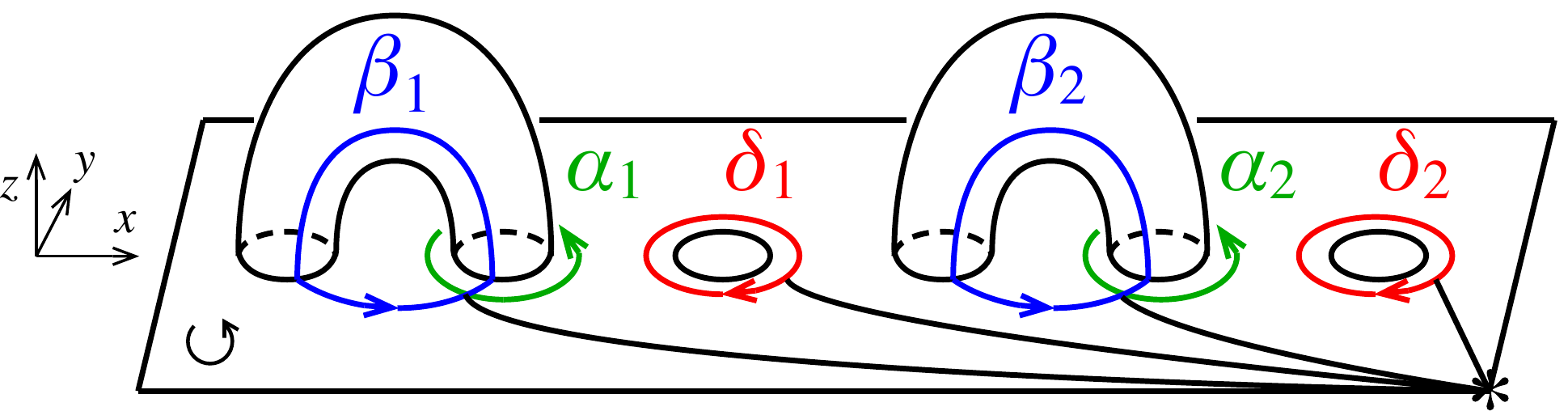}
 \caption{An example of $F_w$ with $w=\bullet\circ\bullet\:\circ$} 
 \label{fig:F}
\end{figure}

Let $w_{+},w_{-} \in \Mag(\bullet,\circ)$ such that $|w_{+}^\circ|=|w_{-}^\circ|=:n$, and let $\sigma$ be an element of the $n$th symmetric group $\SS_n$.
We define the \emph{reference surface} $R^{w_+}_{w_-,\sigma}$ as illustrated in Figure~\ref{fig:R}, which is an oriented closed surface consisting of four kinds of surfaces: the \emph{top surface} $F_{w_+}$, the \emph{bottom surface} $-F_{w_-}$, $n$ tubes and four sides, where $-F_{w_-}$ is the surface $F_{w_-}$ with the opposite orientation.
The boundary of the tubes are attached to the inner boundary of the top surface and bottom surface according to the permutation $\sigma$.

\begin{figure}[h]
 \centering
 \includegraphics[width=0.6\columnwidth]{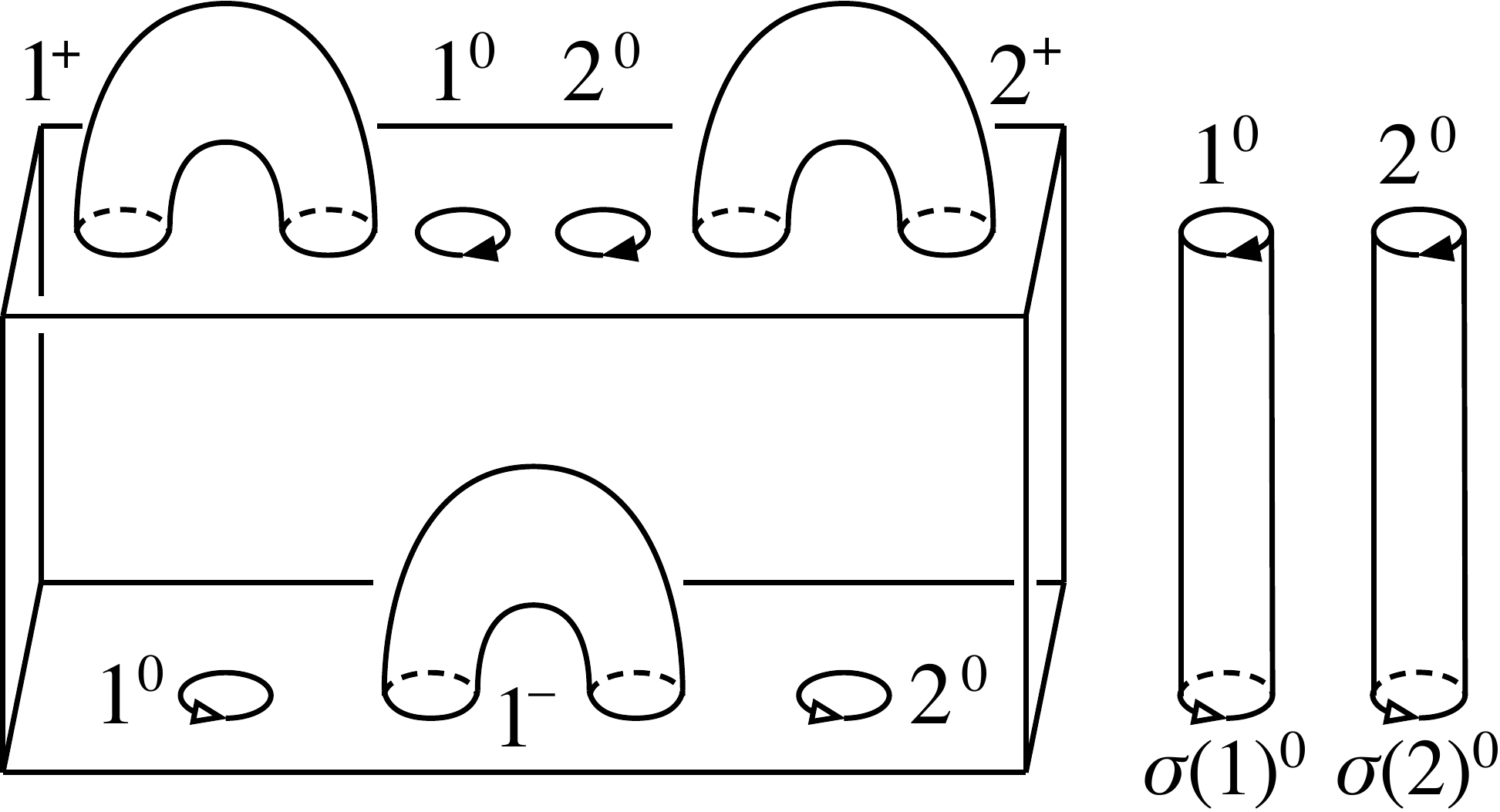}
 \caption{An example of $R^{w_+}_{w_-,\sigma}$ with $w_{+}=\bullet\circ\circ\:\bullet$, $w_{-}=\circ\bullet\circ$}
 \label{fig:R}
\end{figure}

\begin{definition}
 Let $w_+, w_- \in \Mon(\bullet,\circ)$ such that $|w_+^\circ|=|w_-^\circ|=:n$.
 A \emph{cobordism} from $F_{w_+}$ to $F_{w_-}$ is an equivalence class of triples $(M,\sigma,m)$, where
\begin{itemize}
 \item $M$ is a connected, oriented, compact 3-manifold,
 \item $\sigma$ is an element of the $n$th symmetric group $\SS_{n}$,
 \item $m\colon R^{w_+}_{w_-,\sigma} \to \partial M$ is an orientation-preserving homeomorphism.
\end{itemize}
 Here, $(M,\sigma,m)$ is equivalent to $(N,\tau,n)$ if $\sigma=\tau$ and there is an orientation-preserving homeomorphism $f\colon M \to N$ satisfying $f|_{\partial M}\circ m=n$.
 
 Let $m_\pm$ denote the restriction of $m$ to $\pm F_{w_\pm}$.
\end{definition}

We define the category $\Cob$ of cobordisms as follows: 
The set of objects is the monoid $\Mon(\bullet,\circ)$.
A morphism from $w_+$ to $w_-$ is a cobordism from $F_{w_+}$ to $F_{w_-}$.
The composition of $(M,\sigma,m) \in\Cob(v,u)$ and $(N,\tau,n) \in \Cob(w,v)$ is defined by
\[(M,\sigma,m)\circ(N,\tau,n) := (M \cup_{m_+\circ n_-^{-1}} N, \sigma\tau, m_{-} \cup n_{+}) \in \Cob(w,u),\]
namely, the composition is obtained by stacking the bottom of $N$ on the top of $M$.
The identity of $w$ is the cobordism $(F_w\times[-1,1],\Id_{\SS_{|w^\circ|}},\Id)$, where the last $\Id$ means the obvious parametrization.

Moreover, $\Cob$ is a strict monoidal category, where the tensor product of two cobordisms is their horizontal juxtaposition and the unit is the empty word $e$.

\begin{remark}\label{rem:Kill}
 We denote by $\Cob^\bullet$ the full subcategory whose objects belong to the monoid $\Mon(\bullet)$, namely, $\Cob^\bullet$ is the category $\Cob$ defined in \cite{CHM08}.
 The full functor $\K\colon \Cob \to \Cob^\bullet$ is defined to be killing the extra boundary, that is, for $M=(M,\sigma,m)\in \Cob(w_+,w_-)$, $\K(M)$ is obtained by attaching $|w_{+}^\circ|$ tubes to $M$.
\end{remark}

\begin{definition}
 Let $w \in \Mon(\bullet,\circ)$ such that $|w^\circ|=n$, and let $\ast_1,\dots,\ast_n$ be points on each inner boundary component.
 We recall that $\ast$ is the point on the outer boundary as drawn in Figure~\ref{fig:F}.
 The \emph{mapping class group} of $F_w$ is defined by
\begin{align*}
 \M(F_w) &:= \left\{h \in \Homeo_{+}(F_w) \mathrel{}\middle|\mathrel{} \parbox{13em}{$h(\ast)=\ast, \\ h(\{\ast_1,\dots,\ast_b\})=\{\ast_1,\dots,\ast_b\}$} \right\} \bigg/ \text{isotopy},
\end{align*}
 where an isotopy fixes the set $\{\ast,\ast_1,\dots,\ast_n\}$ pointwise.
 For any $h \in \M(F_w)$, we denote by $\C(h)$ the \emph{mapping cylinder} of $h$, that is, the cobordism
 \[(F_w\times[-1,1], \pi(h), \Id\times(-1)\cup h\times 1),\]
 where $\pi(h) \in \SS_n$ is defined by $h(\ast_i)=\ast_{\pi(h)(i)}$.
 Note that an isotopy does not necessarily fix $\partial F_w$, however, by definition, $\C(h)$ is well-defined.
 One can show that the map
 \[\C\colon \M(F_w) \to \Cob(w,w)\]
 is an injective monoid homomorphism.
\end{definition}

It is obvious that the image of $\C$ is contained in the group consisting of invertible elements of $\Cob(w,w)$.
In fact, the following proposition holds.

\begin{proposition}\label{prop:InvertElem}
 Let $w \in \Mon(\bullet,\circ)$ and $M \in \Cob(w,w)$.
 Then $M$ is left-invertible if and only if it is right-invertible.
 Moreover, the group $\Cob(w,w)^\times$ consisting of invertible elements of $\Cob(w,w)$ is equal to $\Im \C$.
\end{proposition}

The definition of the mapping class group is different from \cite{HaMa12}, although the same proof as \cite[Proposition~2.4]{HaMa12} is valid for our monoid homomorphism $\C$.

\begin{remark}
 The mapping class group of $F_w$ is usually defined by
 \[\M_0(F_w) := \{f\in\Homeo(F_w) \mid f|_{\partial F_w}=\Id_{\partial F_w}\} \mathrel{/} \text{isotopy rel $\partial F_w$},\] 
 which is isomorphic to the kernel of the homomorphism $\pi\colon \M(F_w) \twoheadrightarrow \SS_n$.
 If $w \in \Mon(\circ)$, then $\M(F_w)$ and $\M_0(F_w)$ are called the \emph{framed braid group} and the \emph{framed pure braid group} on $n$ strands, respectively.
\end{remark}

We should restrict cobordisms to the good ones in order to define the Kontsevich-LMO invariant (see Section~\ref{subsec:DefKontsevichLMOinv}) and to let it take value in the space of top-substantial Jacobi diagrams (see Section~\ref{subsec:TopSubstantial}).
Let $w \in \Mon(\bullet,\circ)$.
We define the subgroups of $H_1(F_w)$ by
\[A_w:=\ang{\alpha_1,\dots,\alpha_{|w^\bullet|}},\quad B_w:=\ang{\beta_1,\dots,\beta_{|w^\bullet|}},\quad D_w:=\ang{\delta_1,\dots,\delta_{|w^\circ|}}.\]

\begin{definition}
 A cobordism $(M,\sigma,m) \in \Cob(w_+,w_-)$ is \emph{Lagrangian} if the following two conditions are satisfied:
\begin{enumerate}
 \item $H_1(M) = m_{-,\ast}(A_{w_-}) + m_{+,\ast}(H_1(F_{w_+}))$,
 \item $m_{+,\ast}(A_{w_+}) \subset m_{-,\ast}(A_{w_-}) + m_{+,\ast}(D_{w_+})$.
\end{enumerate}
\end{definition}

It follows from a Mayer-Vietoris argument that the composition of two Lagrangian cobordisms is again Lagrangian.
(One can get an alternative proof by Lemma \ref{lem:Lagrangian} and \cite[Claim 3.18]{CHM08} or the fact that $M$ is Lagrangian if and only if $\K(M)$ is Lagrangian in the sense of \cite{CHM08}.)
Thus, we obtain the wide subcategory $\LCob$ whose morphisms are Lagrangian cobordisms.
Actually, we need a non-strict monoidal category $\LCob_q$ defined in Section~\ref{subsec:qCob}.

\begin{lemma}\label{1'}
 Under the condition \textup{(2)}, \textup{(1)} is equivalent to the following condition: \textup{(1')} the map
 \[m_{-,\ast}\oplus m_{+,\ast}\colon A_{w_-}\oplus(B_{w_+}\oplus D_{w_+}) \to H_1(M)\]
 is an isomorphism.
\end{lemma}

\begin{proof}
 It is easy to see that (1') implies (1).
 Let us prove the converse.
 It follows from (1) and (2) that $b_1(M) \le g_{-}+g_{+}+n$, where $b_1(M)$ is the first Betti number of $M$.
 Since $M$ is a compact odd-dimensional manifold, we have $\chi(\partial M)=2\chi(M)$, and thus $b_1(M)=g_{-}+g_{+}+n+b_2(M)$.
 Hence, the above inequality is an equality, that is, $m_{-,\ast}\oplus m_{+,\ast}$ is an isomorphism.
\end{proof}

\begin{remark}\label{rem:HomCube}
 By the proof of the above lemma, we have $b_2(M)=0$.
 Since $H_2(M)$ is free, we conclude $H_2(M)=0$, namely $M$ is a \emph{homology handlebody}.
 In particular, a homology handlebody of genus zero is called a \emph{homology cube}.
\end{remark}

\begin{example}
 The mapping cylinder of $h \in \M(F_w)$ is Lagrangian if and only if $h_\ast(A_w \oplus D_w) = A_w \oplus D_w$.
 Indeed, if $\C(h)$ is Lagrangian, then we have $h_\ast(A_w \oplus D_w) \subset A_w \oplus D_w$.
 Since $h_\ast$ is an isomorphism, so is $h_\ast|_{A_w \oplus D_w}$.
 The converse is obvious.
\end{example}

\subsection{Bottom-top tangles}\label{subsec:BTtangle}
In this subsection, we translate cobordisms to certain tangles in a homology cube.
Let $w \in \Mon(\bullet,\circ)$.
We prepare the reference surface $F_e$ with the pairs of points $\{p_1,q_1\},\dots,\{p_{|w^\bullet|},q_{|w^\bullet|}\}$ and the points $r_1\dots,r_{|w^\circ|}$ corresponding to the letters $\bullet$'s and $\circ$'s in $w$, respectively.
These points are on the $x$-axis according to $w$ as drawn in Figure~\ref{fig:PonitedReferenceSurface}.

\begin{figure}[h]
 \centering
 \includegraphics[width=0.5\columnwidth]{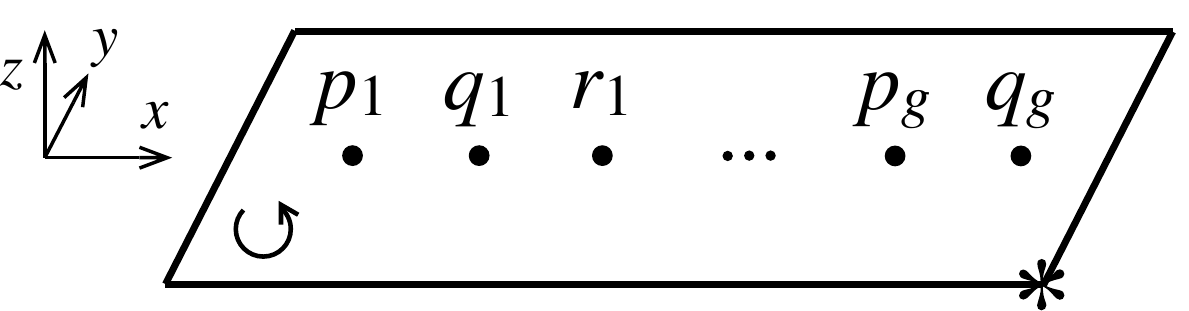}
 \caption{The reference surface with the points corresponding to $w$ such that $|w^\bullet|=g$.}
 \label{fig:PonitedReferenceSurface}
\end{figure}

\begin{definition}
 Let $w_+,w_- \in \Mon(\bullet,\circ)$ such that $|w_+^\circ|=|w_-^\circ|$ and let $\sigma \in \SS_{|w_+^\circ|}$.
 A \emph{bottom-top tangle} of \emph{type} $(w_+,w_-,\sigma)$ is an equivalence class of pairs $(B,\gamma)$, where
\begin{itemize}
 \item $B=(B,\emptyset,b)$ is a cobordism from $F_e$ to $F_e$,
 \item $\gamma$ is a framed oriented tangle in $B$, which consists of the top components $\gamma_1^+,\dots,\gamma_{|w_+^\bullet|}^+$, the bottom components $\gamma_1^-,\dots,\gamma_{|w_-^\bullet|}^-$ and the vertical components $\gamma_1^0,\dots,\gamma_{|w_+^\circ|}^0$, where
\begin{itemize}
 \item each $\gamma^{+}_i$ runs from $p_i\times1$ to $q_i\times1$,
 \item each $\gamma^{-}_i$ runs from $q_i\times(-1)$ to $p_i\times(-1)$,
 \item each $\gamma^0_i$ runs from $r_i\times1$ to $r_{\sigma(i)}\times(-1)$.
\end{itemize}
\end{itemize}
 Here, $(B,\gamma)$ and $(B',\gamma')$ are equivalent if they are of the same type and $B$ is equivalent to $B'$ as cobordisms by a homeomorphism that sends $\gamma$ to $\gamma'$ and respects their framings.
\end{definition}

We define the category $\btT$ as follows:
An object of $\btT$ is an element of $\Mon(\bullet,\circ)$.
A morphism from $w_+$ to $w_-$ is a bottom-top tangle of type $(w_+,w_-,\sigma)$ for some $\sigma \in \SS_{|w_+^\circ|}$.
The composition of $(B,\gamma)\in\btT(v,u)$ and $(C,\omega)\in\btT(w,v)$ is the bottom-top tangle obtained by inserting $([-1,1]^3, T_v)$ between them and performing the surgery along $(2|v^\bullet|)$-components link $\gamma^{+}\cup T_{w^\bullet}\cup\omega^{-}$ as illustrated in Figure~\ref{fig:TvIdw}.
The identity of $w$ is the bottom-top tangle as depicted in Figure~\ref{fig:TvIdw}. 
Indeed, one can check that
\[\Id_{w_-}\circ(B,\gamma) = (B,\gamma) = (B,\gamma)\circ\Id_{w_+},\]
using the Kirby move II and the useful fact that if a disk $D$ in a 3-manifold $M$ intersects a framed knot $K$ once, then the result of surgery along $K$ and $\partial D$ with 0-framing is again $M$ (see Figure~\ref{fig:composition}).

\begin{figure}[h]
 \centering
 \includegraphics[width=0.8\columnwidth]{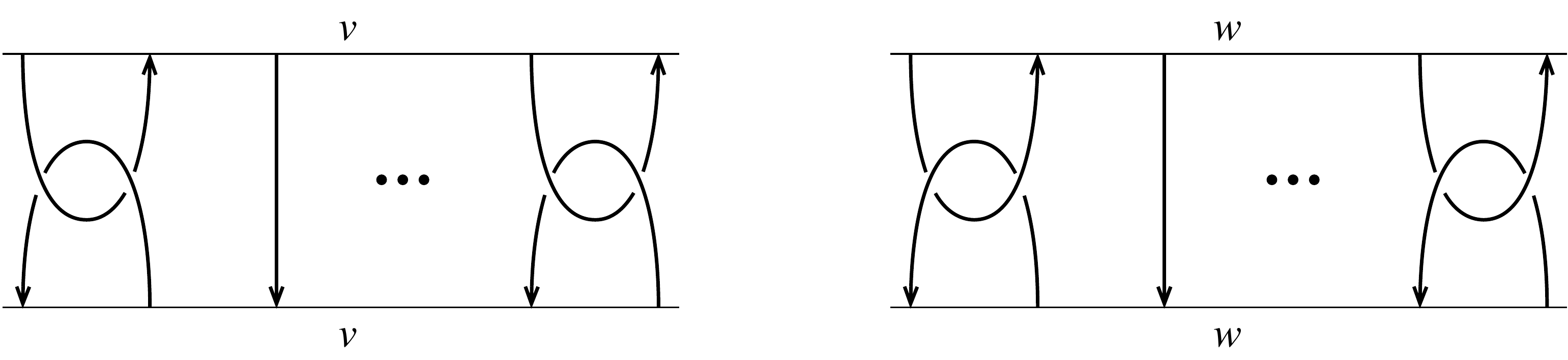}
 \caption{The bottom-top tangle $([-1,1]^3, T_v)$ and the identity $\Id_w$}
 \label{fig:TvIdw}
\end{figure}

\begin{figure}[h]
 \centering
 \includegraphics[width=0.8\columnwidth]{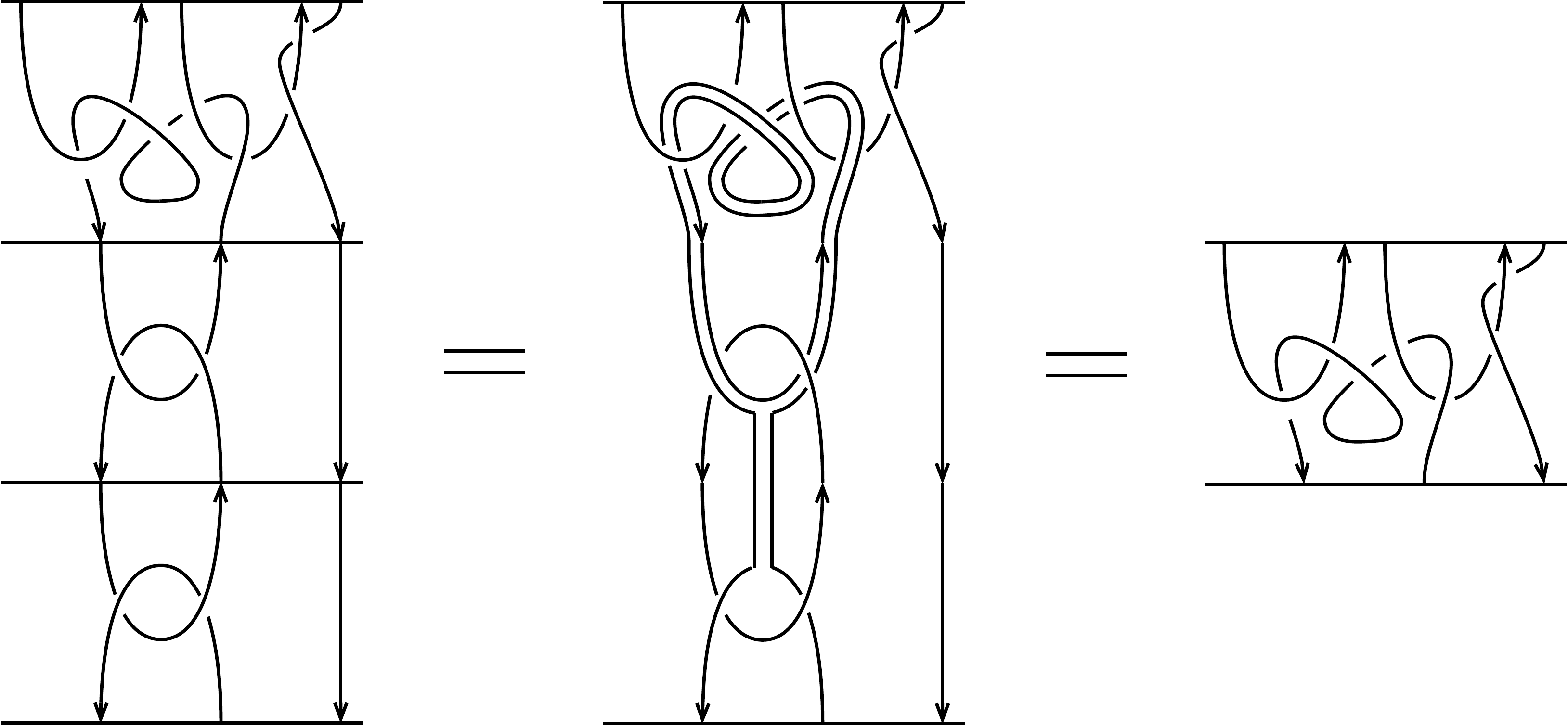}
 \caption{An example of the equality $\Id_{\bullet\circ}\circ(B,\gamma)=(B,\gamma)$}
 \label{fig:composition}
\end{figure}

Moreover, $\btT$ is a strict monoidal category, where the tensor product of two bottom-top tangles is their horizontal juxtaposition and the unit is the empty word $e$.

The next lemma is used in the following theorem and so on.

\begin{lemma}[see {\cite[Definition 3.15]{CHM08}}]
 Let $(B,\gamma)$ be a bottom-top tangle.
 Then there exist a framed link $L$ and a tangle $\gamma'$ in $[-1,1]^3$ such that $(B,\gamma)=([-1,1]^3_L,\gamma')$, where $[-1,1]^3_L$ is the \textup{3}-manifold obtained by surgery along $L$.
 The pair $(L,\gamma')$ is called a \emph{surgery presentation} of $(B,\gamma)$.
\end{lemma}

\begin{theorem}[{\cite[Theorem 2.10]{CHM08}}]
 The operation as drawn in Figure~\textup{\ref{fig:dig}} gives an isomorphism $\D\colon \btT \to \Cob$ between two strict monoidal categories.
\end{theorem}

\begin{figure}[h]
 \centering
 \includegraphics[width=0.9\columnwidth]{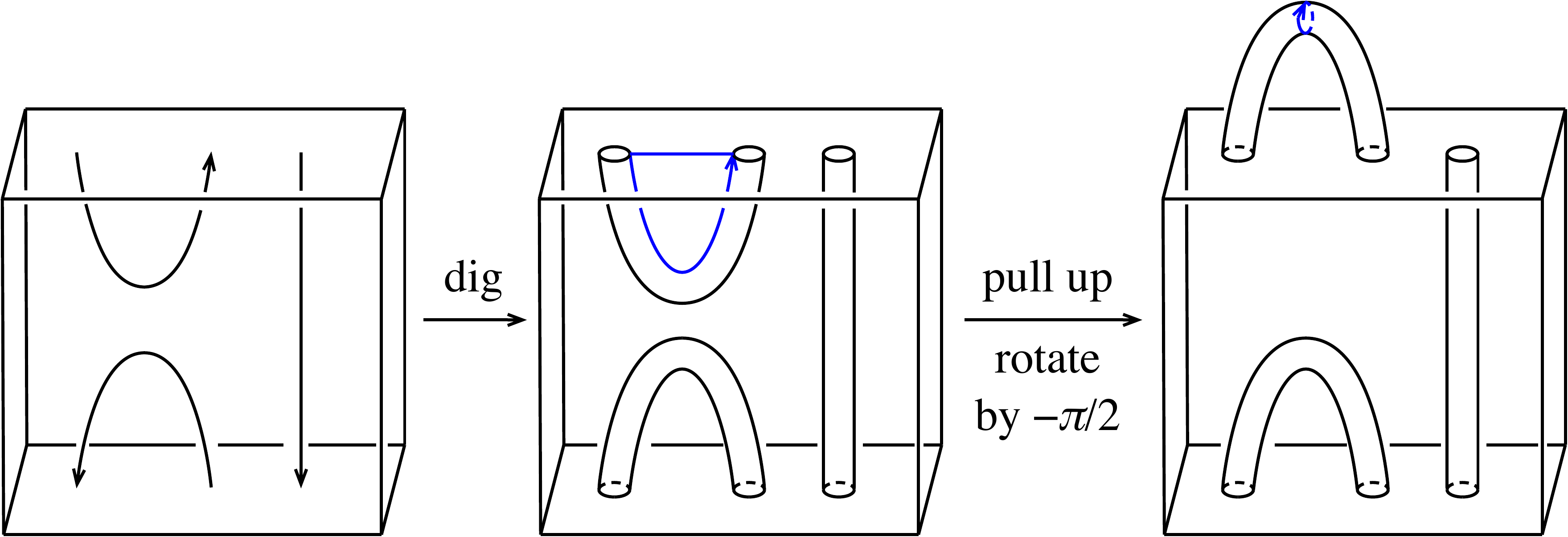}
 \caption{Correspondence between a bottom-top tangle and a cobordism}
 \label{fig:dig}
\end{figure}

\begin{proof}
 Precisely, $\D$ is defined by digging a bottom-top tangle $((B,b),\gamma)$ and parametrizing the boundary of the resulting manifold in accordance with the homeomorphism $b\colon R^e_{e,\Id} \to \partial B$ and the framings of $\gamma$.
 It is easy to see that $\D$ satisfies the definition of a monoidal functor except
 \[\D((B,\gamma)\circ(C,\omega)) = \D(B,\gamma)\circ\D(C,\omega).\]
 To prove this equality, we take a surgery presentation $(L,\gamma')$ of $(B,\gamma)$.
 Moreover, using an additional surgery link $L'$, $(B,\gamma)$ is expressed as the bottom-top tangle obtained from a bottom-top tangle $([-1,1]^3,\widetilde{\gamma})$ such that $\widetilde{\gamma}^+$ is trivial by surgery along $L \cup L'$.
 Each $\widetilde{\gamma}^+_i$ bounds a half-disk $D_i$ that could intersect $\widetilde{\gamma}^- \cup \widetilde{\gamma}^0$ or $L \cup L'$.
 We now gather the intersections on $D_i$, and we regard the components intersecting $D_i$ as a bundle in a neighborhood of $D_i$.
 Since $D_i$ intersects the bundle once, we can apply the argument in the proof of $\Id_{w_-} \circ (B,\gamma) = (B,\gamma)$, and thus the desired equality is obtained.
 
 Next, let $(M,\sigma,m)$ be a cobordism.
 The inverse of $\D$ can be constructed by attaching 2-handles $[0,1]^2\times[0,1]$ to the boundary of $M$.
 (We recall that the way of attaching a 3-dimensional 2-handle is uniquely determined by the way one glues the attaching sphere of the 2-handle with $\partial M$.)
 We obtain the 3-manifold $B$ by attaching 2-handles along $m_{+}(\beta_i)$'s, $m_{-}(\alpha_i)$'s and $m_{+}(\delta_i)$'s and define $\gamma'$ by the co-cores of the 2-handles.
 By construction, this correspondence gives the inverse of $\D$.
\end{proof}

\begin{example}
 Let $c$ be a simple closed curve on the surface $\Int F_w$ that is the interior of $F_w$.
 We denote by $\tau_c$ the Dehn twist along $c$.
 By the definition of a Dehn twist and surgery, the cobordisms $\C(\tau_{\alpha_i})$, $\C(\tau_{\beta_i})$ and $\C(\tau_{\delta_i})$ are sent by $\D^{-1}$ to
\begin{center}
 \raisebox{-2.2em}{
\begin{overpic}[height=7em]{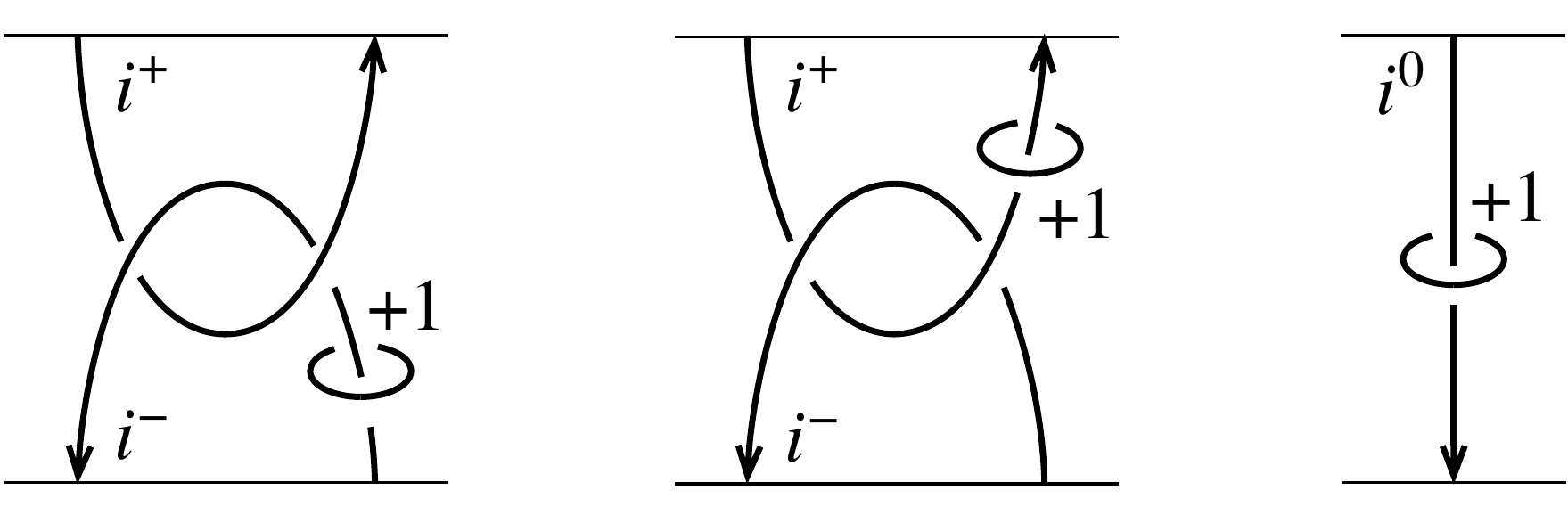}
 \put (35,12){,}
 \put (75,12){and}
\end{overpic}}
\end{center}
 respectively.
 $\C(\tau_{\alpha_i})$, $\C(\tau_{\delta_i})$ are Lagrangian while $\C(\tau_{\beta_i})$ is not.
\end{example}

There is a natural question when a bottom-top tangle is sent to a Lagrangian cobordism by $\D$.
We introduce an extension of the linking matrix defined in \cite[Definition 2.11]{CHM08} to answer this question.

\begin{definition}
 Let $((B,b),\gamma)$ be a bottom-top tangle of type $(w_+,w_-,\sigma)$ such that $B$ is a homology cube, and let $g_\pm:=|w_\pm^\bullet|$, $n:=|w_+^\circ|$.
 Prepare the manifold $S^3\setminus\Int[-1,1]^3$ containing four kinds of arcs: $g_+$ copies of $\cap$-like arcs, $g_-$ $\cup$-like arcs, $n$ parallel arcs and a braid $\beta \in B_n$ such that $\pi(\beta)=\sigma^{-1}$, as illustrated in Figure~\ref{fig:Lk}. 
 Then the \emph{linking matrix} of $\gamma$ in $B$ is defined by
 \[\Lk_B(\gamma) := \Lk_{\widehat{B}}(\widehat{\gamma})-O_{g_{+}+g_{-}}\oplus\sigma^{-1}\cdot\Cr(\beta) \in \tfrac{1}{2}\Sym_{\pi_0 \gamma}(\Z),\]
 where $\Lk_{\widehat{B}}(\widehat{\gamma})$ is the usual linking matrix of the link
 \[\widehat{\gamma} := \gamma \cup (\text{the arcs and braid in $S^3\setminus\Int[-1,1]^3$})\]
 in the homology sphere $\widehat{B} := B \cup_b (S^3 \setminus\Int[-1,1]^3)$.
 Here, $\Cr(\beta)$ is the matrix whose $(i,j)$-entry is half the number of positive crossings of the $i$th strand and the $j$th strand minus the number of negative ones, and $\sigma^{-1}\cdot\Cr(\beta)$ is the matrix obtained from $\Cr(\beta)$ by permuting its rows and columns by $\sigma^{-1}$.
\end{definition}

\begin{figure}[h]
 \centering
 \includegraphics[width=0.8\columnwidth]{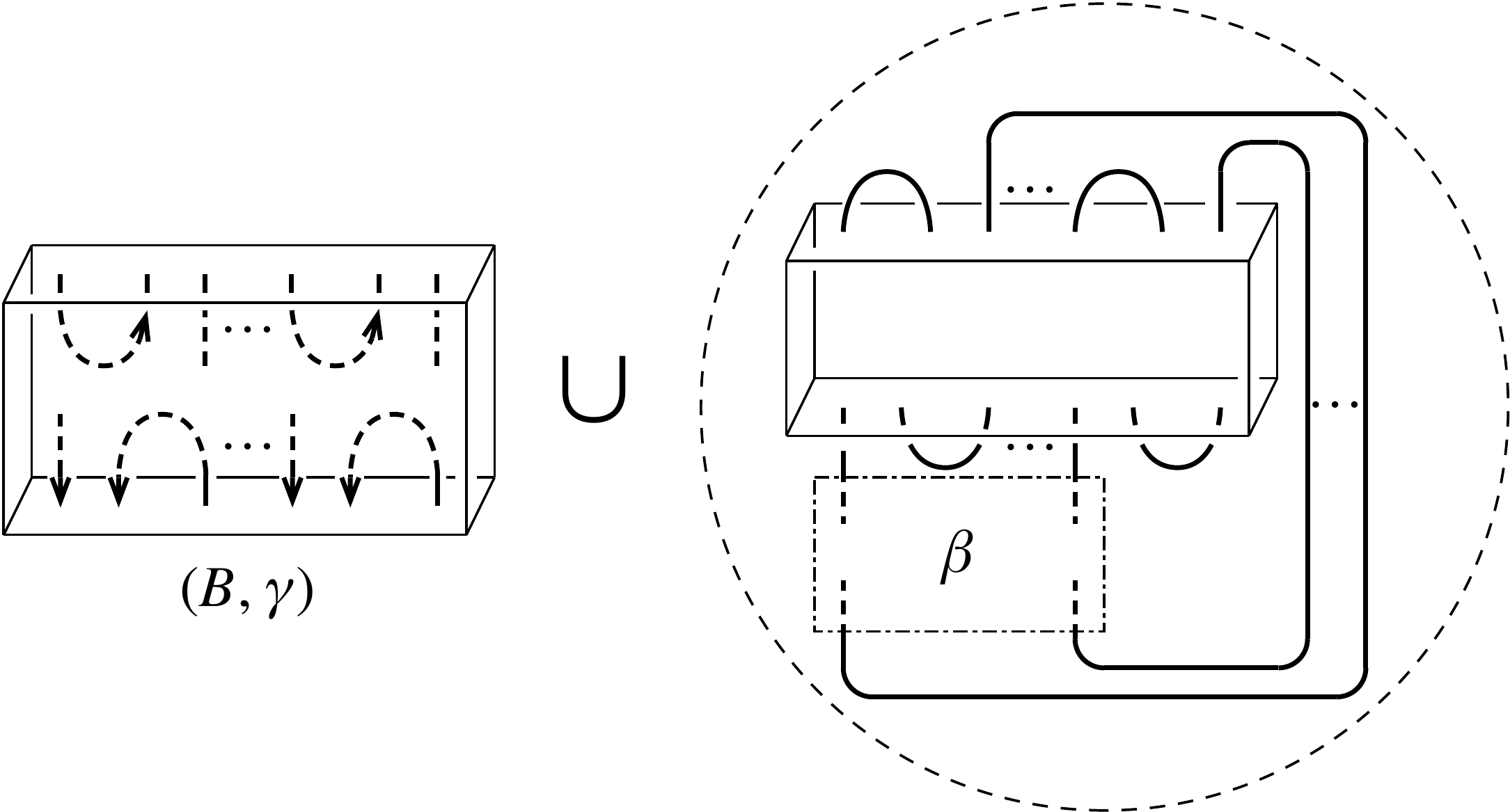}
 \caption{The union of $(B,\gamma)$ and additions, that is, $(\widehat{B},\widehat{\gamma})$}
 \label{fig:Lk}
\end{figure}

Moreover, we set $\Lk(M):=\Lk_B(\gamma)$ for a cobordism $(M,\sigma,m)=\D(B,\gamma)$, where $B$ is a homology cube.

\begin{remark}
 The above map $\Cr$ represents a non-trivial class of the first cohomology group $H^1(B_n^\op;\tfrac{1}{2}\Sym_n(\Z))$, where $B_n^\op$ is the group obtained from $B_n$ by reversing the order of multiplication.
 Indeed, $B_n^\op$ acts on $\tfrac{1}{2}\Sym_n(\Z)$, via the anti-homomorphism $\pi\colon B_n^\op \to \SS_n$, such that $\beta \cdot X:=\pi(\beta)^{-1} \cdot X$.
 Then $\Cr \colon B_n^\op \to \tfrac{1}{2}\Sym_n(\Z)$ is a crossed homomorphism, that is, it satisfies
 \[\Cr(\beta\beta') = \Cr(\beta)+\beta\cdot\Cr(\beta').\]
 Furthermore, $\left|\Im\left(d^0(X)\colon B_n^\op\to\tfrac{1}{2}\Sym_n(\Z)\right)\right| \le n!$ for any $X \in \tfrac{1}{2}\Sym_n(\Z)$.
 While, $|\Im(\Cr)| = \infty$.
\end{remark}

The next lemma is proved in much the same way as \cite[Lemma 2.12]{CHM08}.

\begin{lemma}\label{lem:Lagrangian}
 Let $(B,\gamma)$ be a bottom-top tangle.
 Then $\D(B,\gamma)$ is Lagrangian if and only if $B$ is a homology cube and $\Lk_B(\gamma^+)=O$ holds.
\end{lemma}

\section{Jacobi diagrams and some operations}\label{sec:JacobiDiagOp}

\subsection{Jacobi diagrams}\label{subsec:JacobiDiag}
Let $D$ be a uni-trivalent graph, which is depicted by dashed lines.
A univalent (resp.\ trivalent) vertex of $D$ is called an \emph{internal} (resp.\ \emph{external}) \emph{vertex}.
We denote by $\ideg D$ (resp.\ $\edeg D$) the number of internal (resp.\ external) vertices of $D$ and we set $\deg D := (\ideg D +\edeg D)/2$.

Let $X$ be a compact oriented 1-manifold and let $C$ be a finite set.
A \emph{Jacobi diagram} based on $(X,C)$ is a vertex-oriented uni-trivalent graph whose vertices are either embedded into $X$ or are colored by elements of $C$.
When identifications of $\pi_0 X$ with $\pi_0 X'$ and $C$ with $C'$ are given, two Jacobi diagrams $D$ based on $(X,C)$ and $D'$ based on $(X',C')$ are identified if there exists a homeomorphism $f\colon X\cup D \to X'\cup D'$ respecting the identification, orientations of $X$ and $X'$ and the vertex-orientations of $D$ and $D'$.

A Jacobi diagram is drawn in the plane whose vertex-orientation agrees with the counter-clockwise order.
We denote by $\A(X,C)$ the quotient $\Q$-vector space spanned by Jacobi diagrams based on $(X,C)$ subject to the AS, IHX and STU relations depicted in Figure~\ref{fig:relations}.

\begin{figure}[h]
 \centering
\begin{minipage}{0.36\columnwidth}
 \centering
 \includegraphics[width=0.8\columnwidth]{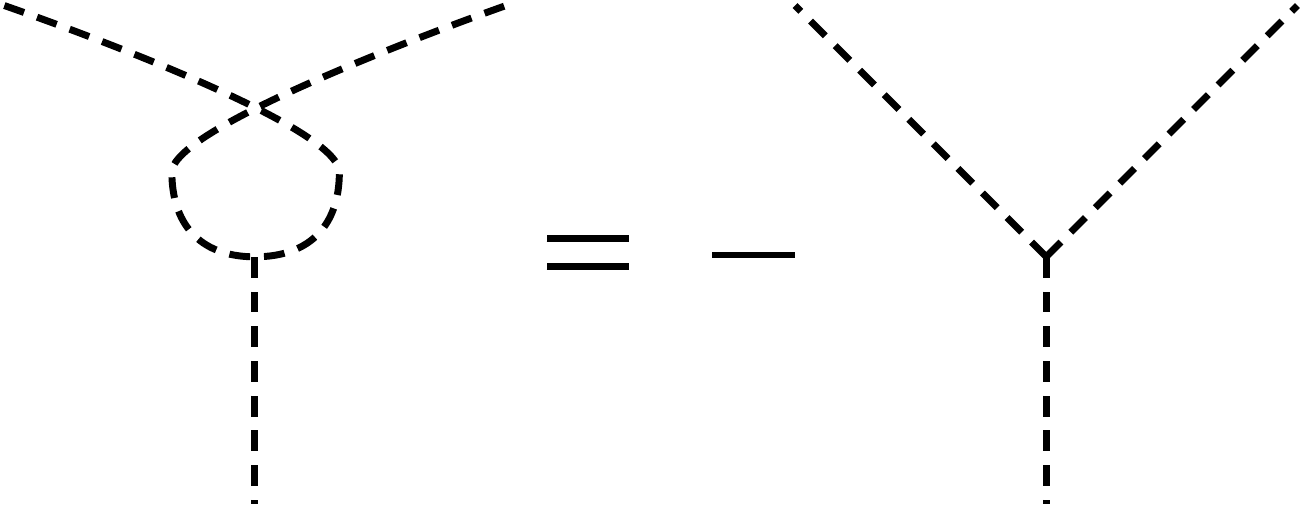}
\end{minipage}
\hspace{12pt}
\begin{minipage}{0.52\columnwidth}
 \centering
 \includegraphics[width=0.8\columnwidth]{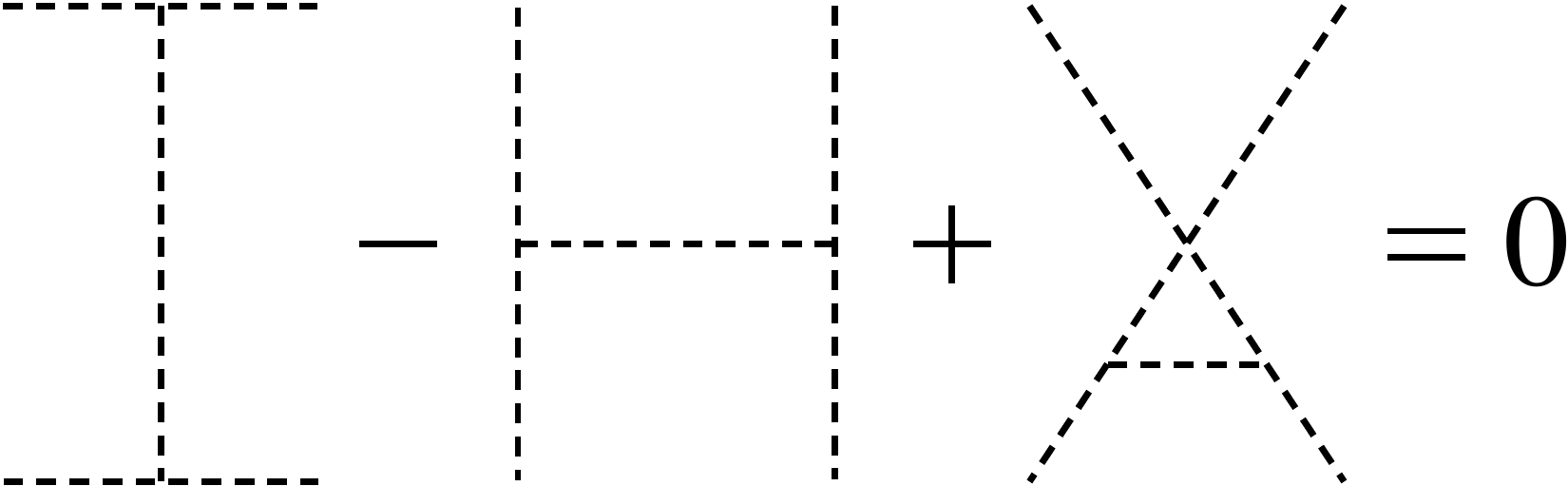}
\end{minipage}
\begin{minipage}{0.52\columnwidth} \vspace{0.5em}%
 \centering
 \includegraphics[width=0.8\columnwidth]{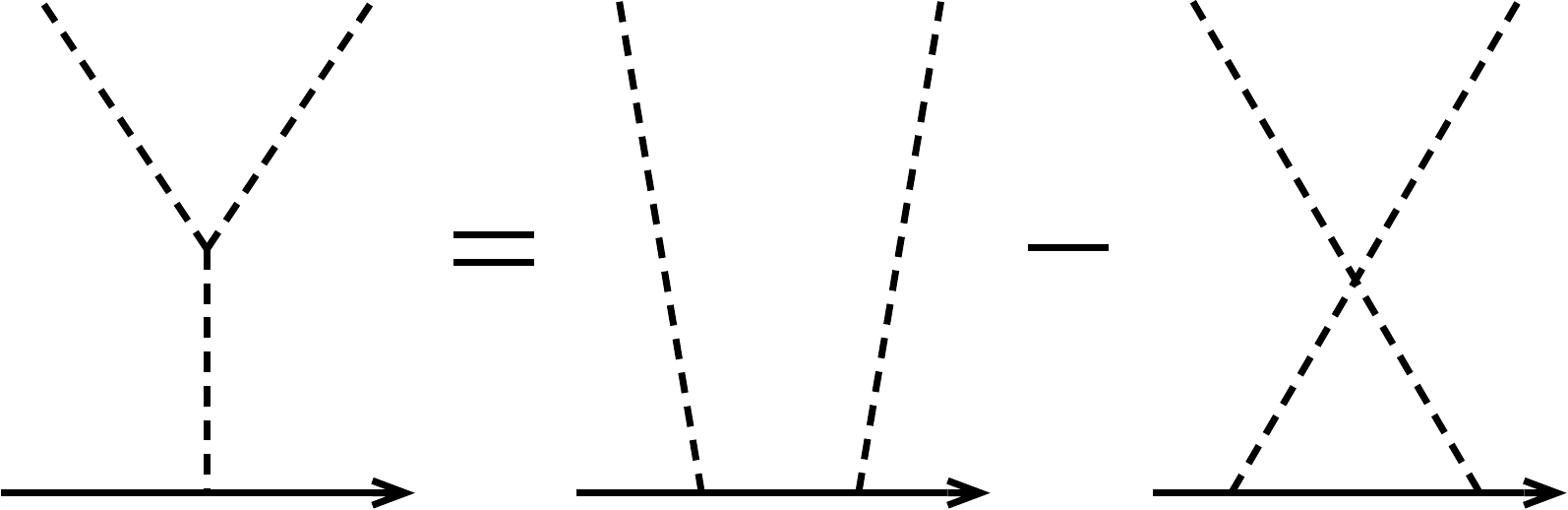}
\end{minipage}
 \caption{The AS, IHX and STU relations. These figures represent local relations, namely, the lest of the diagrams are identical.}
 \label{fig:relations}
\end{figure}

Since these relations preserve the degree of a Jacobi diagram, $\A(X,C)$ is a graded $\Q$-vector space.
Thus, we can take the degree completion of $\A(X,C)$ and denote it by $\A(X,C)$ again.
Note that the internal degree is not well-defined for Jacobi diagrams of $\A(X,C)$.
A Jacobi diagram $D \in \A(X,C)$ is said to have \emph{\textrm{i}-filter at least} $n$ if $D$ is written as a series of Jacobi diagrams whose internal degrees are at least $n$, and then we use the notation $\ifil D \ge n$.

If $X=\emptyset$ (resp.\ $C=\emptyset$), then we simply write $\A(C)$ (resp.\ $\A(X)$) for $\A(X,C)$ when no confusion can arise.
It is well known that $\A(C)$ is naturally equipped with a connected, commutative, cocommutative, graded Hopf algebra structure.
There are two important subsets of $\A(C)$.
We first set
\[\G(\A(C)) := \{x\in\A(C) \mid \Delta(x)=x\otimes x,\ \varepsilon(x)=\varnothing\},\]
where $\varnothing$ denotes the empty diagram.
$\G(\A(C))$ is a group whose multiplication is disjoint union $\sqcup$, and its element is said to be \emph{group-like}.
Let $\A^c(C)$ denote the subspace of $\A(C)$ spanned by non-empty connected Jacobi diagrams, which is identified with the quotient by the subspace generated by $\varnothing$ and disconnected Jacobi diagrams.
The image of $x \in \A(C)$ by the quotient map is denoted by $x^c$.
Moreover, it is well known that the map $\exp_\sqcup\colon \A^c(C) \to \G(\A(C))$ defined by $\exp_\sqcup(x) := \sum_{n\ge0}x^{\sqcup n}/n!$ is a group isomorphism.

We further use two subalgebras of $\A(C)$ which play an important role in this paper.
The subalgebra generated by Jacobi diagrams only with struts (resp.\ without struts) is denoted by $\A^s(C)$ (resp.\ $\A^Y(C)$) that is identified with the quotient by the ideal generated by diagrams with $\ideg>0$ (resp.\ the ideal generated by struts).
The quotient maps are called the \emph{$s$-reduction} (resp.\ \emph{$Y$-reduction}), the images of $x\in\A(C)$ is denoted by $x^s$ (resp.\ $x^Y$).

\begin{lemma}[{\cite[Lemma~3.5]{CHM08}}] \label{lem:group-like}
 The group $\G(\A(C))$ coincides with the set
 \[\{x\in\A(C) \mid \text{$x^s, x^Y \in \G(\A(C))$ and $x=x^s \sqcup x^Y$}\}.\]
\end{lemma}

\begin{remark}
 The subspace $\A^Y(C)$ can also be regarded as the result of the internal degree completion, since $(\deg D)/2 \le \ideg D \le 2\deg D$ holds for all $D \in \A^Y(C)$.
 We denote by $\A_i(C)$ the subspace of all homogeneous elements of internal degree $i$.
\end{remark}

Next, we review an analogue of the Poincar\'e-Birkhoff-Witt isomorphism.
Let $\downarrow^S$ (resp.\ $\circlearrowleft^S$) denote $|S|$ intervals (resp.\ circles) indexed by elements of a finite set $S$.

\begin{definition}
 The graded linear map $\chi_S \colon \A(X, C\cup S) \to \A(X\downarrow^S, C)$ is defined for a Jacobi diagram $D \in \A(X, C\cup S)$ to be the average of all possible ways of attaching the $s$-colored vertices in $D$ to the $s$-indexed interval, for all $s \in S$, as depicted in Figure~\ref{fig:PBW}.
\end{definition}

\begin{figure}[h]
 \centering
 \includegraphics[width=0.6\columnwidth]{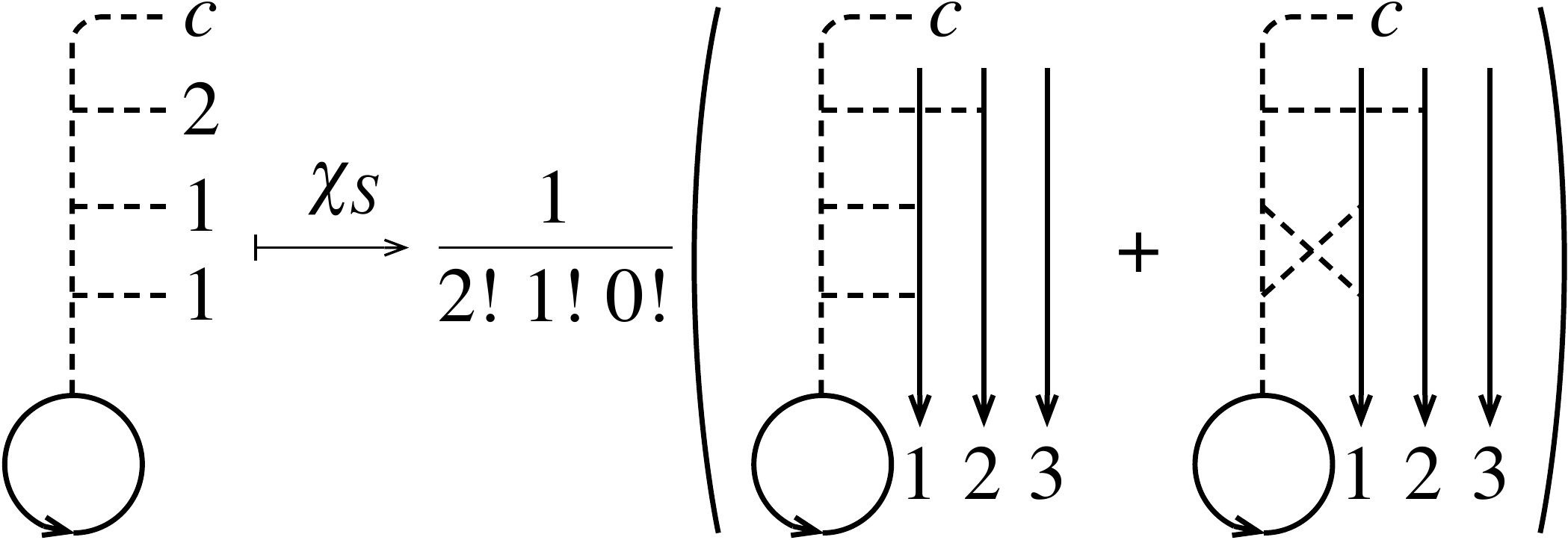}
 \caption{An example of $\chi_S$ with $X=\ \circlearrowleft$, $S=\{1,2,3\}$ and $c \in C$}
 \label{fig:PBW}
\end{figure}

It is well known that $\chi_S$ is a graded coalgebra isomorphism.
Similarly, an isomorphism $\chi_S \colon \A(X, C\cup S)/{\sim} \to \A(X\!\circlearrowleft^S, C)$ is defined, where $\sim$ denotes the $S$-link relation defined in \cite[Section~5.2]{BGRT02b} (see also \cite[Remark~3.7]{CHM08}).
We need an extension of $\chi_S$ to construct the composition law of the category $\tsA$ defined later.

\begin{definition}
 Let $S'$ be a copy of the above set $S$.
 The graded linear map
 \[\chi_{S,S'} \colon \A(X,C\cup S\cup S')\to\A(X\downarrow^S,C)\]
 is defined for a Jacobi diagram $D \in \A(X, C\cup S\cup S')$ to be stacking of the $s'$-indexed interval on the $s$-indexed interval in $\chi_S \circ \chi_{S'}(D)$, for all $s \in S$, as illustrated in Figure~\ref{fig:PBW2}.
\end{definition}

\begin{figure}[h]
 \centering
 \includegraphics[width=0.8\columnwidth]{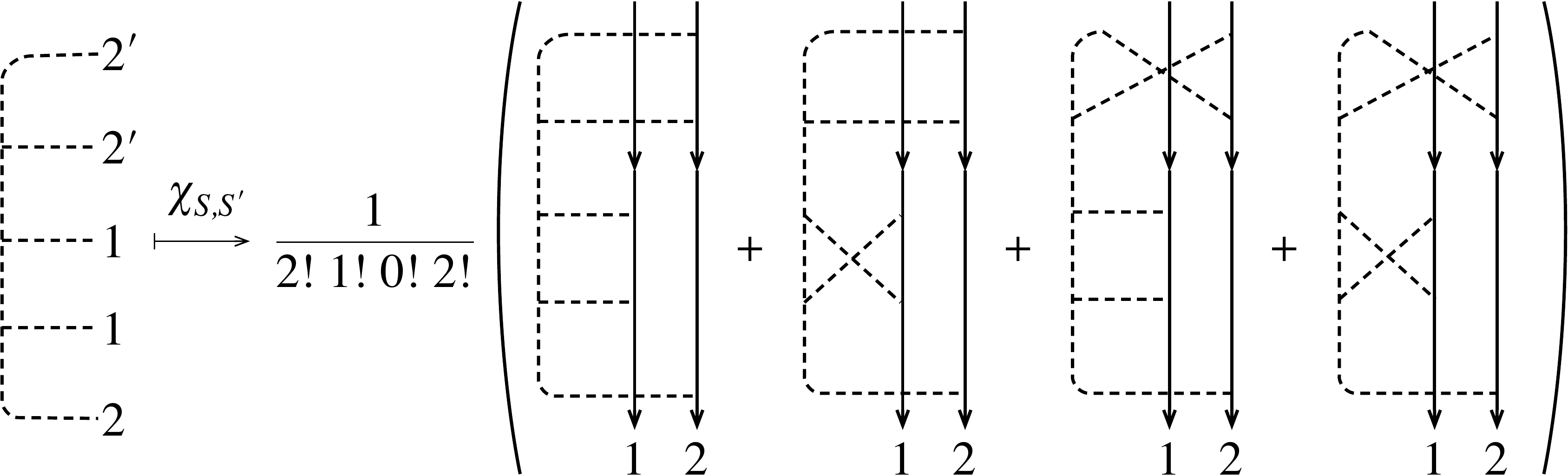}
 \caption{An example of $\chi_{S,S'}$ with $X=C=\emptyset$, $S=\{1,2\}$ and $S'=\{1',2'\}$}
 \label{fig:PBW2}
\end{figure}

Similar ideas appear, for example, in \cite[Claim~5.6]{CHM08}, \cite[Proposition~5.4]{BGRT02b} and \cite[(3.3)]{HaMa09}.
Before finishing this subsection, we recall the following lemma that is used several times.

\begin{lemma}[{\cite[Lemma 8.19]{CHM08}}]\label{lem:filter}
 Let $D$ be a Jacobi diagram based on $(X\!\downarrow^S,C)$.
 Then we have
 \[\chi_S^{-1}(D) = \widetilde{D}+(\ifil>\ideg D),\]
 where $\widetilde{D}$ is obtained from $D$ by deleting the $S$-indexed intervals.
\end{lemma}

\subsection{The formal Gaussian integral}\label{subsec:GaussianInt}
We first review $S$-substantial Jacobi diagrams according to \cite{CHM08}.

\begin{definition}
 A Jacobi diagram $D \in \A(X,C\cup S)$ is \emph{$S$-substantial} if $D$ contains no strut both of whose vertices are colored by the elements of $S$, and an element of $\A(X,C\cup S)$ which can be written as a series of such diagrams is also called $S$-substantial.
 
 Let $D$ and $E$ be Jacobi diagrams based on $(X,C\cup S)$ such that at least one of them is $S$-substantial.
 If $D$ and $E$ have the same number of $s$-colored vertices, for all $s \in S$, then we define $\ang{D,E}_S \in \A(X,C)$ by
\[\ang{D,E}_S:= \left(\parbox{0.6\textwidth}{sum of all ways of connecting the $s$-colored vertices in $D$ with those of $E$, for all $s \in S$}\right).\]
 Otherwise, we set $\ang{D,E}_S :=0$.
 If $D$ and $E$ have no $S$-colored vertex, we naturally interpret $\ang{D,E}_S$ as their disjoint union.
\end{definition}

From now on, we identify a symmetric matrix $(a_{ij})_{i,j\in S}$ with a linear combination of struts $\sum\limits_{i,j\in S} a_{ij} \strutgraph{j}{i}$, and let $[x]:=\exp_\sqcup(x)$ for any $x \in \A(C)$.

\begin{definition}\label{def:GaussianInt}
 An element $G \in \A(X,C\cup S)$ is \emph{Gaussian} in the \emph{variable $S$} if there are a symmetric matrix $L \in \Sym_S(\Q)$ and an $S$-substantial element $P \in \A(X,C\cup S)$ such that $G = [L/2]\sqcup P$.
 Here, if $\det L\ne 0$, then $G$ is said to be \emph{non-degenerate}, and we set
 \[\int_S G := \ang{[-L^{-1}/2],P}_S\ \in \A(X,C)\]
 that is called the \emph{formal Gaussian integral} of $G$ along $S$.
\end{definition}

\begin{remark}
 If $L$ and $P$ exist, they are unique.
 $L$ is called the \emph{covariance matrix} of $G$ in \cite{BGRT02a}.
\end{remark}

\section{The Kontsevich-LMO invariant}\label{sec:KontsevichLMOinv}
In this section, we review the domains and codomains of the Kontsevich invariant and Kontsevich-LMO invariant in accordance with \cite{CHM08}.

\subsection{Domain of the Kontsevich-LMO invariant}\label{subsec:Domain}
In order to define the Kontsevich invariant and the Kontsevich-LMO invariant as functors, we use the categories $\T_q$, $\T_q\Cub$ and $\A$.

We define the non-strict monoidal category $\T_q\Cub$ as follows:
The set of objects is the magma $\Mag(+,-)$.
A morphism from $w_+$ to $w_-$ is a $q$-tangle in a homology cube of type $(w_+,w_-)$, that is, a properly embedded compact 1-manifold whose boundary corresponds to letters in $w_\pm$.
The composition of $(B,\gamma)\in\T_q\Cub(v,u)$ and $(C,\omega)\in\T_q\Cub(w,v)$ is stacking $C$ on the top of $B$.
The identity of $w$ is the $q$-tangle $([-1,1]^3,\downarrow^w)$, where $\downarrow^w$ is the union of $\downarrow$'s and $\uparrow$'s and each $\downarrow$ (resp.\ $\uparrow$) corresponds to a letter $+$ (resp.\ $-$) in $w$.
The tensor product of $q$-tangles is their juxtaposition.
The unit is the empty word $e$.

Moreover, we denote by $\T_q$ the wide subcategory whose morphisms are $q$-tangles in $[-1,1]^3$.

\subsection{Codomain of the Kontsevich-LMO invariant}\label{subsec:Codomain}
Next, we define the strict monoidal category $\A$ as follows:
The set of objects is $\Mon(+,-)$.
A morphism from $w$ to $v$ is an element of $\coprod_X \A(X,\emptyset)$, the union is taken over all oriented compact 1-manifolds $X$ such that each start (resp.\ end) point corresponds to a letter $+$ in $w$ or $-$ in $v$ (resp.\ $-$ in $v$ or $+$ in $w$).
Roughly speaking, we assign $+$ to the boundary of an interval which faces downward.
The composition of two elements $D\in\A(v,u)$ and $E\in\A(w,v)$ is defined to be stacking $E$ on $D$ according to the letters in $v$.
We often denote by $D\natural E$ the composite $D\circ_{\A} E$.
The identity of $w$ is the intervals $\downarrow^w$.
The tensor product of two morphisms is defined to be their disjoint union.
The unit is the empty word $e$.

Here, we prepare three operations in accordance with \cite{CHM08}.
Let $\Delta$ and $S$ denote the doubling map and orientation reversal map defined as usual:
\[\Delta\colon \A(X \downarrow,C) \to \A(X \downarrow \downarrow,C),\quad S\colon \A(X \downarrow,C) \to \A(X \uparrow,C).\]
Let $w,w_1,\dots,w_n \in \Mon(+,-)$ such that $|w|=n$.
We define the linear map
\[\Delta^{w}_{w_1,\dots,w_n}\colon \A(X \downarrow^w,C) \to \A(X \downarrow^{w_1 \dotsm w_n},C)\]
by applying $\Delta$ as much as $(|w_i|-1)$ times to the interval indexed by the $i$th letter in $w$ for $i=1,\dots,n$, and applying $S$ to the intervals where the two letters differ.

\subsection{Definition of the Kontsevich-LMO invariant}\label{subsec:DefKontsevichLMOinv}
We first fix an (rational) associator $\Phi \in \A(\downarrow \downarrow \downarrow)$ to define the Kontsevich invariant of a $q$-tangle in $[-1,1]^3$.
The \emph{Kontsevich invariant} $\ZK$ is defined by $\ZK(\gamma):=Z(\gamma)$ for a $q$-tangle $\gamma$, where $Z$ is the Kontsevich integral in \cite{CHM08}.
By definition, $\ZK$ defines a tensor-preserving functor $\T_q \to \A$, that is, it satisfies the following conditions:
\begin{enumerate}
 \item $\ZK(\Id_w) = \ \downarrow^w$,
 \item $\ZK(\gamma\circ\omega) = \ZK(\gamma)\circ\ZK(\omega)$,
 \item $\ZK(e) = e$,
 \item $\ZK(\gamma\otimes\omega) = \ZK(\gamma)\otimes\ZK(\omega)$.
\end{enumerate}

Next, we review the Kontsevich-LMO invariant of a $q$-tangle in a homology cube $B$.
We define $U_\pm$ by
\[U_\pm := \int_{\{\ast\}} \chi_{\{\ast\}}^{-1}\ZK(\circlearrowleft_{\pm1}) \sharp \nu \in \A(\emptyset,\emptyset),\]
where $\ast$ is the label of the $(\pm1)$-framed unknot, and $\nu$ denotes the Kontsevich invariant of the 0-framed unknot, and $\sharp$ means the multiplication induced by the connected sum of two circles.
The \emph{Kontsevich-LMO invariant} $\ZKLMO$ is defined by
 \[\ZKLMO(B,\gamma) := U_{+}^{-\sigma_{+}(L)} \sqcup U_{-}^{-\sigma_{-}(L)} \sqcup \int_{\pi_0 L} \chi_{\pi_0 L}^{-1} \ZK(L^\nu \cup \gamma) \in \A(\gamma,\emptyset)\]
for a $q$-tangle $(B,\gamma)$, where $(L,\gamma)$ is a surgery presentation of $(B,\gamma)$, $\sigma_\pm(L)$ is the number of positive/negative eigenvalues of the linking matrix of $L$, and let $\ZK(L^\nu \cup \gamma)$ denote $\ZK(L \cup \gamma)\sharp\nu^{\otimes L}$.
One can check that $\ZKLMO$ defines a tensor-preserving functor $\T_q\Cub \to \A$.

Here, a bottom-top $q$-tangle in a homology cube $B$ is regarded as a $q$-tangle in $B$ via the embedding $\btT_q(w,v) \to \T_q\Cub(w',v')$, where words $v',w' \in \Mag(\bullet,\circ)$ are obtained from words $v,w \in \Mag(+,-)$ by the rule
\[\bullet\mapsto(+-),\quad \circ\mapsto+.\]
Then we denote by $Z(B,\gamma)$ the Kontsevich-LMO invariant of a bottom-top $q$-tangle $(B,\gamma)$.

The next lemma is used in Section~\ref{subsec:ConstExtLMOfunc} and proved similarly to \cite[Lemma 3.17]{CHM08}.
The only difference is that the definition of $\Lk_B(\gamma)$ is extended.

\begin{lemma}\label{lem:bottom-top tangle}
 Let $(B,\gamma)$ be a bottom-top $q$-tangle.
 Then $\chi_{\pi_0 \gamma}^{-1}Z(B,\gamma) \in \A(\pi_0 \gamma)$ is group-like and its $s$-reduction is equal to $[\Lk_B(\gamma)/2]$.
\end{lemma}

\section{Extension of the LMO functor}\label{sec:ExtLMOfunc}
In this section, we construct categories $\LCob_q$, $\tsA$ and an extension of the LMO functor $\Zt\colon \LCob_q \to \tsA$.

\subsection{$q$-cobordisms}\label{subsec:qCob}
We refine cobordisms by replacing the monoid $\Mon(\bullet,\circ)$ with the magma $\Mag(\bullet,\circ)$.
Let $w \in \Mag(\bullet,\circ)$.
We denote by $F_w$ the compact surface $F_{w'}$ defined in Section~\ref{subsec:Cob}, where $w'$ is the word in $\Mon(\bullet,\circ)$ obtained from $w$ by forgetting its parenthesization. 
In the same manner, we use the notation $R^{w_+}_{w_-,\sigma}$.

\begin{definition}\label{def:qCob}
 Let $w_+, w_- \in \Mag(\bullet,\circ)$ with $|w_+^\circ|=|w_-^\circ|$.
 A \emph{$q$-cobordism} from $F_{w_+}$ to $F_{w_-}$ is a cobordism from $F_{w_+}$ to $F_{w_-}$ equipped with the parenthesizations of $w_+$ and $w_-$.
\end{definition}

We define the category $\Cob_q$ of $q$-cobordisms by the same way as the category $\Cob$ in Section~\ref{subsec:Cob}.
Similarly, the wide subcategory $\LCob_q$ of \emph{Lagrangian $q$-cobordisms} is obtained, which is the domain of an extension of the LMO functor.

Furthermore, we set $Z(M):=Z(\D^{-1}(M))$ for a $q$-cobordism $M=(M,\sigma,m)$, where the functor $\D$ has been defined in Section~\ref{subsec:BTtangle}, and $Z(M)$ is called the \emph{unnormalized LMO invariant} of $M$ in \cite{CHM08}.

\subsection{Top-substantial Jacobi diagrams}\label{subsec:TopSubstantial}
An element $x \in \A(\fc{g}^{+}\cup\fc{f}^{-}\cup\fc{n}^0)$ is said to be \emph{top-substantial} if $x$ is $\fc{g}^+$-substantial.
Recall that the term $S$-substantial is defined in Section~\ref{subsec:GaussianInt}.

We define the strict monoidal category $\tsA$ as follows:
The set of objects is the set $\Z_{\ge0}^2$ of pairs of non-negative integers.
A morphism from $(g,n)$ to $(f,n)$ is a pair $(x,\sigma)$, where $x \in \A(\fc{g}^{+}\cup\fc{f}^{-}\cup\fc{n}^0)$ is top-substantial and $\sigma \in \SS_n$.
There is no morphism from $(g,n)$ to $(f,m)$ if $n \neq m$, and we simply denote $\tsA((g,n),(f,n))$ by $\tsA(g,f,n)$.
The composition of $(x,\sigma) \in \tsA(g,f,n)$ and $(y,\tau) \in \tsA(h,g,n)$ is defined to be
\[\left(
\chi_{\fc{n}^0}^{-1} \circ \chi_{\fc{n}^0,\fc{n}^{0'}} \ang{
\left(x \middle/ \parbox{5.6em}{$i^{+}\mapsto i^\ast,\\ i^0\mapsto \tau^{-1}(i)^0$}\right),
\left(y \middle/ \parbox{3.8em}{$i^{-}\mapsto i^\ast,\\ i^0\mapsto i^{0'}$}\right)
}_{\fc{g}^\ast}, \sigma\tau
\right),\]
where two kinds of $\chi$'s have been defined in Section~\ref{subsec:JacobiDiag}.
The identity of $(g,n)$ is $\left[\sum\limits_{i=1}^g \strutgraph{i^+}{i^-}\right]$.
The tensor product of $(x,\sigma) \in \tsA(g,f,n)$ and $(y,\tau) \in \tsA(g',f',n')$ is defined to be
\[\left(
x\sqcup(y/j^{\pm}\mapsto(n+j)^{\pm},j^0\mapsto(n+j)^0), \sigma\oplus\tau
\right),\]
where $\sigma\oplus\tau \in \SS_{n+n'}$ maps $i$ to $\sigma(i)$ (resp.\ $n+\tau(i-n)$) if $i\leq n$ (resp.\ $i>n$).
Finally, $(g,n)\otimes(g',n'):=(g+g',n+n')$ and the unit is $(0,0)$.
Note that we omit $\sigma$ from $(x,\sigma)$ when no confusion can arise.

Let us describe the above composition law concretely.

\begin{lemma}[see {\cite[Lemma 4.4]{CHM08}}]\label{lem:PreComposition}
 Let $x=(x,\sigma)\in\tsA(g,f,n)$, $y=(y,\tau)\in\tsA(h,g,n)$ and let $C=(c_{i^-,j^+})$ be a $\fc{f}^{-}\times\fc{g}^{+}$ matrix that is regarded as a linear map $C\colon \Q\fc{g}^{+} \to \Q\fc{f}^{-}$.
 Then $([C]\sqcup x)\circ y$ is equal to
 \[\chi_{\fc{n}^0}^{-1}\chi_{\fc{n}^0,\fc{n}^{0'}}\ang{(x/j^{+}\mapsto j^\ast, j^0\mapsto \tau^{-1}(j)^{0}), (y/j^{-}\mapsto j^\ast +Cj^{+},j^0\mapsto j^{0'})}_{\fc{g}^\ast},\]
 where $(y/j^{-}\mapsto Cj^+) := \sum_{i=1}^f (c_{i^-,j^+}y / j^-\mapsto i^-)$.
 
 Similarly, let $D$ be a $\fc{g}^-\times\fc{h}^+$ matrix. Then $x\circ([D]\sqcup y)$ is equal to
 \[\chi_{\fc{n}^0}^{-1}\chi_{\fc{n}^0,\fc{n}^{0'}}\ang{(x/i^{+}\mapsto i^\ast +Di^{-}, i^0\mapsto \tau^{-1}(i)^{0}), (y/i^-\mapsto i^\ast,i^0\mapsto i^{0'})}_{\fc{g}^\ast}.\]
\end{lemma}

The next lemma is proven by applying the previous lemma repeatedly.

\begin{lemma}[see {\cite[Lemma 4.5]{CHM08}}] \label{lem:Composition}
 Let $x,y$ be as above and suppose that they are decomposed as $x=[X/2]\sqcup x^Y$ and $y=[Y/2]\sqcup y^Y$, where
 \[X=\begin{pmatrix}
 O & X^{+-} & X^{+0} \\
 X^{-+} & X^{--} & X^{-0} \\
 X^{0+} & X^{0-} & X^{00}
 \end{pmatrix},\quad
 Y=\begin{pmatrix}
 O & Y^{+-} & Y^{+0} \\
 Y^{-+} & Y^{--} & Y^{-0} \\
 Y^{0+} & Y^{0-} & Y^{00}
 \end{pmatrix}.\]
 Then $x\circ y$ is equal to
\begin{align*}
 & [X^{--}/2+Y^{+-}X^{+-}+X^{-+}Y^{--}X^{+-}/2] \\
 &\quad \sqcup\chi_{\fc{n}^0}^{-1}\chi_{\fc{n}^0,\fc{n}^{0'}}\left([\tau^{-1}(X^{00}/2+X^{-0})+Y^{00}/2+Y^{+0}]\sqcup\ang{L,R}_{\fc{g}^\ast}\right),
\end{align*}
where
\begin{align*}
 L &:= \left([\tau^{-1}X^{+0}]\sqcup x^Y/i^{+} \mapsto i^\ast +X^{-+}Y^{--}i^{-} +Y^{+-}i^{-}, i^0 \mapsto \tau^{-1}(i)^0 \right), \\
 R &:= \left([Y^{--}/2]/i^{-} \mapsto i^\ast\right)\sqcup\left([Y^{-0}]\sqcup y^Y/i^{-} \mapsto i^\ast +X^{-+}i^{+}, j^0 \mapsto j^{0'}\right).
\end{align*}
\end{lemma}

\begin{corollary}\label{cor:Composition}
 In the above lemma, if $Y=
\begin{pmatrix}
 O & I_g & O \\
 I_g & O & O \\
 O & O & O
\end{pmatrix}$
 and $\varepsilon(x)=\varepsilon(y)=\varnothing$, then $(x\circ y)^s = [X/2]$.
\end{corollary}

\begin{proof}
 By the previous Lemma, we have
\begin{align*}
 x\circ y &= [X^{--}/2+X^{+-}]\sqcup\chi_{\fc{n}^0}^{-1}\chi_{\fc{n}^0,\fc{n}^{0'}}\bigl([X^{00}/2+X^{-0}/2] \\
 &\quad \sqcup\ang{([X^{+0}]\sqcup x^Y/i^{+}\mapsto i^\ast+i^{+}),(y^Y/i^{-}\mapsto i^\ast +X^{-+}i^{+})}_{\fc{g}^\ast}\bigr).
\end{align*}
 Hence, using Lemma~\ref{lem:filter}, $(x\circ y)^s$ is equal to
 \[[X^{--}/2+X^{+-}]\sqcup\chi_{\fc{n}^0}^{-1}\chi_{\fc{n}^0,\fc{n}^{0'}}\left([X^{00}/2+X^{-0}/2]\sqcup[X^{+0}]\right) = [X/2].\]
\end{proof}

\subsection{Construction of an extension of the LMO functor}\label{subsec:ConstExtLMOfunc}
In \cite{CHM08}, a certain element $\TT_g \in \A(\fc{g}^+\cup\fc{g}^-)$ is introduced, which roughly corresponds to the bottom-top tangle $T_w$ defined in Section \ref{subsec:BTtangle}.
We first set
\[\lambda(x,y;r):=\chi_{\{r\}}^{-1}\left(\BCHgraph{x}{y}{r}\right) \in \A(\emptyset,\{x,y,r\}),\]
where \rotatebox{90}{$[\ ]$} is the map $\exp_\natural\colon \A_{>0}(\downarrow) \to \A(\downarrow)$.
Next, $\TT(x_+,x_-) \in \A(\{x_+,x_-\})$ is defined to be
\[U_+^{-1} \sqcup U_-^{-1} \sqcup \int_{\{r_+,r_-\}} \ang{\lambda(y_-,x_-;r_-)\sqcup\lambda(x_+,y_+;r_+),\chi^{-1}Z(T_\bullet^\nu)}_{\{y_+,y_-\}}.\]
Finally, $\TT_g$ is defined by
\[\TT_g := \TT(1^+,1^-) \sqcup \dots \sqcup \TT(g^+,g^-) \in \A(\fc{g}^+\cup\fc{g}^-).\]
Since $\TT(x_+,x_-)$ is group-like and its $s$-reduction is equal to $\left[\strutgraph{x_+}{x_-}\right]$ (\cite[Lemma~4.9]{CHM08}), $\TT_g$ is also group-like and its $s$-reduction is equal to $[I_g]$.

The proof of the following key lemma is similar to \cite[Lemma~4.10]{CHM08}, however, we must pay attention to the additional map $\chi_{\fc{n}^0}^{-1} \circ \chi_{\fc{n}^0,\fc{n}^{0'}}$ in the definition of the composition law in the category $\tsA$.

\begin{lemma} \label{lem:key}
 Let $B=(B,\gamma)$, $C=(C,\omega)$ be bottom-top $q$-tangles of type $(v,u,\sigma)$, $(w,v,\tau)$ respectively and suppose that $\D(B)$, $\D(C)$ are Lagrangian.
 Then we have
 \[\left(\chi_{\pi_0(\gamma^{-}\cup\omega^{+}\cup\gamma^0\omega^0)}^{-1} Z(B\circ C), \sigma\tau\right) = (\chi_{\pi_0 \gamma}^{-1}Z(B), \sigma) \circ \TT_{|v^\bullet|} \circ (\chi_{\pi_0 \omega}^{-1}Z(C), \tau),\]
 where $\TT_{|v^\bullet|}=(\TT_{|v^\bullet|},\Id_{\SS_{|v^\circ|}})$ is regarded as an element of $\tsA(|v^\bullet|,|v^\bullet|,|v^\circ|)$.
\end{lemma}

\begin{proof}
 Let $(K,\gamma)$, $(L,\omega)$ be surgery presentations of $B$, $C$ respectively.
 Let $T^-$ (resp.\ $T^+$) be the framed link in $B\circ C$ obtained from $\gamma^+$ and the bottom components of $T_v$ (resp.\ $\omega^-$ and the top components of $T_v$) by gluing their boundaries.
 Finally, let $T:=T^-\cup T^+$ and $S:=K\cup T\cup L$.
 Then $(S,\gamma^{-}\cup\omega^{+}\cup\gamma^0\omega^0)$ is a surgery presentation of $B\circ C$, and thus
 \[Z(B\circ C) = U_+^{-\sigma_+(S)} \sqcup U_-^{-\sigma_-(S)} \sqcup \int_{\pi_0 S} \chi_{\pi_0 S}^{-1}\ZK\left( S^\nu \cup (\gamma^{-}\cup\omega^{+}\cup\gamma^0\omega^0) \right).\]
 Here, the above integration is computed as follows:
\begin{align*}
 & \int_{\pi_0 S}\chi_{\pi_0 S}^{-1} \left(\ZK(K^\nu\cup\gamma)\natural (Z(T_w)\sharp\nu^{\otimes 2g}) \natural\ZK(L^\nu\cup\omega)\right) \\
 &= \int_{\pi_0 T}\chi_{\pi_0 T}^{-1}\int_{\pi_0 K}\chi_{\pi_0 K}^{-1}\ZK(K^\nu\cup\gamma)\natural (Z(T_w)\sharp\nu^{\otimes 2g}) \natural\int_{\pi_0 L}\chi_{\pi_0 L}^{-1}\ZK(L^\nu\cup\omega),
\end{align*}
 where $g:=|v^\bullet|$, and $\natural$ means the composition in the category $\A$.
 It follows from the identity $\sigma_\pm(S) =\sigma_\pm(K) +g +\sigma_\pm(L)$ (\cite[(4-2)]{CHM08}) that 
 \[Z(B\circ C) = U_{+}^{-g}\sqcup U_{-}^{-g}\sqcup\int_{\pi_0 T}\chi_{\pi_0 T}^{-1} \left( Z(K^\nu\cup\gamma)\natural (Z(T_w)\sharp\nu^{\otimes 2g}) \natural Z(L^\nu\cup\omega) \right).\]
 Therefore, $\chi_{\pi_0(\gamma^{-}\cup\omega^{+}\cup\gamma^0\omega^0)}^{-1} Z(B\circ C)$ is equal to
\begin{align*}
 & U_{+}^{-g}\sqcup U_{-}^{-g}\sqcup \int_{\pi_0 T} \chi_{\pi_0(T\cup\gamma^0\omega^0)}^{-1} \left( \chi_{\pi_0 \gamma^{-}}^{-1}Z(K^\nu\cup\gamma)\natural (Z(T_w)\sharp\nu^{\otimes 2g}) \natural\chi_{\pi_0 \omega^{+}}^{-1}Z(L^\nu\cup\omega) \right).
\end{align*}
 
 Let us introduce notation to calculate the above integration:
\begin{align*}
 \pi_0(T^{-}) &= \fc{g}^\bot, & \pi_0(T_{w^\bullet}^{+}) &= \fc{g}^\triangledown, & \pi_0(\omega^{+}) &= \fc{h}^+, \\
 \pi_0(\gamma^{+}) &= \fc{f}^\cup, & \pi_0(T_{w^\bullet}^{-}) &= \fc{g}^\vartriangle, & \pi_0(\omega^{-}) &= \fc{h}^\cap, \\
 \pi_0(\gamma^{-}) &= \fc{f}^{-}, & & & \pi_0(T^{+}) &= \fc{g}^\top,
\end{align*}
 where $f:=|u^\bullet|$, $h:=|w^\bullet|$.
 Then, the above integrand is computed as follows:
\begin{align*}
 & \chi_{\pi_0(T\cup\gamma^0\omega^0)}^{-1} \left( \chi_{\pi_0 \gamma^+}\chi_{\pi_0 \gamma^\pm}^{-1}Z(B) \natural \chi_{\pi_0 T_{w^\bullet}}\chi_{\pi_0 T_{w^\bullet}}^{-1}(Z(T_{w^\bullet}^\nu)\downarrow^{\fc{n}^0}) \natural \chi_{\pi_0 \omega^-}\chi_{\pi_0 \omega^\pm}^{-1}Z(C) \right) \\
 &= \chi_{\pi_0(T\cup\gamma^0\omega^0)}^{-1} \biggl\langle \bigsqcup_{i=1}^g \BCHgraph{i^\vartriangle}{i^\cup}{i^\bot} \sqcup \bigsqcup_{i=1}^g \BCHgraph{i^\cap}{i^\triangledown}{i^\top}\ , \chi_{\pi_0 T_{w^\bullet}}^{-1}Z(T_{w^\bullet}^\nu) \sqcup \chi_{\pi_0 \gamma^\pm}^{-1}Z(B) \natural \chi_{\pi_0 \omega^\pm}^{-1}Z(C) \biggr\rangle_{\parbox{2em}{\scriptsize $\mathord{\vartriangle}\mathord{\cap} \\[-4pt] \mathord{\cup}\mathord{\triangledown}$}} \\ 
 &= \chi_{\pi_0(\gamma^0\omega^0)}^{-1} \ang{ \ang{ \Lambda, \chi_{\pi_0 T_{w^\bullet}}^{-1}Z(T_{w^\bullet}^\nu) }_{\vartriangle \triangledown}, \chi_{\pi_0 \gamma^\pm}^{-1}Z(B) \natural \chi_{\pi_0 \omega^\pm}^{-1}Z(C) }_{\cup \cap} \\
 &= \chi_{\pi_0(\gamma^0\omega^0)}^{-1} \ang{ \ang{ \Lambda, \chi_{\pi_0 T_{w^\bullet}}^{-1}Z(T_{w^\bullet}^\nu) }_{\vartriangle \triangledown}, \chi_{\pi_0 \gamma^0, \pi_0 \omega^0} \left(\chi_{\pi_0 \gamma}^{-1}Z(B) \sqcup \chi_{\pi_0 \omega}^{-1}Z(C)\right) }_{\cup \cap},
\end{align*}
 where the $i$th component of $\gamma^0$ corresponds to the $\tau^{-1}(i)$th component of $\omega^0$, and $\Lambda \in \A(\fc{g}^\vartriangle \cup \fc{f}^\cup \cup \fc{g}^\bot \cup \fc{h}^\cap \cup \fc{g}^\triangledown \cup \fc{g}^\top)$ is defined by
 \[\Lambda := \bigsqcup_{i=1}^g \lambda(i^\vartriangle,i^\cup;i^\bot) \sqcup \bigsqcup_{i=1}^g \lambda(i^\cap,i^\triangledown;i^\top).\]
 Therefore, the above integration is written as follows:
\[\chi_{\pi_0(\gamma^0\omega^0)}^{-1} \chi_{\pi_0 \gamma^0, \pi_0 \omega^0} \ang{ \int_{\pi_0 T} \ang{\Lambda, \chi_{\pi_0 T_{w^\bullet}}^{-1}Z(T_{w^\bullet}^\nu) }_{\vartriangle \triangledown}, \chi_{\pi_0 \gamma}^{-1}Z(B) \sqcup \chi_{\pi_0 \omega}^{-1}Z(C) }_{\cup \cap}.\]
 Here,
\begin{align*}
 & U_{+}^{-g} \sqcup U_{-}^{-g} \sqcup \int_{\pi_0 T} \ang{\Lambda, \chi_{\pi_0 T_{w^\bullet}}^{-1}Z(T_{w^\bullet}^\nu) }_{\vartriangle \triangledown} \\
 &= \bigsqcup_{i=1}^g U_{+}^{-1} \sqcup U_{-}^{-1} \sqcup \int_{\{i^\bot,i^\top\}} \ang{ \lambda(i^\vartriangle,i^\cup;i^\bot) \sqcup \lambda(i^\cap,i^\triangledown;i^\top)} \\
 &= \bigsqcup_{i=1}^g \TT(i^\cap,i^\cup) 
 = \left(\TT_g \middle/ \parbox{40pt}{$i^+ \mapsto i^\cap \\ i^- \mapsto i^\cup$} \right). 
\end{align*}
 Therefore, the left-hand side of this lemma is equal to
\begin{align*}
 \chi_{\pi_0(\gamma^0\omega^0)}^{-1} \chi_{\pi_0 \gamma^0, \pi_0 \omega^0} \ang{ \left(\TT_g \middle/ \parbox{40pt}{$i^+ \mapsto i^\cap \\ i^- \mapsto i^\cup$} \right), \chi_{\pi_0 \gamma}^{-1}Z(B) \sqcup \chi_{\pi_0 \omega}^{-1}Z(C) }_{\cup \cap}.
\end{align*}
 
 On the other hand, the right-hand side of this lemma is obtained as follows:
 Connect the $i^-$-colored vertices and $i^+$-colored vertices in $\TT_g$ with the $i^\cup$-colored vertices in $\chi_{\pi_0 \gamma}^{-1}Z(B)$ and $i^\cap$-colored vertices in $\chi_{\pi_0 \omega}^{-1}Z(C)$ respectively and apply the composite map $\chi_{\pi_0(\gamma^0\omega^0)}^{-1} \chi_{\pi_0 \gamma^0, \pi_0 \omega^0}$.
\end{proof}

We can now define $\Zt$, which is the main purpose of this paper.

\begin{definition}
 Let $v, w \in \Mag(\bullet,\circ)$ with $|v^\circ|=|w^\circ|=n$ and let $f:=|v^\bullet|$, $g:=|w^\bullet|$.
 The \emph{normalized LMO invariant} $\Zt$ of $(M,\sigma,m) \in \LCob_q(w,v)$ is defined by
 \[\Zt(M,\sigma,m) := \left(\chi_{\pi_0 \gamma}^{-1} Z(B,\gamma), \sigma\right) \circ_{\tsA} \left(\TT_g, \Id_{\SS_n}\right) \in \tsA(g,f,n).\]
\end{definition}

Let $\Zt^s(M)$ (resp.\ $\Zt^Y(M)$) denote the $s$- (resp.\ $Y$-) reduction of $\Zt(M)=\Zt(M,\sigma,m)$.
The next lemma follows from Lemma~\ref{lem:bottom-top tangle} and Corollary~\ref{cor:Composition}~(1).

\begin{lemma}
 $\Zt(M,\sigma,m)$ is group-like and $\Zt^s(M,\sigma,m) = [\Lk(M)/2]$.
\end{lemma}

\begin{theorem}\label{thm:Zt}
 The normalized LMO invariant defines a tensor-preserving functor
 \[\Zt \colon \LCob_q\to\tsA,\]
 which is called \emph{(}an extension of\/\textup{)} the \emph{LMO functor}, that is, $\Zt$ satisfies the following conditions:
\begin{enumerate}
 \item $\Zt(\Id_w)=\Id_{(|w^\bullet|,|w^\circ|)}$,
 \item $\Zt(M\circ N)=\Zt(M)\circ\Zt(N)$,
 \item $\Zt(e)=(0,0)$,
 \item $\Zt(M\otimes N)=\Zt(M)\otimes\Zt(N)$.
\end{enumerate}
\end{theorem}

\begin{proof}
 (3) and (4) follows from these of $\ZK$ (see Section~\ref{subsec:DefKontsevichLMOinv}).
 (2) is due to Lemma~\ref{lem:key}.
 Let us prove (1).
 By the previous lemma, $\Zt(\Id_w)$ is group-like, thus it is written as $\varnothing+(\ideg>0)$.
 We assume that the higher-degree terms are not zero, and write
\[\Zt(\Id_w) = \varnothing +x +(\ideg>k) \quad(\deg x=k).\]
 It follows from (2) that $\Zt(\Id_w) = \Zt(\Id_w)\circ\Zt(\Id_w)$.
 By comparing both side of this equality,
\begin{align*}
 & \chi_{\fc{n}^0}^{-1} \chi_{\fc{n}^0,\fc{n}^{0'}} \ang{ \left( \Zt^Y(\Id_w) \middle/ i^{+}\mapsto i^\ast \right), \left( \Zt^Y(\Id_w) \middle/ \parbox{3.7em}{$i^{-}\mapsto i^\ast\\ i^0\mapsto i^{0'}$} \right) }_{\fc{g}^\ast} \\ 
 &= \chi_{\fc{n}^0}^{-1} \chi_{\fc{n}^0,\fc{n}^{0'}}\left( \varnothing +x +(x/i^0\mapsto i^{0'}) +(\ideg>k) \right) \\
 &= \varnothing +2x +(\ideg>k).
\end{align*}
 Thus we have $x=2x$, and it is a contradiction.
 Therefore, $\Zt^Y(\Id_w)=\varnothing$, and we conclude (4).
\end{proof}

\begin{remark}
 By definition, if $w \in \Mag(\bullet)$, then $\Zt$ coincides with the original LMO functor defined in \cite{CHM08}.
 Moreover, the diagram
 \[\xymatrix{\LCob_q \ar[r]^-{\K} \ar[d]_-{\Zt} & \LCob^\bullet_q \ar[d]^-{\text{original $\Zt$}} \\
 \tsA \ar[r]^-{(-/i^0\mapsto0)} & \tsA}\]
 commutes, where $\K$ is the functor defined in the same manner as $\K$ in Remark~\ref{rem:Kill}.
 In contrast, if $w \in \Mag(\circ)$, then we have
 \[\Zt(\D([-1,1]^3,-\sigma)) = \left(\chi_{\pi_0 \sigma}^{-1} \ZK(-\sigma), \Id_{\SS_n}\right)\]
 for any (not necessarily pure) string link $\sigma$ in $[-1,1]^3$ on $n$ strands.
\end{remark}

\begin{remark}\label{rem:Rational}
 By the same manner, one can define the category $\Q\LCob$ of \emph{rational Lagrangian $q$-cobordisms} and the functor $\Zt\colon \Q\LCob_q \to \tsA$, see \cite[Remark~2.8, 3.19, 4.14]{CHM08}.
 Indeed, the formal Gaussian integral is originally defined under rational condition and linking matrices in this paper are defined homologically.
\end{remark}

\section{Generators and the values on them}\label{sec:GenValue}
In \cite{CHM08}, it was proven that the category $\Cob^\bullet$, see Remark~\ref{rem:Kill}, is generated as a monoidal category by the morphisms $\mu$, $\eta$, $\Delta$, $\varepsilon$, $S^{\pm 1}$, $\psi_{\bullet,\bullet}^{\pm 1}$, $v_\pm$ and $Y$ listed in Table~\ref{tab:Gen}, and the values on them was calculated up to internal degree 2.
We shall add some morphisms and calculate the values on them.

\subsection{Generators of the category $\LCob$}\label{subsec:Gen}
We first define some bottom-top tangles as Table~\ref{tab:Gen}.
By Lemma~\ref{lem:Lagrangian}, the corresponding cobordisms are Lagrangian, and we use the same notation to represent the cobordisms.
The bottom-top tangles except $\psi_{\bullet,\circ}$, $\psi_{\circ,\bullet}$, $\psi_{\circ,\circ}$, $\tau$ and $\beta$ are same as in \cite{CHM08}.

\begin{table}
\centering
$ \begin{array}{cccc}
 \eta:=\raisegraph{-1em}{2.5em}{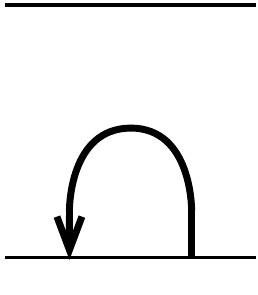} & \varepsilon:=\raisegraph{-0.8em}{2.5em}{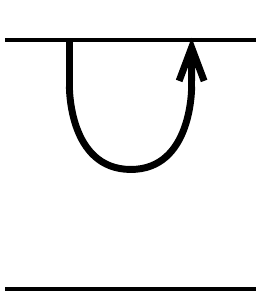} & v_{+}:=\raisegraph{-1em}{2.5em}{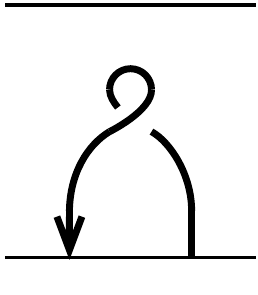} & v_{-}:=\raisegraph{-1em}{2.5em}{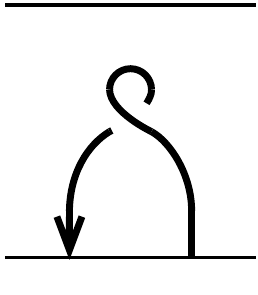} \\
 \mu:=\raisegraph{-2em}{5em}{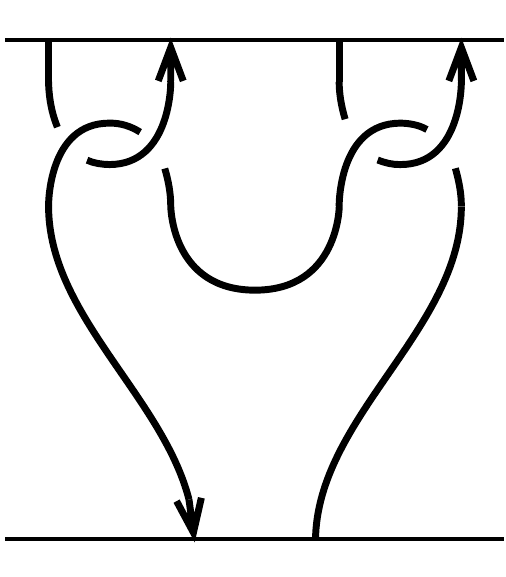} & \Delta:=\raisegraph{-2em}{5em}{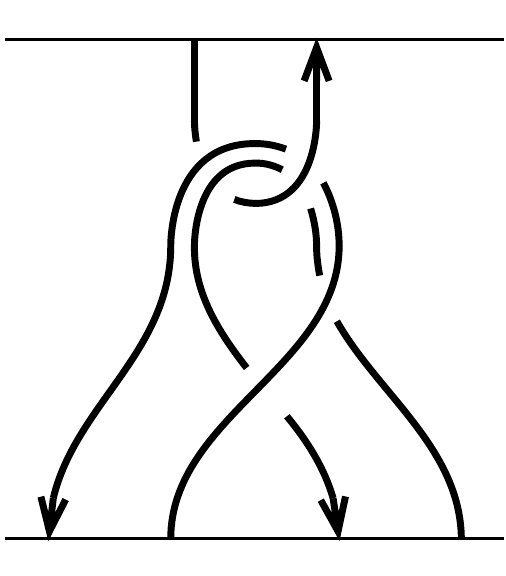} & \psi_{\bullet,\bullet}:=\raisegraph{-2em}{5em}{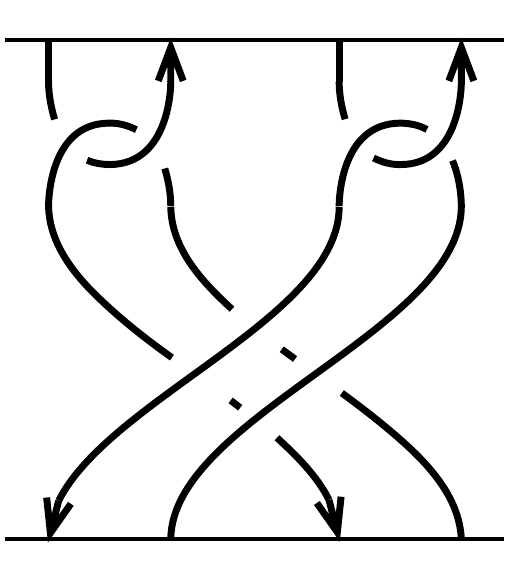} & \psi_{\circ,\circ}:=\raisegraph{-2em}{5em}{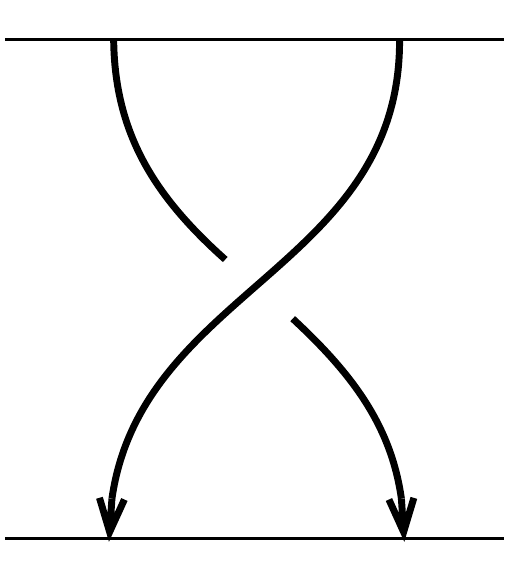} \\
 \tau:=\raisegraph{-2em}{5em}{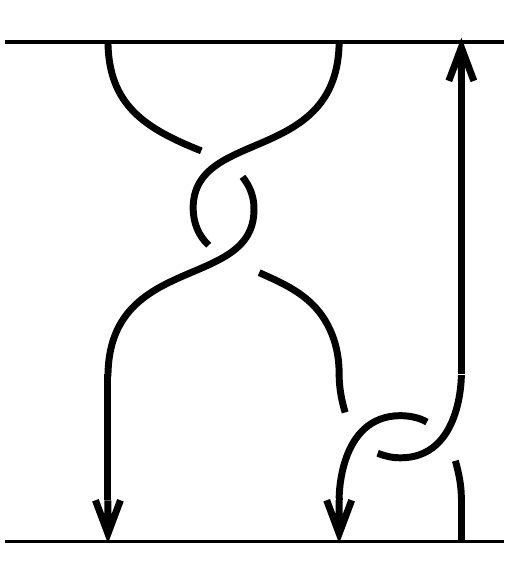} & \beta:=\raisegraph{-2em}{5em}{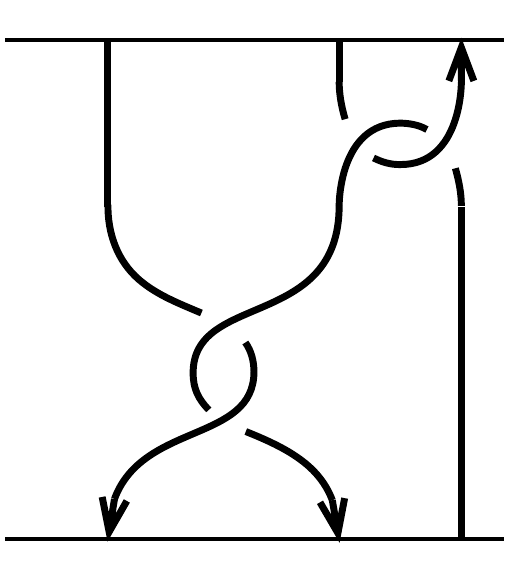} & \psi_{\bullet,\circ}:=\raisegraph{-2em}{5em}{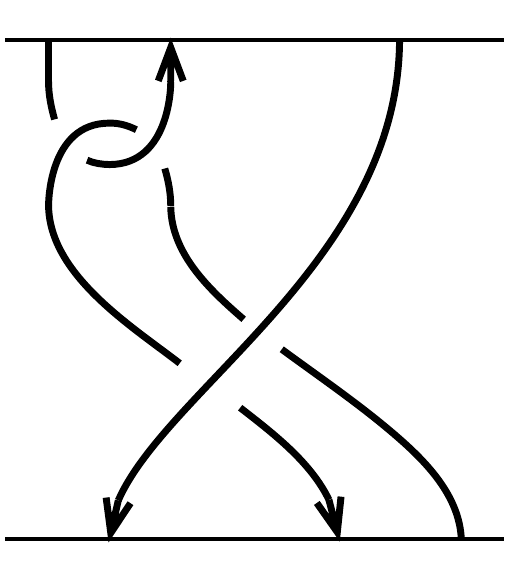} & \psi_{\circ,\bullet}:=\raisegraph{-2em}{5em}{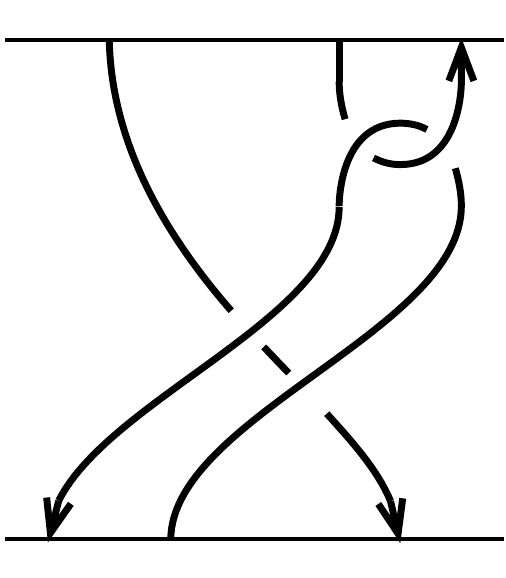} \\
 S:=\raisegraph{-2em}{5em}{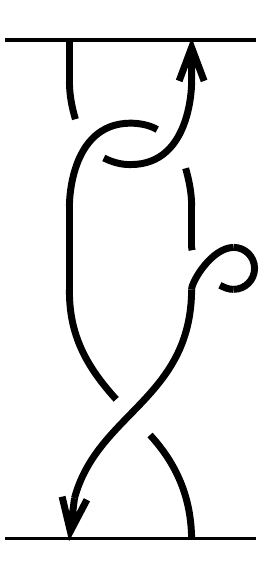} & \multicolumn{3}{l}{Y:=\raisegraph{-2em}{5em}{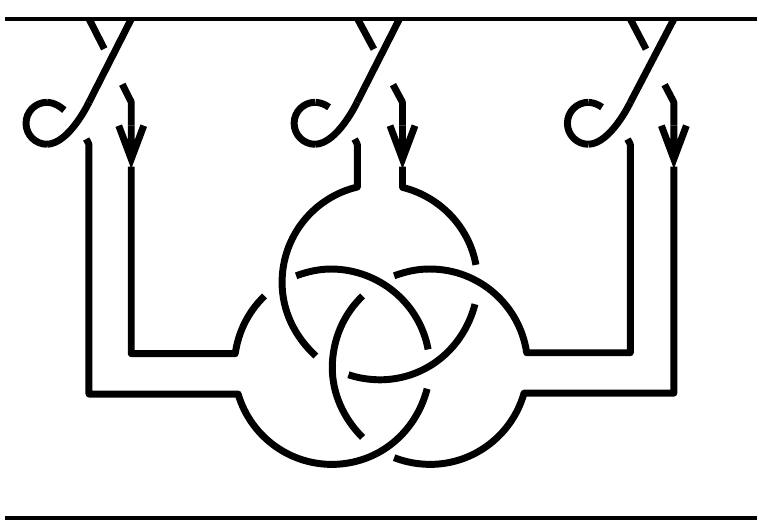}}
\end{array} $ \vspace{1em}
\caption{Some morphisms of the category $\LCob$}
\label{tab:Gen}
\end{table}%

\begin{remark}
 When we write $Y$ as a morphism of $\LCob_q$, we always regard it as a morphism from $(\bullet\bullet)\bullet$ to $e$.
 In general, when we regard a morphism of $\Cob$ as one of $\Cob_q$, we use the left-handed parenthesization $(\cdots((\ast\ast)\ast)\dots\ast)$ unless otherwise stated.
\end{remark}

\begin{remark}\label{rem:braided}
 Four kinds of $\psi$'s and their inverses induce braidings $\psi_{v,w}\colon v\otimes w \to w\otimes v$, which make the category $\LCob$ into a braided strict monoidal category.
\end{remark}

It is easy to see that the cobordisms
\begin{center}
 \raisegraph{-2em}{5em}{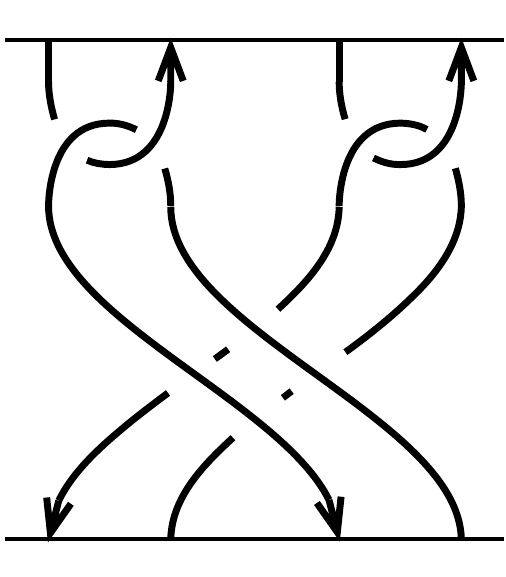}\ ,\quad
 \raisegraph{-2em}{5em}{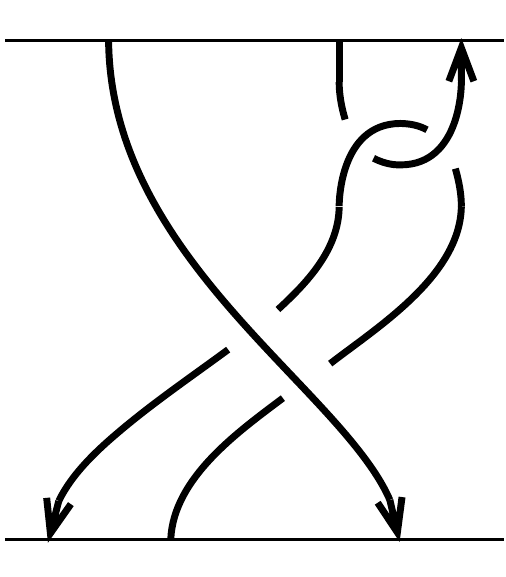}\ ,\quad
 \raisegraph{-2em}{5em}{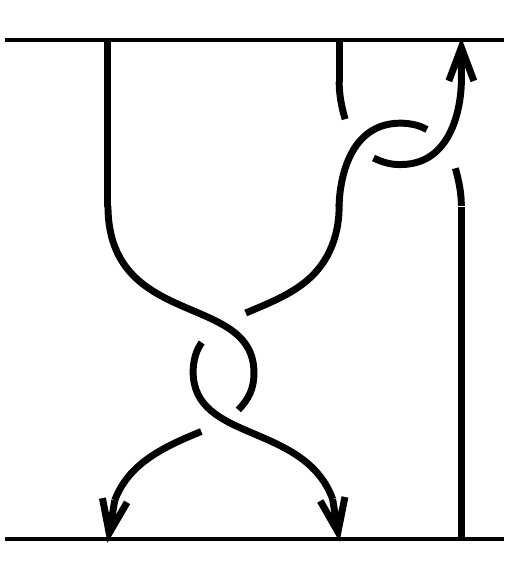} and
 \raisegraph{-2em}{5em}{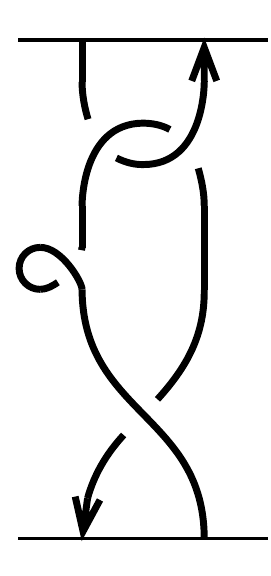}
\end{center}
are the (two-sided) inverses of $\psi_{\bullet,\bullet}$, $\psi_{\bullet,\circ}$, $\beta$ and $S$ in the category $\LCob_q$.
In the same manner, one finds the inverses of $\psi_{\circ,\circ}$, $\psi_{\circ,\bullet}$ and $\tau$.
(By Proposition~\ref{prop:InvertElem}, it suffices to check that they are left or right inverses.)

\begin{proposition}\label{prop:Gen}
 The category $\LCob$ is generated as monoidal category by the morphisms $\mu$, $\eta$, $\Delta$, $\varepsilon$, $S^{\pm1}$, $\psi_{\bullet,\bullet}^{\pm 1}$, $v_\pm$, $Y$, $\psi_{\bullet,\circ}^{\pm1}$, $\psi_{\circ,\bullet}^{\pm1}$, $\psi_{\circ,\circ}^{\pm1}$ and $\tau$.
 Moreover, $\LCob_q$ is generated by the morphism
 \[P_{u,v,w}:=\raisegraph{-2em}{5em}{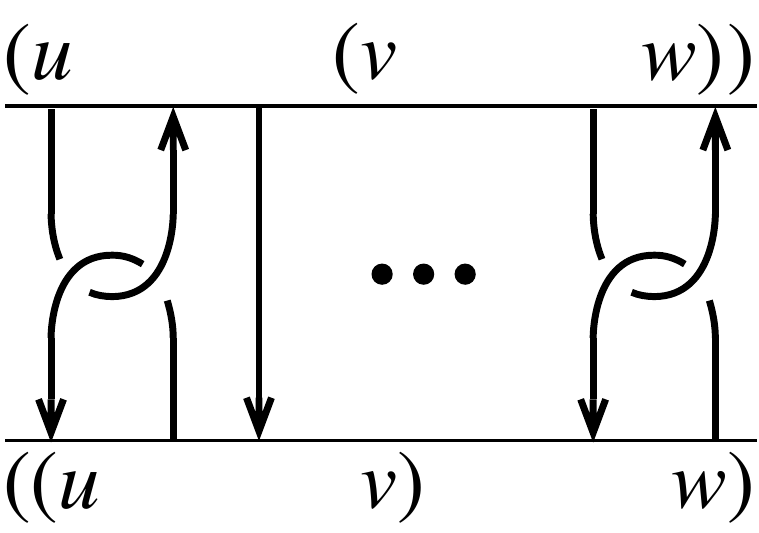}\]
 for $u,v,w \in \Mag(\bullet,\circ)$ and their inverses together with the above morphisms.
\end{proposition}

\begin{proof}
 The latter follows from the former.
 We prove that every $M \in \LCob(w,v)$ belongs to the category $\CC$ generated by the above morphisms.
 Using the braidings, it suffices to consider the case of $M=(M,\Id_{\SS_n},m) \in \LCob(\bullet^{\otimes g}\circ^{\otimes n},\bullet^{\otimes f}\circ^{\otimes n})$.
 Namely, the bottom-top tangle $(B,\gamma):=\D^{-1}(M)$ is represented as follows: 
\begin{center}
 \includegraphics[width=0.99\textwidth]{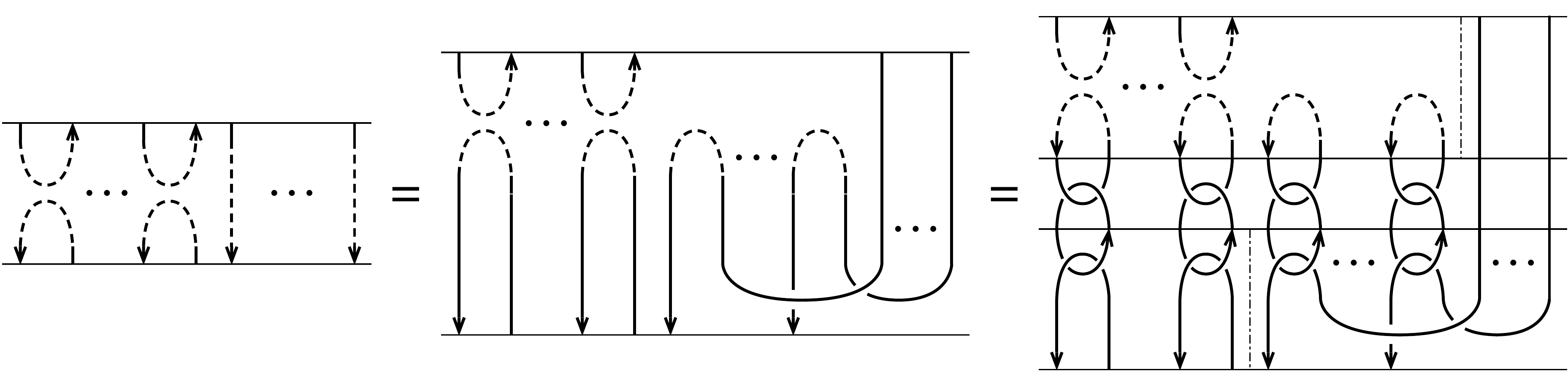}.
\end{center}
 (This correspondence is essentially same as the preferred bijection $\tau_n$ in \cite[Section~13]{Hab06}.)
 Since $\Lk_B(\gamma^+)=O$, it suffice to show that the bottom-top tangle $\alpha_n$ at the lower right in the right-hand side of the above equality belongs to $\CC$.
 Here, one can show the following equalities
 \[\alpha_1 = (\varepsilon \otimes \Id_\circ) \circ \psi_{\bullet,\circ}^{-1} \circ \tau \circ \psi_{\circ,\bullet}^{-1},\quad
 \alpha_n = (\text{braidings})\circ\alpha_1^{\otimes n}.\]
 Therefore, the proof is completed.
\end{proof}

\begin{example}
 One can check the following decomposition
 \[\raisegraph{-2.2em}{5em}{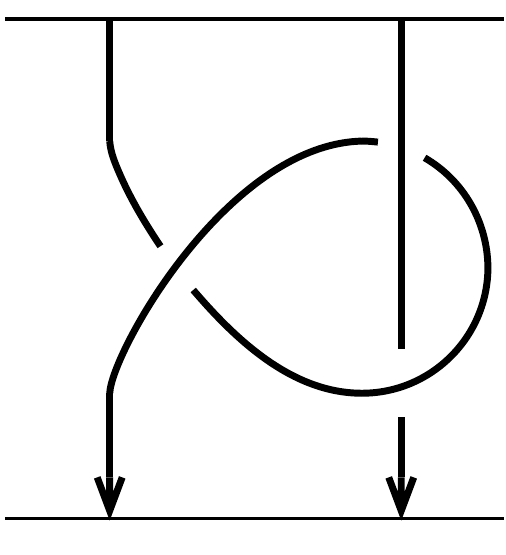} =
((\varepsilon\otimes\Id_\circ)\check{\tau}\otimes\Id_\circ)(\Id_\bullet\otimes\psi_{\circ,\circ}^2)(\psi_{\bullet,\circ}^{-1}\otimes\Id_\circ)(\tau\otimes\Id_\circ)(\Id_\circ \otimes v_{-}\otimes\Id_\circ),\]
 where the composition $\circ$ is omitted and the morphism $\check{\tau}$ is defined by
 \[\check{\tau} := \raisegraph{-2em}{5em}{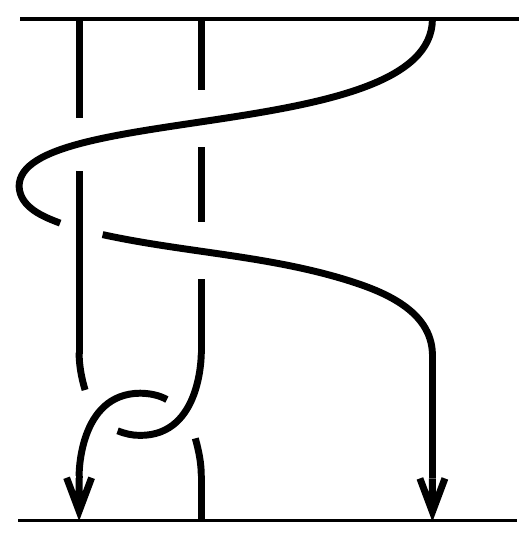} = \psi_{\bullet,\circ}^{-1} \circ \tau \circ \psi_{\bullet,\circ}\ .\]
\end{example}

\subsection{The values on the generators}\label{subsec:Value}
From now on, we assume an associator $\Phi \in \A(\downarrow\downarrow\downarrow)$ is the one derived from a rational even Drinfel'd series $\varphi(A,B) \in \Q\angg{A,B}$, which was introduced in \cite[Section~3]{LeMu97}.
By the pentagon and hexagon relations and the evenness, $\varphi(A,B)$ must be of the form
\[1 +\frac{1}{24}[A,B] +(\deg>3).\] 
Thus, under the above assumption, we have
\[\Phi = \varphi\left(\raisegraph{-1em}{2.5em}{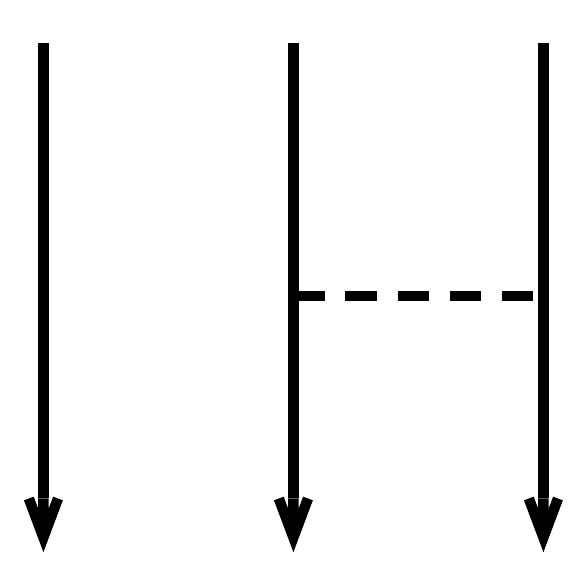},\raisegraph{-1em}{2.5em}{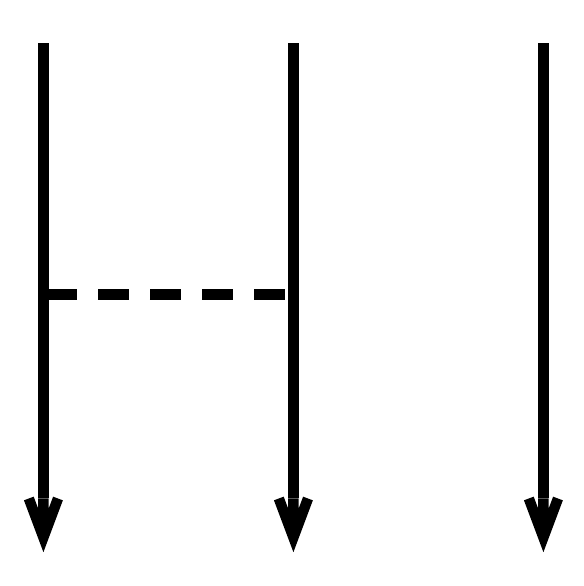}\right) = \raisegraph{-1em}{2.5em}{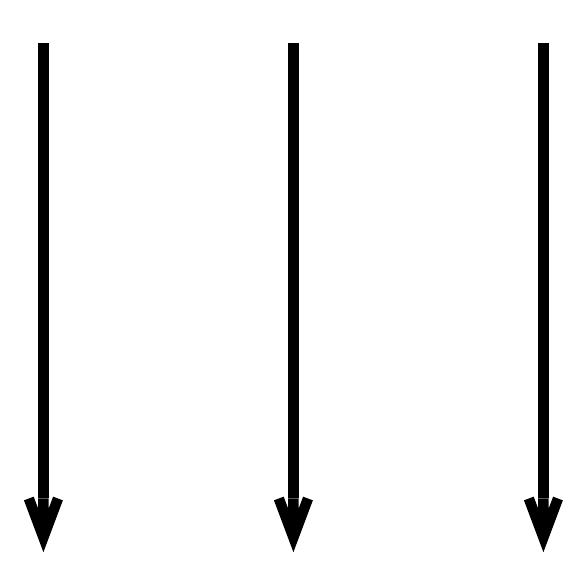} +\frac{1}{24}\raisegraph{-1em}{2.5em}{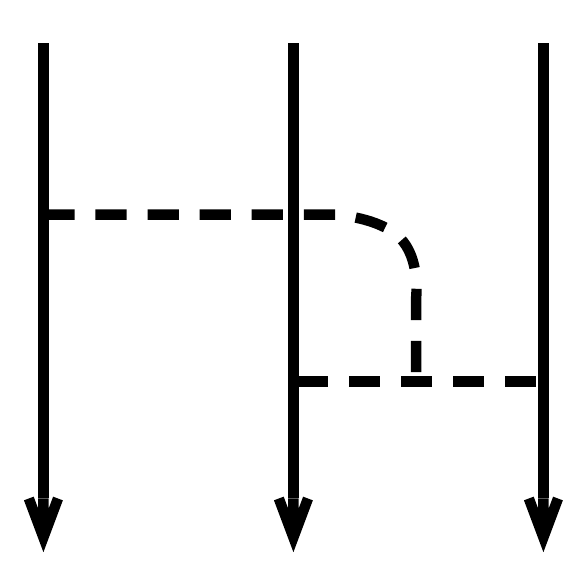} +(\deg>3).\]
We can calculate the values on $\psi_{\bullet,\circ}^{\pm1}$, $\psi_{\circ,\bullet}^{\pm1}$, $\psi_{\circ,\circ}^{\pm1}$ and $\beta$ using \cite[Lemma~5.7]{CHM08}, and the results are written in Table~\ref{tab:Value}.

\begin{table}
\centering
$ \begin{array}{c|c|c}
 M & \log_\sqcup\Zt^s(M) & \log_\sqcup\Zt^Y(M) \bmod (\ideg>2) \\ \hline
 \mu & \strutgraph{1^+}{1^-} +\strutgraph{2^+}{1^-} & -\frac{1}{2}\Ygraph{1^+}{2^+}{1^-} +\frac{1}{12}\yengraph{1^+}{1^+}{2^+}{1^-} +\frac{1}{12}\myengraph{1^+}{2^+}{2^+}{1^-} \rule[-1.5em]{0em}{3.7em} \\ \hline
 \eta & 0 & 0 \rule[-0.7em]{0em}{2em} \\ \hline
 \Delta & \strutgraph{1^+}{1^-} +\strutgraph{1^+}{1^-} & \frac{1}{2}\dYgraph{1^+}{1^-}{2^-} +\frac{1}{12}\mlambdagraph{1^+}{1^-}{2^-}{2^-} +\frac{1}{12}\lambdagraph{1^+}{1^-}{1^-}{2^-} - \frac{1}{4}\Hgraph{1^+}{1^+}{1^-}{2^-} \rule[-1.5em]{0em}{3.7em} \\ \hline
 \varepsilon & 0 & 0 \rule[-0.7em]{0em}{2em} \\ \hline
 S^{\pm1} & -\strutgraph{1^+}{1^-} & \mp\frac{1}{4} \Phigraph{1^+}{1^-} \mp\frac{1}{4}\Hgraph{1^+}{1^+}{1^-}{1^-} \rule[-1.5em]{0em}{3.7em} \\ \hline
 \psi_{\bullet,\bullet}^{\pm1} & \strutgraph{1^+}{2^-} +\strutgraph{2^+}{1^-} & \mp\frac{1}{2}\Hgraph{1^+}{2^+}{2^-}{1^-} \rule[-1.2em]{0em}{3em} \\ \hline
 v_\pm & \mp\frac{1}{2}\dstrutgraph{1^-}{1^-} & \frac{1}{48} \dPhigraph{1^-}{1^-} \rule[-1em]{0em}{3em} \\ \hline
 Y & 0 & -\uYgraph{1^+}{2^+}{3^+} +\frac{1}{2}\uHgraph{1^+}{1^+}{2^+}{3^+} +\frac{1}{2}\uHgraph{2^+}{2^+}{3^+}{1^+} +\frac{1}{2}\uHgraph{3^+}{3^+}{1^+}{2^+} \rule[-1.5em]{0em}{3.7em} \\ \hline
 P_{u,v,w} & \sum\limits_{i=1}^{|u^\bullet|+|v^\bullet|+|w^\bullet|} \!\strutgraph{i^+}{i^-} & 0 \rule[-1.1em]{0em}{3em} \\ \hline
 \psi_{\bullet,\circ}^{\pm1} & \strutgraph{1^+}{1^-} & \pm\frac{1}{2}\rTgraph{1^+}{1^-}{1^0} +\frac{1}{8}\rPigraph{1^+}{1^-}{1^0}{1^0} \rule[-1.5em]{0em}{3.7em} \\ \hline
 \psi_{\circ,\bullet}^{\pm1} & \strutgraph{1^+}{1^-} & \pm\frac{1}{2}\rTgraph{1^+}{1^-}{1^0} +\frac{1}{8}\rPigraph{1^+}{1^-}{1^0}{1^0} \rule[-1.5em]{0em}{3.7em} \\ \hline
 \psi_{\circ,\circ}^{\pm1} & \pm\frac{1}{2}\dstrutgraph{1^0}{2^0} & -\frac{1}{32}\lrPhigraph{1^0}{2^0} \pm\frac{1}{24}\lrHgraph{1^0}{2^0}{1^0}{2^0} \rule[-1.2em]{0em}{3em}\\ \hline
 \tau & \strutgraph{1^+}{1^-} +\lustrutgraph{1^+}{1^0} & -\frac{1}{2}\lTgraph{1^+}{1^-}{1^0} +\frac{1}{4}\uchairgraph{1^+}{1^+}{1^-}{1^0} +\frac{1}{12}\chairgraph{1^+}{1^-}{1^-}{1^0} +\frac{1}{12}\lPigraph{1^+}{1^-}{1^0}{1^0} -\frac{1}{4}\luPhigraph{1^+}{1^0} +\frac{1}{4}\luHgraph{1^+}{1^+}{1^0}{1^0} \rule[-1.5em]{0em}{3.7em}\\ \hline
 \tau^{-1} & \strutgraph{1^+}{1^-} -\lustrutgraph{1^+}{1^0} & \frac{1}{2}\lTgraph{1^+}{1^-}{1^0} -\frac{1}{4}\uchairgraph{1^+}{1^+}{1^-}{1^0} -\frac{1}{12}\chairgraph{1^+}{1^-}{1^-}{1^0} +\frac{1}{12}\lPigraph{1^+}{1^-}{1^0}{1^0} \rule[-1.5em]{0em}{3.7em}\\ \hline
 \beta^{\pm1} & \strutgraph{1^+}{1^-} \pm\ldstrutgraph{1^-}{1^0} & \mp\frac{1}{2}\lTgraph{1^+}{1^-}{1^0} -\frac{1}{8}\ldPhigraph{1^-}{1^0} \pm\frac{1}{12}\uchairgraph{1^+}{1^+}{1^-}{1^0} -\frac{1}{12}\lPigraph{1^+}{1^-}{1^0}{1^0} \pm\frac{1}{8}\ldHgraph{1^-}{1^-}{1^0}{1^0} \rule[-1.5em]{0em}{3.7em}\\
\end{array} $ \vspace{1em}
\caption{The values on the generators, $\tau^{-1}$ and $\beta^{\pm1}$.}
\label{tab:Value}
\end{table}%

On the other hand, to calculate the value on $\tau^{\pm1}$, we need Lemma~\ref{lem:updown} that is an upside-down version of \cite[Lemma~5.5]{CHM08}.
One can prove Lemma~\ref{lem:updown} in a similar way.

\begin{lemma}\label{lem:updown}
 Let $M \in \LCob_q(w,v)$.
 Suppose $\D^{-1}(M)$ is the composition of the $q$-tangle
 \[U:=\raisegraph{-4em}{8em}{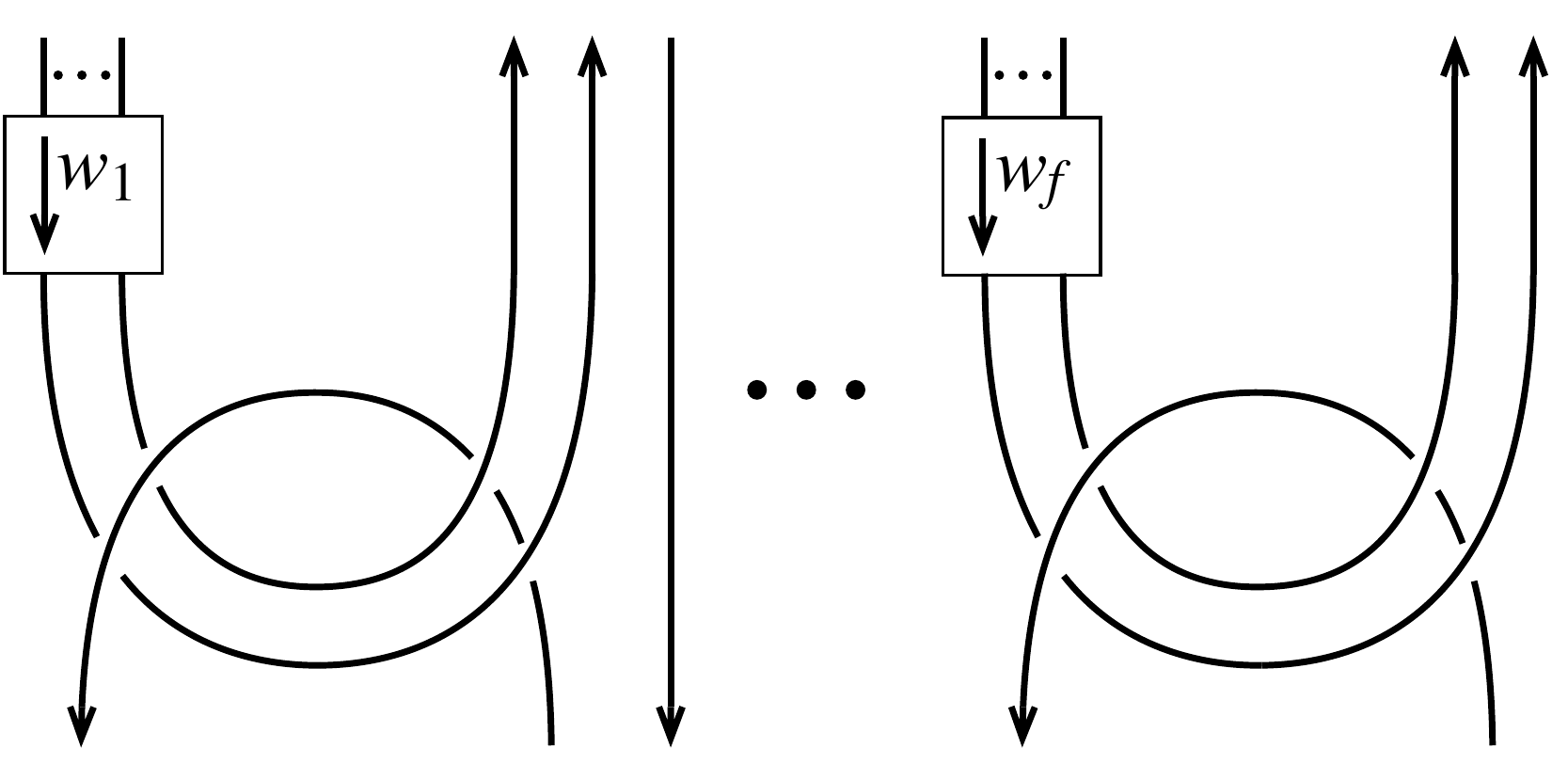}\]
 and some $q$-tangle $L$ in $[-1,1]^3$.
 Then $\TT_f\circ\Zt(M)\circ\TT_g^{-1} \in \tsA(g,f,n)$ is equal to the image by $\chi_{\pi_0(U^{+}\cup L)}^{-1}$ of the composition of the series of Jacobi diagrams
 \[\raisegraph{-5em}{10em}{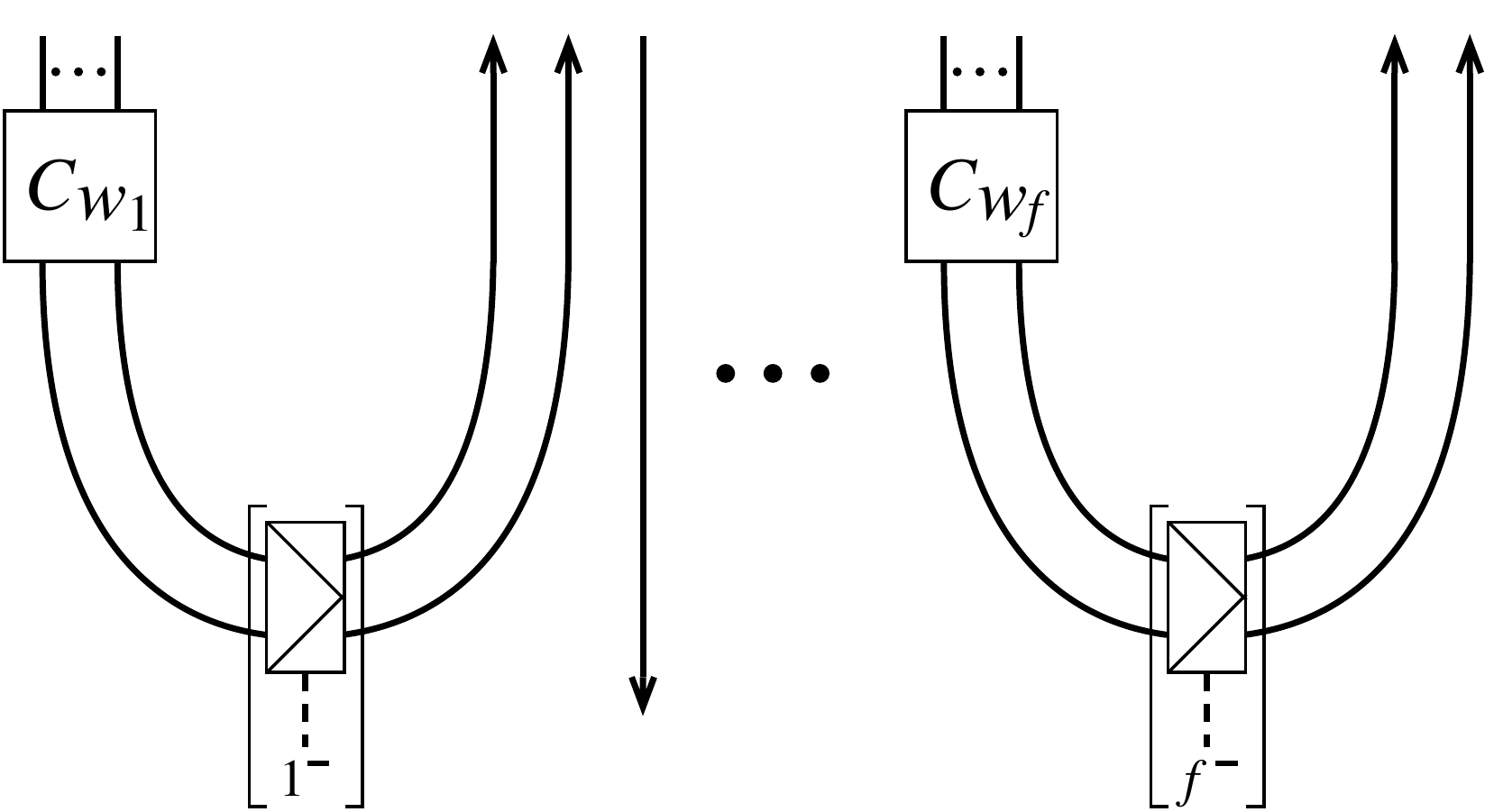}\]
 and $\ZK(L) \in \A(L,\emptyset)$ in the category $\A$.
 Here, $C_{w_i}$ and a directed rectangle is same as in \cite[Lemma~5.5]{CHM08}.
\end{lemma}

\begin{remark}
 The lower degree terms of $\TT_1$ and $\TT_1^{-1}$ are given in \cite[Lemma~5.7]{CHM08} and its proof.
\end{remark}

\section{Universality among finite-type invariants}\label{sec:Universal}
Le~\cite{Le97} proved that the LMO invariant is universal among rational-valued finite-type invariants of rational homology spheres (see \cite[p.~89, Remark~1]{Le97}).
As mentioned before, Cheptea, Habiro and Massuyeau showed that the LMO functor is universal among rational-valued finite-type invariants of 3-manifolds.
We prove that the extension $\Zt$ has the same property.

\subsection{Clasper calculus}\label{subsec:Clasper}
In this subsection, we prepare terminology of clasper calculus based on \cite{HabK00}.
Consider pairs $(M,\gamma)$ where $M$ is a compact oriented 3-manifold and $\gamma$ is a framed oriented tangle in $M$.
If $M$ has a boundary, we take into account a parametrization of $\partial M$.
Moreover, if $\gamma$ has a boundary, it must be attached to assigned points in $\partial M$.

\begin{definition}
 A \emph{graph clasper} is an oriented compact surface with a certain decomposition into three kinds of pieces: disks, bands and annuli, which are called \emph{nodes}, \emph{edges} and \emph{leaves} respectively.
 Each leaf should be connected to a node and no leaf by a band.
 Each node should be connected to nodes or leaves by exactly three bands.
\end{definition}

The surface on the left of Figure~\ref{fig:Ygraph} is one of the simplest example of a graph clasper.
For simplicity, it is represented as the graph on the right of Figure~\ref{fig:Ygraph}, that is, leaves, nodes and bands of a graph clasper $G$ are replaced with circles, points and arcs respectively, here we should remember information required for recovering $G$.

\begin{figure}[h]
 \centering
 \includegraphics[width=0.7\columnwidth]{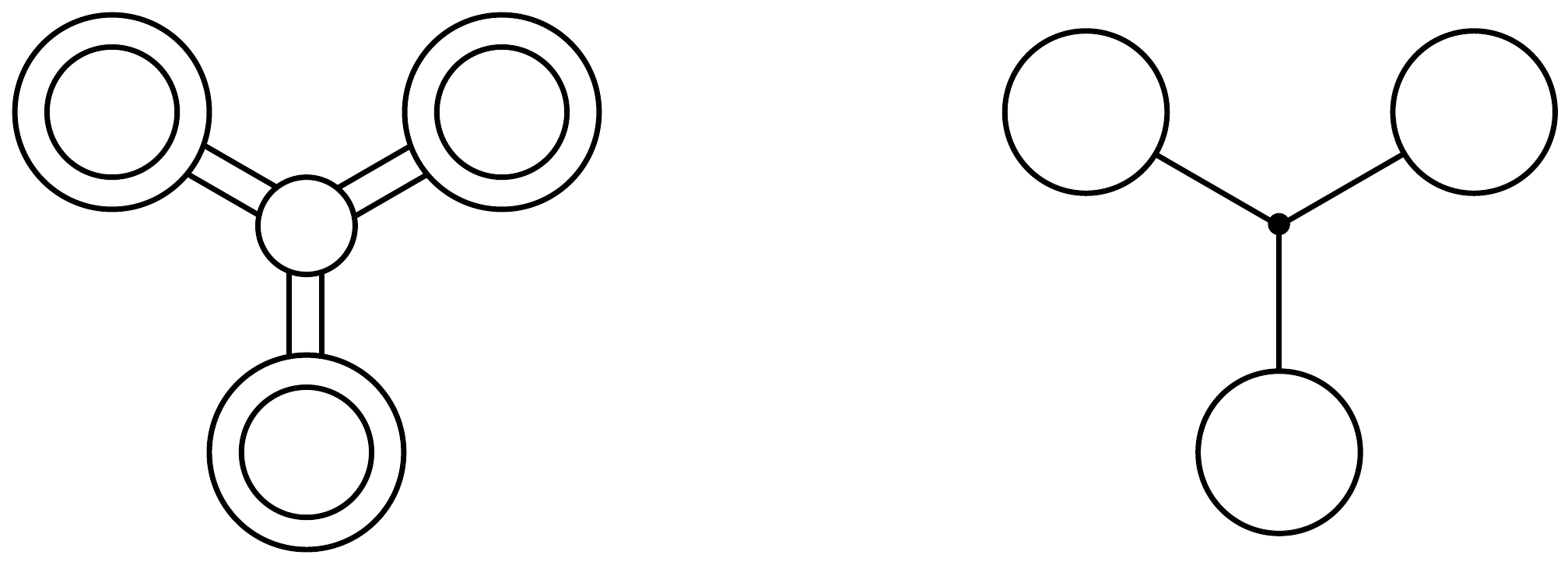}
 \caption{An example of a graph clasper and its alternative expression}
 \label{fig:Ygraph}
\end{figure}%

Next, consider surgery along a graph clasper $G$ in a 3-manifold $M$.
Let $Y(G)$ denotes the graph clasper obtained from $G$ by applying the \emph{fission rule} illustrated in Figure~\ref{fig:FissionRule} to each edge connecting two nodes.
The resulting graph is the disjoint union of $\ideg G$ copies of $Y$-graphs, where $\ideg G$ is the \emph{internal degree} of $G$, that is, the number of nodes of $G$.

\begin{figure}[h]
 \centering
 \includegraphics[width=0.6\columnwidth]{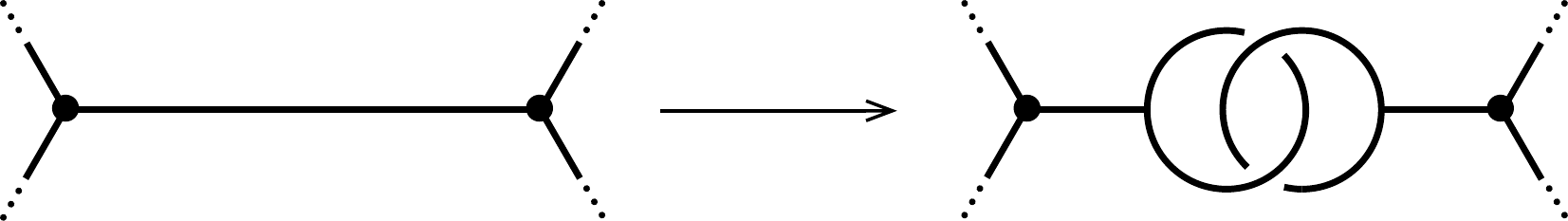}
 \caption{The fission rule}
 \label{fig:FissionRule}
\end{figure}%

The pair $(M_G,\gamma_G)$ obtained from $(M,\gamma)$ by surgery along $G$ is defined as follows:
Consider the case in which $G$ is a $Y$-graph.
Let $Y_0$ be a $Y$-graph in $\R^3$ and $N(Y_0)$ be its neighborhood, namely a genus three handlebody.
Then there is an orientation-preserving homeomorphism $h\colon N(Y_0) \to N(G)$ sending $Y_0$ to $G$.
Here, we replace $Y_0$ with the six-components framed link drawn in Figure~\ref{fig:YgraphSurgery}.
We denote by $M_G$ the manifold obtained from $G$ by surgery along the framed link in $N(G)$ corresponding to the above link, and let $\gamma_G$ denote the image of $\gamma$ in $M_G$.
In the general case, we perform surgery along each $Y$-graph of $Y(G)$.

\begin{figure}[h]
 \centering
 \includegraphics[width=0.8\columnwidth]{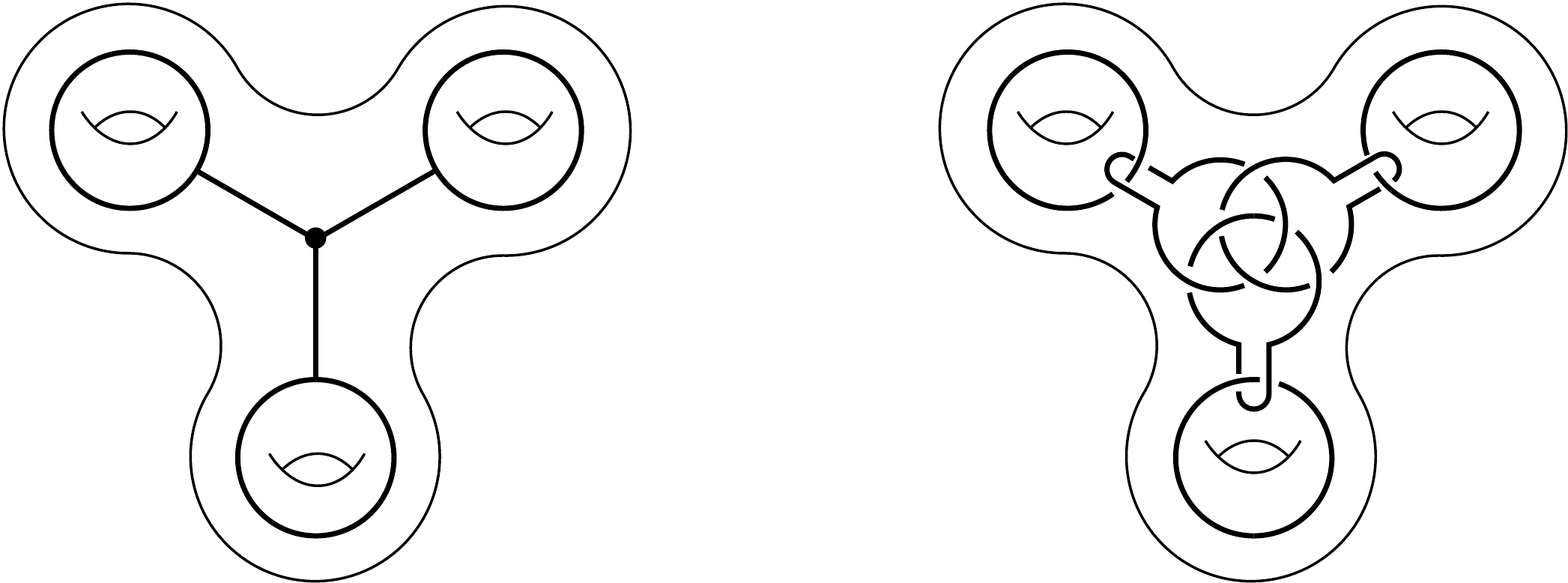}
 \caption{A $Y$-graph, its neighborhood and the corresponding link}
 \label{fig:YgraphSurgery}
\end{figure}%

\begin{definition}
 Let $k\ge1$.
 A pair $(M,\gamma)$ is \emph{$Y_k$-equivalent} to $(M',\gamma')$ if there exists a graph clasper $G$ such that the internal degree of each connected component is equal to $k$ and $G$ satisfies $(M_G,\gamma_G)=(M',\gamma')$.
\end{definition}

It is known that the $Y_k$-equivalence is an equivalence relation (see, for example, \cite[Theorem~3.2]{GGP01}).
Let us fix a $Y_1$-equivalence class $\M^0$ of 3-manifolds with tangles (see \cite[Theorem~7.6]{CHM08}).

\begin{definition}
 We define $\F_k(\M^0)$ to be the $\Q$-vector space spanned by $\{[M,G] \mid M \in \M^0,\ \ideg G=k\}$.
 Here, $[M,G]$ is defined by
 \[[M,G] := \sum_{G' \subset G}(-1)^{|G'|}M_{G'} \in \Q\M^0,\]
 where $G'$ runs over all connected components of $G$, and $|G'|$ is the number of connected components of $G'$.
 (The above sum consists of $2^{|G|}$ terms.)
\end{definition}

We have the \emph{$Y$-filtration}
\[\Q\M^0 = \F_0(\M^0) \supset \F_1(\M^0) \supset \dotsb,\] 
and we set $\G_i(\M^0) := \F_i(\M^0)/\F_{i+1}(\M^0)$.
Moreover, the associated graded vector space $\G(\M^0)$ is defined by $\G(\M^0) := \prod_{i\geq1}\G_i(\M^0)$.

\begin{definition}
 Let $V$ be a $\Q$-vector space and let $d\ge0$.
 A map $f \colon \M^0 \to V$ is a ($V$-valued) \emph{finite-type invariant} of \emph{degree} $d$ if the linear extension $\widetilde{f} \colon \Q\M^0 \to V$ is non-trivial on $\F_d(\M^0)$ and trivial on $\F_{d+1}(\M^0)$.
\end{definition}

\subsection{Proof of the universality}\label{subsec:ProofUniversal}
To state Theorem~\ref{thm:Gar02} below, we review the surgery map $\SS$ according to \cite{CHM08}.
Fix a $Y_1$-equivalence class $\M^0$ of $\LCob_q(w,v)$ and let $g:=|w^\bullet|$, $f:=|v^\bullet|$, $n:=|w^\circ|$.
Let $D$ be a Jacobi diagram in the vector space $\A^Y_i(\fc{g}^{+}\cup\fc{f}^{-}\cup\fc{n}^0)$.

We first construct an oriented surface $S(D)$ by replacing the internal vertices, external vertices and edges to disks, annuli and bands respectively, where the vertex orientation of $D$ induces the orientation of disks, and the bands connect the disks so that the orientations are compatible.
For each annulus, $A$, the core of $A$ is defined to be the push-off into $\Int A$ of the outer boundary of $A$ that is the boundary connecting with an edge.
Namely, the core is an oriented simple closed curve on $A$.

Next, we take a $q$-cobordism $(M,\sigma,m) \in \LCob_q(w,v)$.
The graph clasper $G(D)$ is defined to be the image of an embedding $S(D)$ into $M$ such that each annulus corresponding to the $j^-$- (resp.\ $j^+$-, $j^0$-) colored vertices is the push-off into $\Int M$ of a (unoriented) neighborhood of $m_{-}(\alpha_j) \subset \partial M$ (resp.\ $m_{+}(\beta_j)$, $m_{+}(\delta_j)$) and its core corresponds to the oriented curve $m_{-}(\alpha_j)$ (resp.\ $m_{+}(\beta_j)$, $m_{+}(\delta_j)$).
Here, $G(D)$ is called a \emph{topological realization} of a Jacobi diagram $D$.

\begin{figure}[h]
 \centering
 \includegraphics[width=0.9\columnwidth]{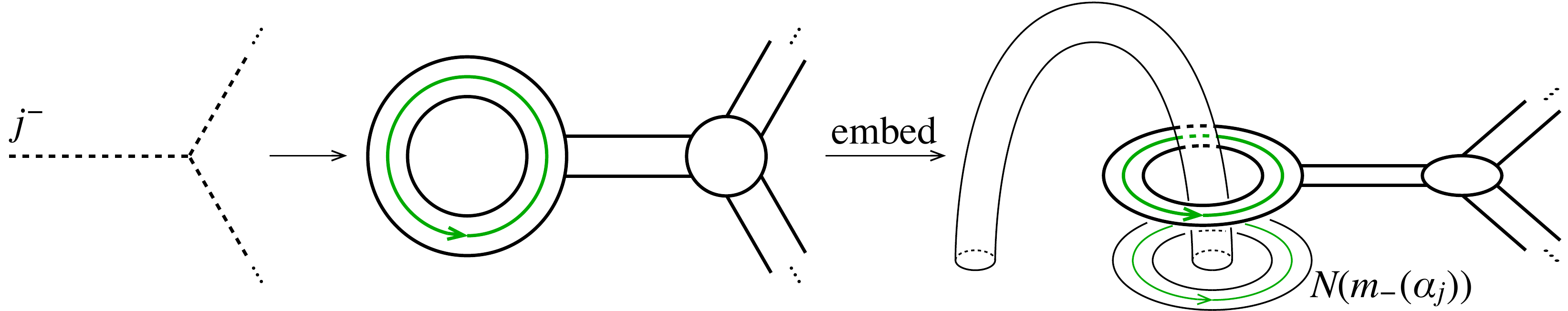}
 \caption{A Jacobi diagram, the corresponding graph clasper and its embedding into a 3-manifold}
 \label{fig:TopologicalRealization}
\end{figure}

The manifold $M_{G(D)}$ depends on the choice of the cobordism $(M,\sigma,m)$ and the topological realization $G(D)$.
However, it is known that $[M,G(D)] \in \G_i(\M^0)$ is independent of the choice, and the following theorem holds.

\begin{theorem}[{\cite[Corollary 1.4]{Gar02}}]\label{thm:Gar02}
 The graded linear map
 \[\SS \colon \A^Y(\fc{g}^{+}\cup\fc{f}^{-}\cup\fc{n}^0) \to \G(\M^0)\]
 defined by $D \mapsto [M,G(D)]$ is well-defined and surjective.
\end{theorem}

Using this theorem, we prove the next theorem saying that $\Zt$ is universal among rational-valued finite-type invariants.
The following proof is an obvious extension of the proof of \cite[Theorem 7.11]{CHM08}.

\begin{theorem}\label{thm:universal}
 Let $\A^Y:=\A^Y(\fc{g}^{+}\cup\fc{f}^{-}\cup\fc{n}^0)$.
 The map $\Zt^Y_i \colon \M^0 \to \A^Y\times\{\sigma\} = \A^Y$ is a finite-type invariant of degree at most $i$.
 Moreover, the induced map $\Gr\Zt^Y \colon \G(\M^0) \to \G(\A^Y)=\A^Y$ satisfies
 \[\Gr\Zt^Y\circ\SS(D)=(-1)^{i+c+e}D,\]
 where $i:=\ideg D$, $c:=|D|$ and $e$ is the number of internal edges of $D$.
 Consequently, $\SS$ and $\Gr\Zt$ are isomorphisms, and $\Zt^Y_i$ is a finite-type invariant of degree $i$.
\end{theorem}

\begin{proof}
 Let us prove that $\Zt^Y_{i-1}$ is a finite-type invariant of degree at most $i-1$.
 We first take a graph clasper $G$ in a cobordism $M=(M,\sigma,m) \in\M^0$ such that $[M,G] \in\F_i(\M^0)$.
 Let $i:=\ideg G$, $c:=|G|$ and $Y(G) = G_1\sqcup\dots\sqcup G_i$.
 Then we have
\begin{align*}
 [M,Y(G)] &= \sum_{J \subset\fc{i}} (-1)^{|J|} M_{\bigsqcup_{j\in J} G_j} \\
 &= \sum_{G'\subset G} (-1)^{|G\setminus G'|+|Y(G\setminus G')|} (-1)^{|Y(G')|} M_{G'} \\
 &= \sum_{G'\subset G} (-1)^{c-|G'|+i} M_{G'} = (-1)^{i+c}[M,G],
\end{align*}
 where the second equality follows from the equality
 \[\prod_{j=1}^{|G\setminus G'|} \left((1+t_j)^{|Y(G_j)|}-t_j^{|Y(G_j)|}\right)\Big|_{t_j=-1} = (-1)^{|G\setminus G'|+|Y(G\setminus G')|}\]
 and the fact that surgery along a part of a connected component is same as doing nothing.
 Using the transformation depicted in Figure~\ref{fig:InsertHopfLink}, the bottom-top tangle $\D^{-1}(M_G)$ can be expressed as Figure~\ref{fig:universal}, where $R$ is some cobordism of $\LCob(w,v\otimes((\bullet\bullet)\bullet)^{\otimes i})$.
 Moreover, $Y_1,\dots,Y_i$ are the bottom-top tangles illustrated in Figure~\ref{fig:YbyClasper}, which are same as the cobordism $Y$ listed in Table~\ref{tab:Gen}.

\begin{figure}[h]
 \centering
\begin{minipage}{0.45\columnwidth}
 \centering
 \includegraphics[width=0.6\columnwidth]{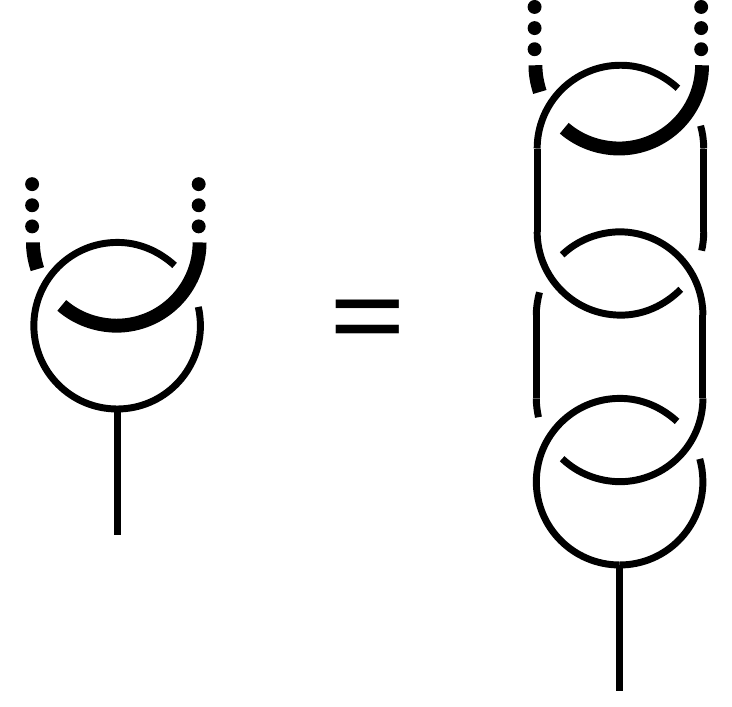}
 \caption{A transformation}
 \label{fig:InsertHopfLink}
\end{minipage}
\begin{minipage}{0.45\columnwidth}
 \centering
 \includegraphics[width=0.7\columnwidth]{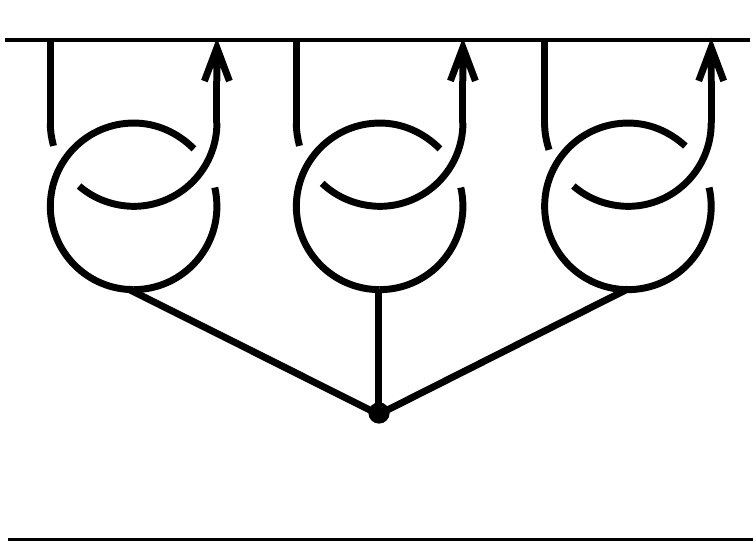}
 \caption{An expression of $Y$ by using a $Y$-graph}
 \label{fig:YbyClasper}
\end{minipage}
\end{figure}

\begin{figure}[h]
 \centering
 \includegraphics[width=0.7\columnwidth]{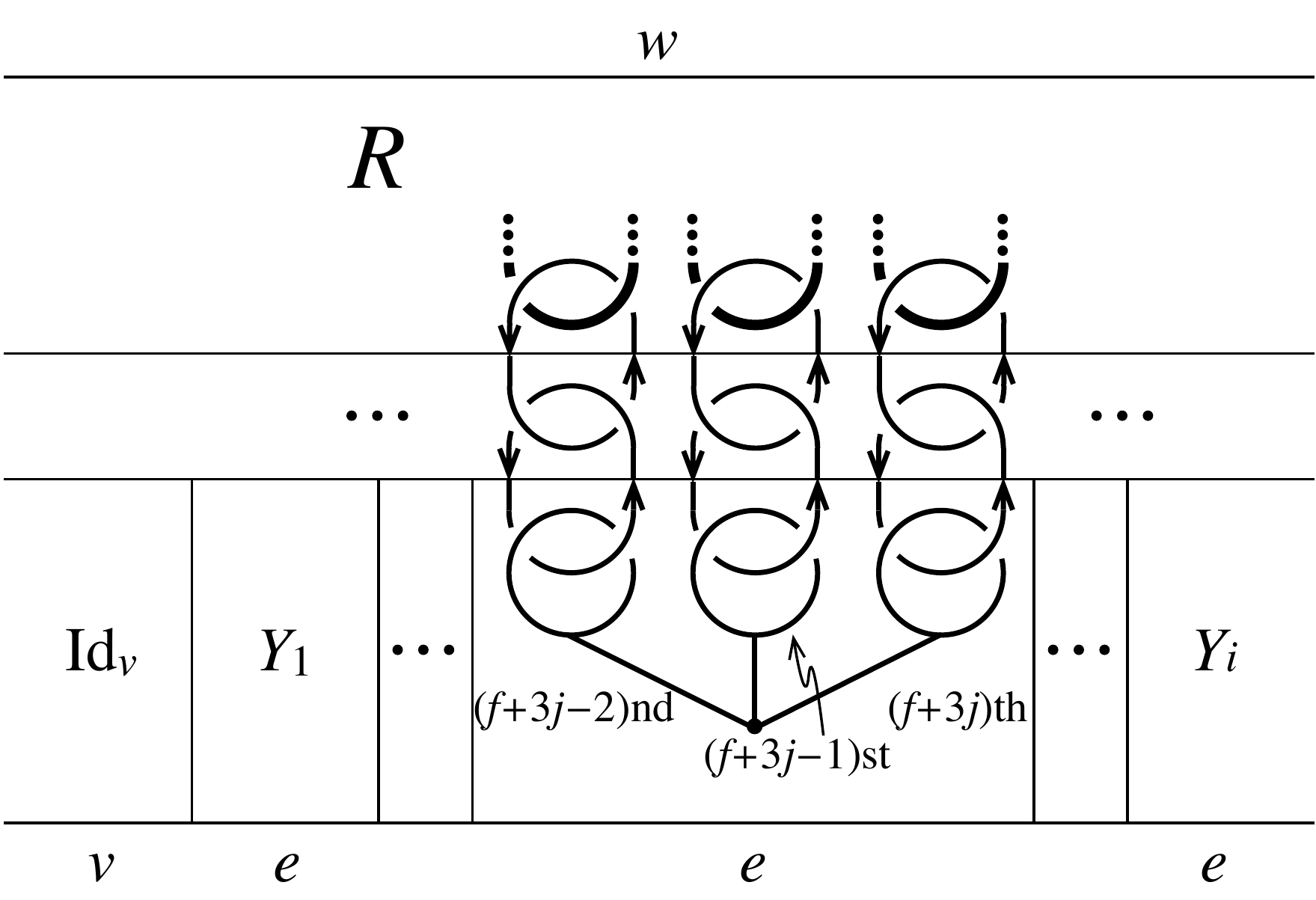}
 \caption{A bottom-top tangle (The thick lines are parts of tangles and claspers.)}
 \label{fig:universal}
\end{figure}
 
 Then we have, by the functoriality of $\Zt$,
\begin{align}\label{eq:Universal}
 & (-1)^{i+c}\Zt([M,G]) = \Zt([M,Y(G)]) \\
 &= \left(\Zt(\Id_v) \otimes \bigotimes_{j=1}^i \left( \Zt((\varepsilon\otimes\varepsilon)\otimes\varepsilon) -\Zt(Y_j) \right) \right) \circ \Zt(R) \notag \\
 &= \left( \left[\sum_{k=1}^f \strutgraph{k^+}{k^-}\right] \sqcup \bigsqcup_{j=1}^i \left(-\Zt_1(Y_j) +(\ideg>1)\right) \right) \circ \Zt(R). \notag
\end{align}
 Since the composition in $\tsA$ preserves the filtration $\{\tsA_{\geq i}\}_{i\geq0}$, we conclude $\Zt^Y_{i-1}([M,G])=0$.
 
 Next, we prove the equality in the second statement.
 We take a Jacobi diagram $D \in\A^Y_i$ and define $G$ to be a topological realization $G(D)$ in some cobordism $M$.
 The goal is to show $\Gr\Zt([M,G])=(-1)^{i+c+e}D$.
 We first define three sets by
\begin{align*}
 \LL &:= \bigcup_{j=1}^i \{f+3j-2, f+3j-1, f+3j\} \subset \fc{f+3i}, \\
 \OO &:= \{l \in \LL \mid \text{the $l$th leaf is originally a leaf}\}, \\
 \T &:= \{(t_1,t_2) \in \LL\times\LL \mid \text{the $t_1$th and $t_2$th leaves arise by the fission rule}\}.
\end{align*}
 Since the leaf corresponding to $l \in \OO$ arises from some external vertex of $D$, we define $c(l)$ to be its color.
 On the other hand, $\T$ is divided into $\T_<$ and $\T_>$, where $\T_\lessgtr:=\{(t_1,t_2) \in \T \mid t_1 \lessgtr t_2\}$ and $|\T_\lessgtr|=e$.
 We now have
\begin{align*}
 \Lk(R) = \Lk(M_G) +\sum_{l\in\OO}(E_{l^-,c(l)}+E_{c(l),l^-}) -\sum_{(t_1,t_2)\in\T}E_{t_1^-,t_2^-}\ (=:B),
\end{align*}
 where $E_{i,j}$ is the matrix unit.
 It follows from this equality, \eqref{eq:Universal} and Lemma~\ref{lem:Composition} that
\begin{align*}
 & (-1)^{i+c}\Zt^Y_i([M,G]) \\
 &= \ang{ \left( \bigsqcup_{j=1}^i -\Zt_1(Y_j) \middle/ \parbox{7.8em}{$k^+ \mapsto k^\ast +B^{--}k^{-} \\\hspace*{2.5em} +B^{+-}k^{-}$} \right), \left( [B^{--}/2+B^{-0}] \middle/ k^{-}\mapsto k^\ast \right) }_{\fc{f+3i}^\ast} \\ 
 &= \ang{ \left( \bigsqcup_{j=1}^i -\Zt_1(Y_j) \middle/ \parbox{11em}{$l^+ \mapsto
\begin{cases}
 c(l) & \text{if $l\in\OO$}, \\
 l^\ast & \text{otherwise.}
\end{cases}$} \right),
 \bigsqcup_{(t_1,t_2)\in\T_<} -\strutgraph{t_2^\ast}{t_1^\ast} }_{\{l^\ast \mid l\in\LL\setminus\OO\}} \\
 &= (-1)^e D,
\end{align*}
 where the second equality follows from the following argument:
 Since the linear combination on the left in the angular brackets contains no $k^0$-colored vertex, $[B^{-0}]$ can be moved to the left as Lemma~\ref{lem:PreComposition}.
\end{proof}

\section{Relations with knot theory}\label{sec:Knot}
We focus on cobordisms derived from knots with additional information.

\subsection{Cobordisms between two annuli}\label{subsec:Annuli}
We first review a natural construction of a cobordism from a knot according to \cite[Section~2.2]{CFK11}.
Let $Y$ be a homology sphere and let $K$ be an oriented framed knot in $Y$.
The boundary of $Y\setminus \Int N(K)$ is a torus with the oriented longitude determined by the orientation and framing of $K$.
Here, let $m\colon R^{\circ}_{\circ,\Id} \to \partial(Y\setminus \Int N(K))$ be a homeomorphism that maps the oriented simple closed curves $\delta_1$ and $\widehat{\delta_1}$ depicted in Figure~\ref{fig:CFK11} to the meridian and longitude of $K$, respectively.
Then we obtain $\C(Y,K):=(Y\setminus \Int N(K), \Id_{\SS_1}, m) \in \LCob_q(\circ,\circ)$.

\begin{figure}[h]
 \centering
 \includegraphics[width=0.2\columnwidth]{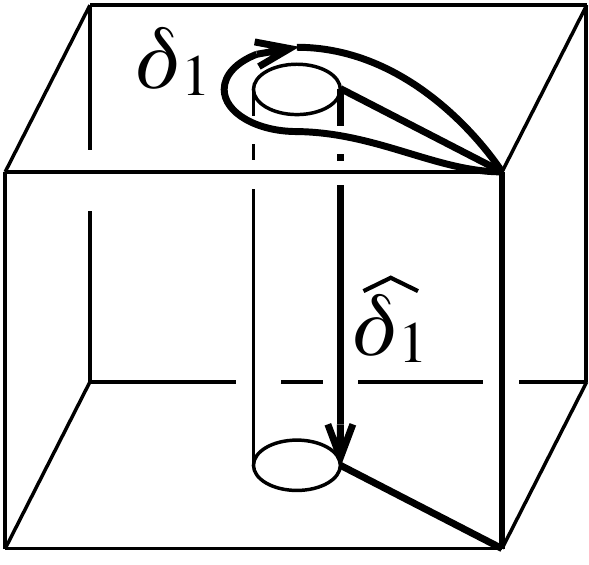}
 \caption{The reference surface $R^{\circ}_{\circ,\Id}$ with two curves $\delta_1$, $\widehat{\delta_1}$}
 \label{fig:CFK11}
\end{figure}

\begin{proposition}
 Let $K \subset Y$ be an oriented framed knot in a homology sphere.
 Then, under the canonical isomorphism $\A(\downarrow)\cong\A(\circlearrowleft)$,
 \[\chi_{\fc{1}^0}\Zt(\C(Y,K)) = \ZKLMO(Y\setminus\Int[-1,1]^3,K) \sharp \nu^{-1}.\]
\end{proposition}

\begin{proof}
 Let $(B,\gamma):=\D^{-1}(\C(Y,K))$.
 Then the closure $(\widehat{B},\widehat{\gamma})$ defined in Section~\ref{subsec:BTtangle} coincides with the pair $(Y,K)$.
 Therefore,
\begin{align*}
 \chi_{\fc{1}^0}\Zt(\C(Y,K)) &= \ZKLMO(Y\setminus\Int[-1,1]^3,\gamma) \\
 &= \ZKLMO(Y\setminus\Int[-1,1]^3,K) \sharp \nu^{-1}.
\end{align*}
 The proof is completed.
\end{proof}

Since $\Zt$ is a functor, we have
\[\Zt(\C(Y_1,K_1)) \circ \Zt(\C(Y_2,K_2)) = \Zt(\C(Y_1,K_1) \circ \C(Y_2,K_2)).\]
Here, $\C(Y_1,K_1) \circ \C(Y_2,K_2)$ can be interpreted as a certain connected sum defined as follows:
Consider an embedding between pairs $\iota_i \colon ([-1,1]^3,0\times0\times[-1,1]) \to (Y_i,K_i)$ that preserves the orientations and framing.
Then the above composition is equal to the pair
\[\left( (Y_1,K_1)\setminus\Int[-1,1]^3 \sqcup (Y_2,K_2)\setminus\Int[-1,1]^3 \right) / \iota_1(x,y,z)\sim\iota_2(x,y,-z).\]

\subsection{Cobordisms derived from a link and its Seifert surfaces}\label{subsec:SeifertSurf}
Let $Y$ be an integral homology sphere and let $L$ be an oriented link in $Y$.
One can take a Seifert surface $F$ of $L$, and let $g$ be its genus.
We fix an orientation-preserving homeomorphism $h\colon F_w \to F$, where $w \in \Mon(\bullet,\circ)$ satisfies $|w^\bullet|=g$, $|w^\circ|=|L|-1$.
Using the above information, we construct the cobordism $(M,\Id,m) \in \Cob(w,w)$ (that is not necessarily Lagrangian) as follows:
Define $M$ to be the manifold obtained from $Y \setminus \Int N(L)$ by cutting along $F$.
Now, $\partial M$ is the union of the surface $F_+$ to which the normal vectors point, $F_-$ and the $[-1,1]\times L$.
Therefore, we decide $m_+\colon F_w \to F_-$ and $m_-\colon F_w \to F_+$ by $h$, and the rest is uniquely determined by the meridians of $L$.

\begin{example}
 Let $L$ be the negative Hopf link in $S^3$.
 Since we need a Seifert surface that is easy to see, we represent the pair $(S^3,L)$ by surgery along the $(+1)$-framed unknot depicted in Figure~\ref{fig:HopfLinkSeifert}.
 One can now easily construct a Seifert surface by avoiding the surgery knot, and we give the parametrization indicated by the line drawn in the middle of Figure~\ref{fig:HopfLinkSeifert}.
 Then we have the Lagrangian cobordism $M$ illustrated in Figure~\ref{fig:HopfLinkSeifert}, where the oriented closed curve represents $h(\widehat{\delta_1})$.
 
 Hence, $\Zt(M)$ is exactly equal to $\chi_{\fc{1}^0}^{-1} \exp_\natural \left(\dfrac{1}{2}\ \raisegraph{-1em}{2.5em}{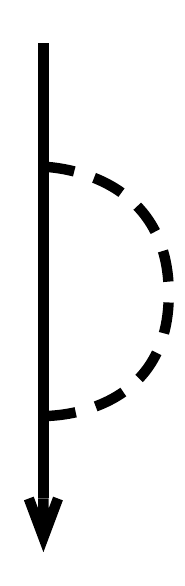}\right)$.
 It follows from \cite[Exercise~5.4]{BLT03} that
\begin{align*}
 \Zt(M) &= \Omega \sqcup \left(\exp_\sqcup \left(\frac{1}{2}\dstrutgraph{1^0}{1^0} \right) \right)
 = \left[\frac{1}{2}\dstrutgraph{1^0}{1^0} +\frac{1}{48}\dPhigraph{1^0}{1^0} +(\ideg>2)\right].
\end{align*}
\end{example}

\begin{figure}[h]
 \centering
 \includegraphics[width=0.9\columnwidth]{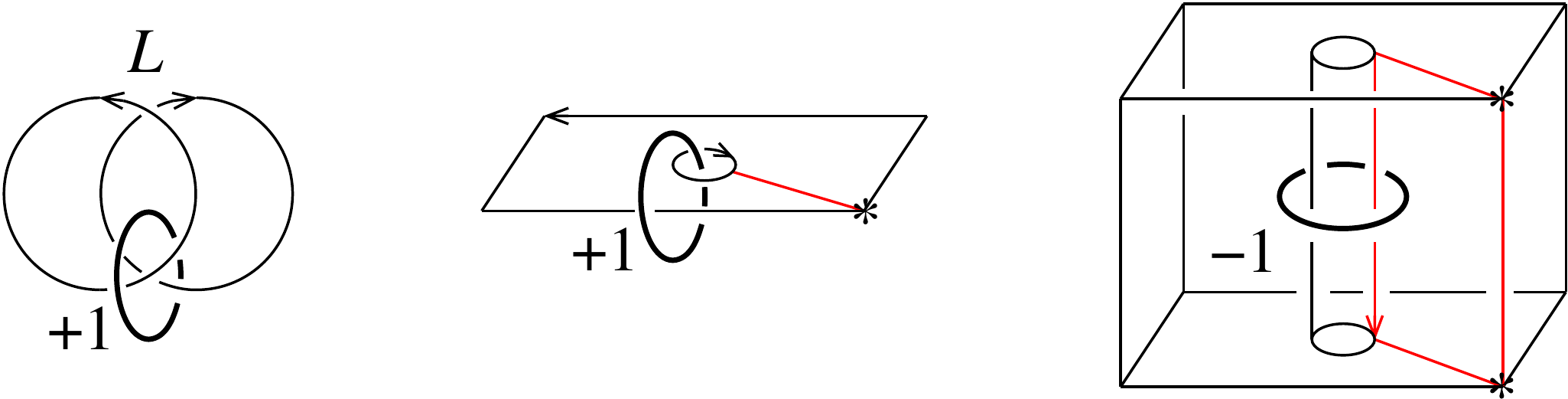}
 \caption{A Hopf link $L$, its Seifert surface and the cobordism obtained by cutting}
 \label{fig:HopfLinkSeifert}
\end{figure}%

\subsection{Milnor invariants of string links}\label{subsec:MilnorInv}
Habegger and Masbaum~\cite{HaMa00} proved that the first non-vanishing Milnor invariant of a string link appears in the tree reduction of its Kontsevich invariant.
The same is true for the Kontsevich-LMO invariant of string links in a homology cube, which was shown in \cite{Mof06}.
Moreover, the same holds for the LMO functor of the cobordisms derived from string links in a homology cube, which was proven in \cite{CHM08}. 
We prove that the same is true for the extension of the LMO functor.

We first review the Milnor invariants of string links along \cite{CHM08}.

\begin{definition}
 A \emph{string link} on $l$ strands is an orientation-reversed bottom-top tangle of type $(w,w,\Id_{\SS_l})$, where $w$ is a word in $\Mon(\circ)$ of length $l$.
 Namely, each strand runs from $r_i\times(-1)$ to $r_i\times1$.
\end{definition}

Actually, we are interested in the monoid
\[\SSS_l := \{\text{string link $(B,\sigma)$ on $l$ strands} \mid \text{$B$ is an \emph{integral} homology cube}\}.\]
Let $(B,\sigma) \in \SSS_l$.
Then $(S,\Id_{\SS_l},s):=\D(B,\sigma)$ is a cobordism from $F_w$ to $F_w$.
However, in this subsection, we use the surface $D_l$ with the oriented simple closed curves $x_1,\dots,x_l$ as in Figure~\ref{fig:D} instead of $F_w$ with the oriented simple closed curves $\delta_1,\dots,\delta_l$
Note that the orientation of $x_i$ differs from the one of $\delta_i$.

\begin{figure}[h]
 \centering
 \includegraphics[width=0.4\columnwidth]{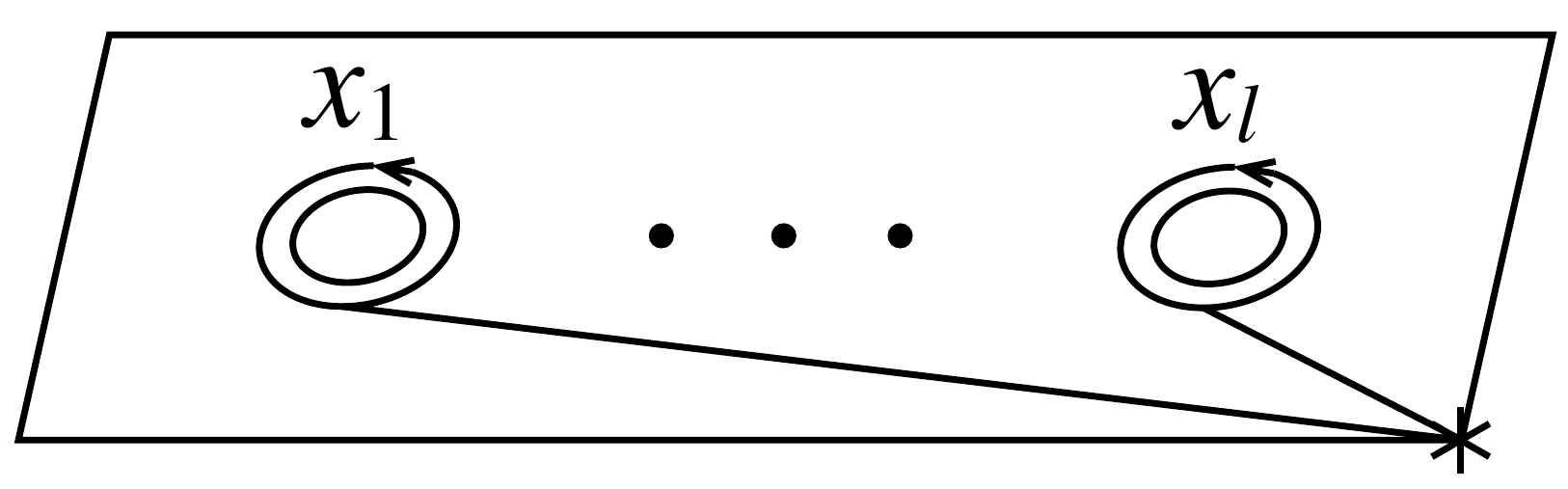}
 \vspace{-0.8em}
 \caption{The surface $D_l = \Sigma_{0,l+1}$ with the closed curves $x_1,\dots,x_l$}
 \label{fig:D}
\end{figure}

Let $\varpi$ denote the fundamental group $\pi_1(D_l,\ast)$ and $\varpi_n$ is its lower central series defined by $\varpi_k:=[\varpi_{k-1},\varpi]$, $\varpi_1=\varpi$.
Since the continuous maps
\[s_{+},\ s_{-}\colon D_l \hookrightarrow \partial(D_l\times[-1,1]) \xrightarrow{s} S\]
induce isomorphisms on their first homologies and surjections on their second homologies (\cite[Theorem~5.1]{Sta65}), they induce isomorphisms
\[s_{+,\ast},\ s_{-,\ast}\colon \varpi/\varpi_k \to \pi_1(S)/\pi_1(S)_k\]
for all $k\ge1$.
The $k$th \emph{Artin representation} is the monoid anti-homomorphism $A_k\colon \SSS_l \to \Aut(\varpi/\varpi_{k+1})$ defined by $A_k(B,\sigma) := s_{+,\ast}^{-1} \circ s_{-,\ast}$.
We set $\SSS_l[k]:=\Ker A_k$, and the filtration
\[\SSS_l=\SSS_l[1] \supset \SSS_l[2] \supset \dotsb\]
is obtained.

From now on, we suppose that each strand of $\sigma$ is 0-framed, and let $\lambda_i$ be the preferred longitude of the $i$th strand of $\sigma$.
It is well known that $(B,\sigma) \in \SSS_l$ belongs to $\SSS_l[k]$ if and only if $\lambda_i \in \pi_1(S)_k$ for all $i$ (see, for example, \cite[Remark~5.1]{HaMa00}).

\begin{definition}
 The $k$th \emph{Milnor invariant} is the monoid homomorphism $\mu_k\colon \SSS_l[n] \to \varpi/\varpi_2 \otimes \varpi_k/\varpi_{k+1}$ defined by
\[\mu_k(B,\sigma) := \sum_{i=1}^l x_i \otimes s_{+,\ast}^{-1}(\lambda_i).\]
\end{definition}

The following lemma is well known and directly follows from the above fact.

\begin{lemma}\label{lem:mu}
 Let $(B,\sigma) \in \SSS_l[k]$.
 Then $\mu_k(B,\sigma)=0$ if and only if $(B,\sigma)$ belongs to $\SSS_l[k+1]$.
\end{lemma}

Next, we regard $\mu_k(B,\sigma)$ as a linear combination of connected tree Jacobi diagrams.
The subalgebra of $\A(C)$ generated by tree Jacobi diagrams is denoted by $\A^t(C)$, which is identified with the quotient by the ideal generated by looped Jacobi diagrams, and the image of $x \in \A(C)$ by the quotient map is denoted by $x^t$.
Moreover, the subspace of $\A^t(C)$ spanned by connected Jacobi diagrams is denoted by $\A^{t,c}(C)$.

The free Lie algebra over $\Q$ generated by a set $C$ is denoted by $\Lie(C)$ and its degree $n$ part is denoted by $\Lie_n(C)$.
We define the linear map $\eta_n\colon \A^{t,c}_n(C) \to \Q C \otimes_{\Q} \Lie_n(C)$ by
\[\eta_k(D) := \sum_{\text{$v$}} (\text{color of $v$}) \otimes \comm(D_v),\]
where $v$ runs over all external vertices in $D$, and $D_v$ is the tree rooted at $v$.
Here, $\comm(D_v)$ is defined as follows:
\[\comm \left( \raisegraph{-2.3em}{5em}{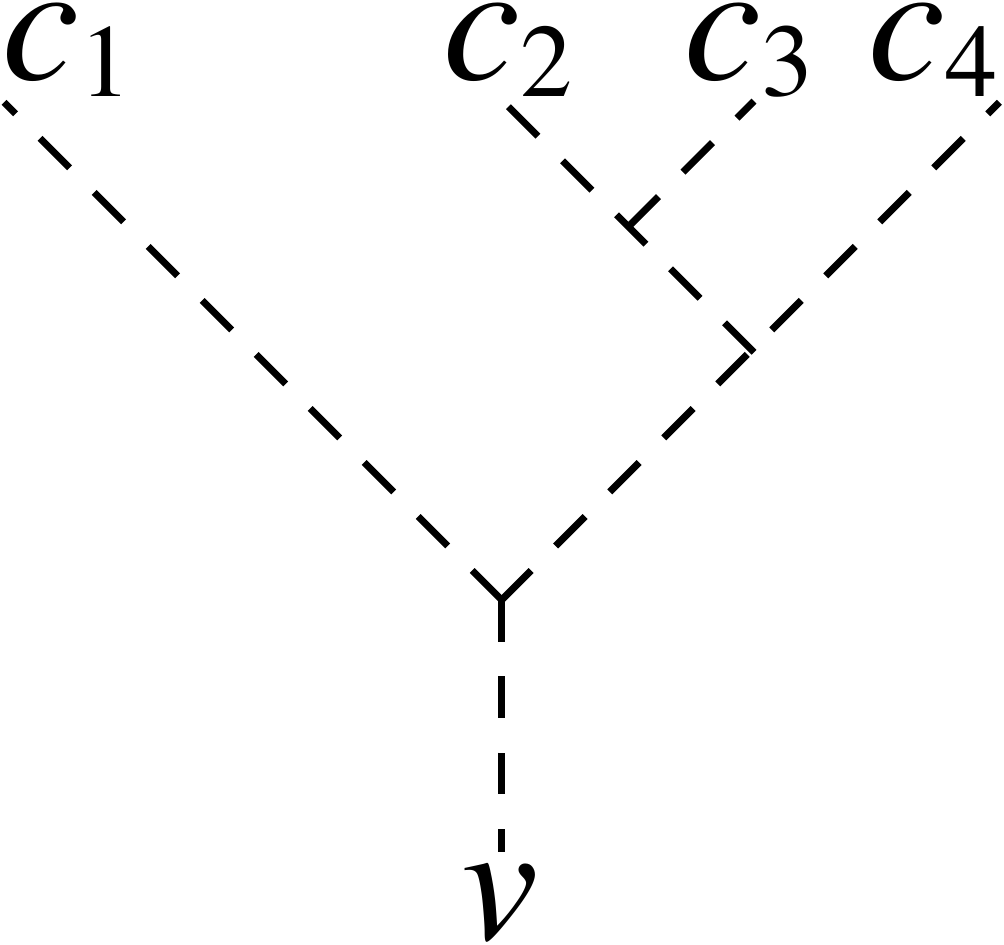} \right) = [c_1,[[c_2,c_3],c_4]].\]
Moreover, it is well known that the sequence
\begin{align}\label{eq:EtaBracket}
 0 \to \A^{t,c}_{k-1}(C) \xrightarrow{\eta_{k-1}} \Q C \otimes_\Q \Lie_k(C) \xrightarrow{[-,-]} \Lie_{k+1}(C) \to 0
\end{align}
is exact (see, for example, \cite[Theorem 1]{Lev02}).

\begin{remark}
 The map $\eta_k$ is same as in \cite[Section~4.3]{HaMa12}, which differs from \cite{HabN00}, \cite{Lev02}, \cite{Mof06} and \cite{CHM08} by $(-1)^k$.
\end{remark}

Let $G$ be the free group generated by $C=\{c_1,\dots,c_r\}$ and let $G_k$ be its lower central series.
Recall the well-known fact that there exists a natural isomorphism $\alpha_n\colon (G/G_2 \otimes G_k/G_{k+1}) \otimes \Q \to \Q C \otimes_\Q \Lie_k(C)$.
In particular, if $C=\fc{l}^\ast$, then one can show that $[-,-] \circ \alpha_k \circ \mu_k = 0$.
The exact sequence \eqref{eq:EtaBracket} enable us to define $\mu_k^{\A}(B,\sigma) \in \A^{t,c}_{k-1}(C)$ to be the unique element sent to $\alpha_k(\mu_k(B,\sigma))$ by $\eta_{k-1}$.

Finally, we discuss a relation between cobordisms and string links.
Let $w \in \Mon(\bullet,\circ)$ and let $g:=|w^\bullet|$, $n:=|w^\circ|$, $l:=2g+n$. 
Figure~\ref{fig:MilnorJohnson} represents a map from $\{(M,\tau,m) \in \LCob(w,w) \mid \tau=\Id_{\SS_n}\}$ to $\SSS_l$, which sends $\Id_w$ to the trivial string link on $l$ strands, which is an extension of the original one defined in \cite{HabN00}, \cite{CHM08}.

\begin{figure}[h]
 \centering
 \includegraphics[width=0.9\columnwidth]{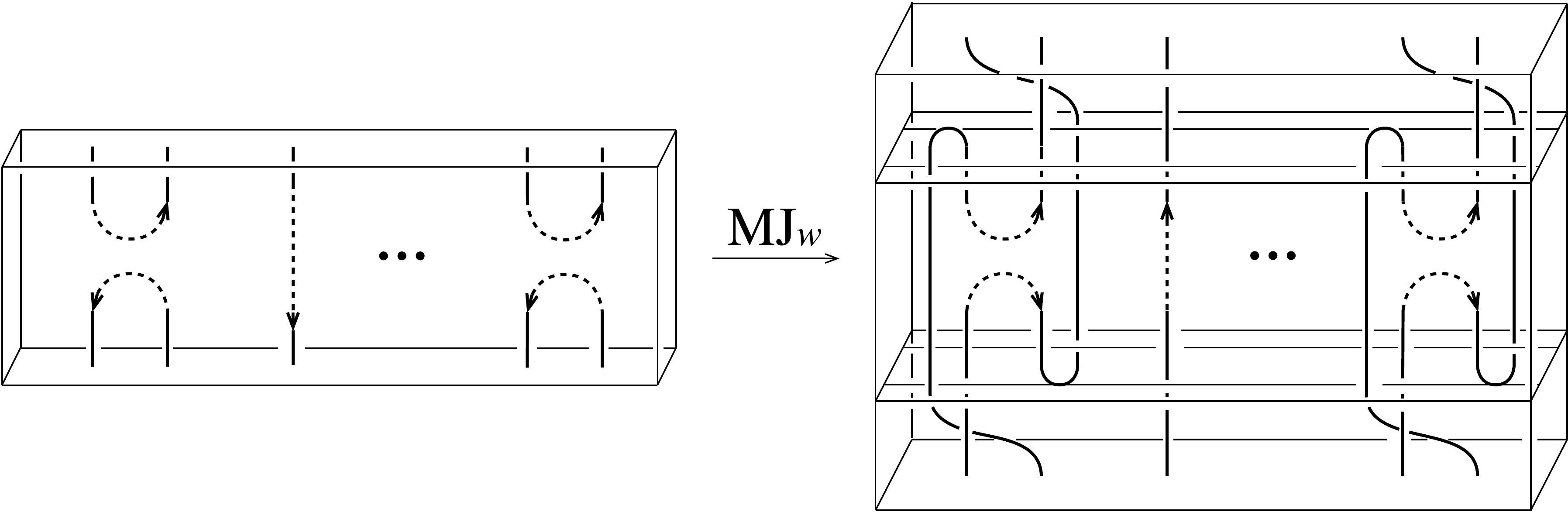}
 \caption{The Milnor-Johnson correspondence}
 \label{fig:MilnorJohnson}
\end{figure}

\begin{definition}
 The \emph{Milnor-Johnson correspondence} $\MJ_w$ is the bijection from the $Y_1$-equivalence class of $\Id_w$ in $\LCob(w,w)$ to the $Y_1$-equivalence class of the trivial string link in $\SSS_l$, depicted in Figure~\ref{fig:MilnorJohnson}.
\end{definition}

Let us introduce a map $R_w$ for $w \in \Mon(\bullet,\circ)$ to state the main theorem in this section.
We first define the bijection $\rho_w\colon \fc{g}^{+}\cup\fc{g}^{-}\cup\fc{n}^0 \to \fc{2g+n}^\ast$ by sending the $j$th color in the sequence $\left(w / \text{$i$th}\ \bullet \mapsto i^{-}i^{+},\ \text{$i$th}\ \circ \mapsto i^0 \right)$ to $j^\ast$.
Next, $R_w\colon \A(X,\fc{g}^{+}\cup\fc{g}^{-}\cup\fc{n}^0) \to \A(X,\fc{2g+n}^\ast)$ is defined by
\[R_w(D) := \left( D \middle/ i^{-}\mapsto -\rho_w(i^-),\ i^{+}\mapsto \rho_w(i^+),\ i^0\mapsto -\rho_w(i^0) \right).\]
For example, if $w=\bullet\circ$, then we have
\[R_w(D) = \left( D \middle/ 1^{-}\mapsto -1^\ast,\ 1^{+}\mapsto 2^\ast,\ 1^0\mapsto -3^\ast \right).\]

\begin{figure}[h]
 \centering
 \includegraphics[width=0.8\columnwidth]{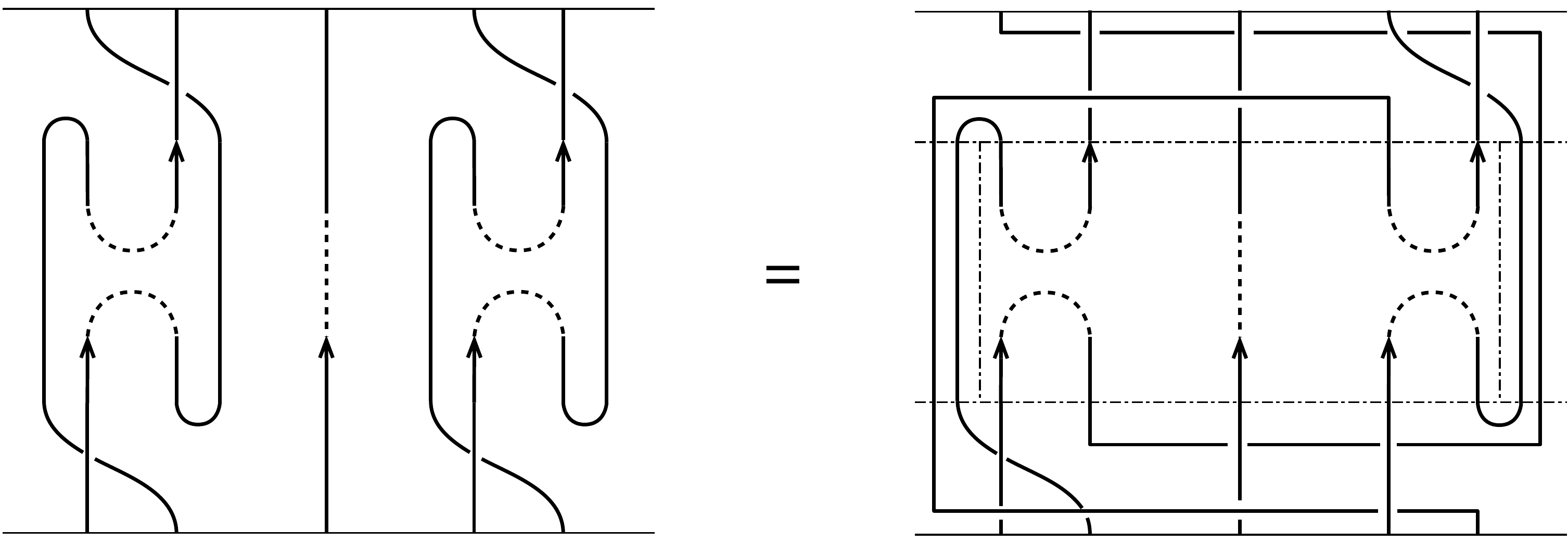}
 \caption{An example of a transformation with $w=\bullet\circ\bullet$}
 \label{fig:Psi}
\end{figure}

\begin{lemma}\label{lem:Psi}
 There exists a linear map
 \[\Psi\colon \A(\curvearrowru^{\fc{g}^+} \curvearrowleft^{\fc{g}^-} \downarrow^{\fc{n}^0}) \to \A(\downarrow^{\fc{2g+n}^\ast})\]
 such that $\Psi(Z(\D^{-1}(M))) = \ZKLMO(\MJ_w(M))$.
 Moreover, $\Psi$ satisfies
 \[\chi_{\fc{2g+n}^\ast}^{-1}\Psi(E) = \left[\sum_{i=1}^g \strutgraph{\rho_w(i^+)}{-\rho_w(i^-)} \qquad \right] \sqcup R_w(\widetilde{E}) + (\ideg>\ideg E),\]
 for any Jacobi diagram $E \in \A(\curvearrowru^{\fc{g}^+} \curvearrowleft^{\fc{g}^-} \downarrow^{\fc{n}^0})$.
\end{lemma}

\begin{proof}
 The string link $\MJ_w(M)$ can be transformed as illustrated in Figure~\ref{fig:Psi}.
 Therefore, we set 
 \[\Psi(E) := \ZKLMO(\text{bottom part}) \circ (\Id \otimes S_{\fc{g}^{-}\cup\fc{n}^0}(E) \otimes \Id) \circ \ZKLMO(\text{top part}),\]
 where $S_{\fc{g}^{-}\cup\fc{n}^0}(E)$ is the element obtained from $E$ by applying the orientation reversal map $S$ to the $(\fc{g}^{-}\cup\fc{n}^0)$-indexed components of $E$.
 Then, $\Psi$ has the desired property.
\end{proof}

\begin{theorem}\label{thm:MilnorInv}
 Let $l\ge1$, $k\ge2$, $(B,\sigma) \in \SSS_l[k]$ and let $w \in \Mon(\bullet,\circ)$ such that $l=2|w^\bullet|+|w^\circ|$.
 Then $\Zt^{Y,t}_{<k}(\MJ_w^{-1}(B,\sigma))$ is equal to
 \[\varnothing +R_w^{-1}(\mu_k^{\A}(B,\sigma)) \in \A^Y_{<k}(\fc{g}^{+}\cup\fc{g}^{-}\cup\fc{n}^0).\]
 Conversely, if $(B,\sigma) \in \SSS_l[2]$ and $\Zt^{Y,t}_{<k}(\MJ_w^{-1}(M,\sigma))$ is of the form
 \[\varnothing + x \quad(\ideg x=k-1),\]
 then $(B,\sigma)$ belongs to $\SSS_l[k]$, and $R_w(x)$ is equal to $\mu_k^{\A}(B,\sigma)$.
\end{theorem}

\begin{proof}
 Let $(B,\sigma) \in \SSS_l[k]$ and let $g=|w^\bullet|$, $n=|w^\circ|$.
 We set $(M,\Id_{\SS_n},m):=\MJ_w^{-1}(B,\sigma)$.
 Let us prove the former by induction on $k$.
 We first assume that it is true up to $k-1$, where $k\ge3$.
 (The case $k=2$ is mentioned later.)
 It follows from Lemma~\ref{lem:mu} that $\Zt(M)$ is written as
 \[[I_g^{+-}] \sqcup \left(\varnothing +x +(\text{$\ideg \ge k$ or looped}) \right),\]
 where $\deg x=k-1$.
 Hence, Lemma~\ref{lem:Composition} implies
\begin{align*}
 & \Zt(M) \circ_{\tsA} \TT_g^{-1} \\
 &= \TT_g^{-1} + [I_g^{+-}] \sqcup \ang{\left( x +(\text{$\ideg \ge k$ or looped}) \middle/ \parbox{3em}{$i^{+}\mapsto \\ i^\ast +i^{+}$}\right), \left( (\TT_g^{-1})^Y \middle/ \parbox{3em}{$i^{-}\mapsto \\ i^\ast +i^{-}$} \right)}_{\fc{g}^\ast} \\
 &= \TT_g^{-1} + [I_g^{+-}] \sqcup x +(\text{$\ideg \ge k$ or looped}).
\end{align*}
 By applying the composite map $\chi_{\fc{2g+n}^\ast}^{-1} \Psi \chi_{\fc{g}^{+}\cup\fc{g}^{-}\cup\fc{n}^0}$, we have
\begin{align*}
 & \chi_{\fc{2g+n}^\ast}^{-1} \Psi \chi_{\fc{g}^{+}\cup\fc{g}^{-}\cup\fc{n}^0} (\Zt(M) \circ \TT_g^{-1}) \\
 &= \chi^{-1}\Psi\chi(\TT_g^{-1}) + \left[\sum_{i=1}^g \strutgraph{\rho_w(i^+)}{\rho_w(i^-)} \qquad \right] \sqcup R_w\left( \widetilde{\chi([I_g^{+-}] \sqcup x)} \right) + \left(\parbox{4.3em}{$\ideg \ge k$ \\ \text{or looped}}\right) \\
 &= \chi^{-1}\Psi(Z(\Id_w)) + \left[\sum_{i=1}^g \strutgraph{\rho_w(i^+)}{\rho_w(i^-)} \qquad \right] \sqcup R_w\left( [I_g^{+-}] \sqcup x \right) +\left(\parbox{4.3em}{$\ideg \ge k$ \\ \text{or looped}}\right) \\
 &= \varnothing +R_w(x) +(\text{$\ideg \ge k$ or looped}),
\end{align*}
 where we use Lemma~\ref{lem:Psi} and the fact that $\chi^{-1}\Psi\chi(\text{looped}) \subset (\text{looped})$.
 On the other hand, by \cite[Theorem~3]{Mof06},
\begin{align*}
 \chi^{-1}\Psi\chi (\Zt(M) \circ \TT_g^{-1})
 &= \chi^{-1}\Psi Z(M) \\
 &= \chi^{-1}\ZKLMO(\MJ_w(M)) \\
 &= \varnothing + \mu_n(B,\sigma) +(\text{$\ideg \ge k$ or looped}).
\end{align*}
 Comparing the internal degree $k$ parts of each tree reduction, we conclude $\mu_k(B,\sigma)=x$.
 
 Here, we have to consider the case $n=2$.
 In general, $\Zt^{Y,t}(M)$ is of the form $\varnothing +x +(\ideg>1)$ where $\ideg x=1$.
 Hence, applying the above argument, we conclude $x=\mu_2(B,\sigma)$.
 
 The latter follows from the former and Lemma~\ref{lem:mu}.
\end{proof}

\begin{example}\label{ex:Borromean}
 Let us compute the second Milnor invariant of $\sigma=([-1,1]^3,\sigma) \in \SSS_3[2]$ illustrated in Figure~\ref{fig:Borromean}.
 
\begin{figure}[h]
 \centering
 \includegraphics[width=0.9\columnwidth]{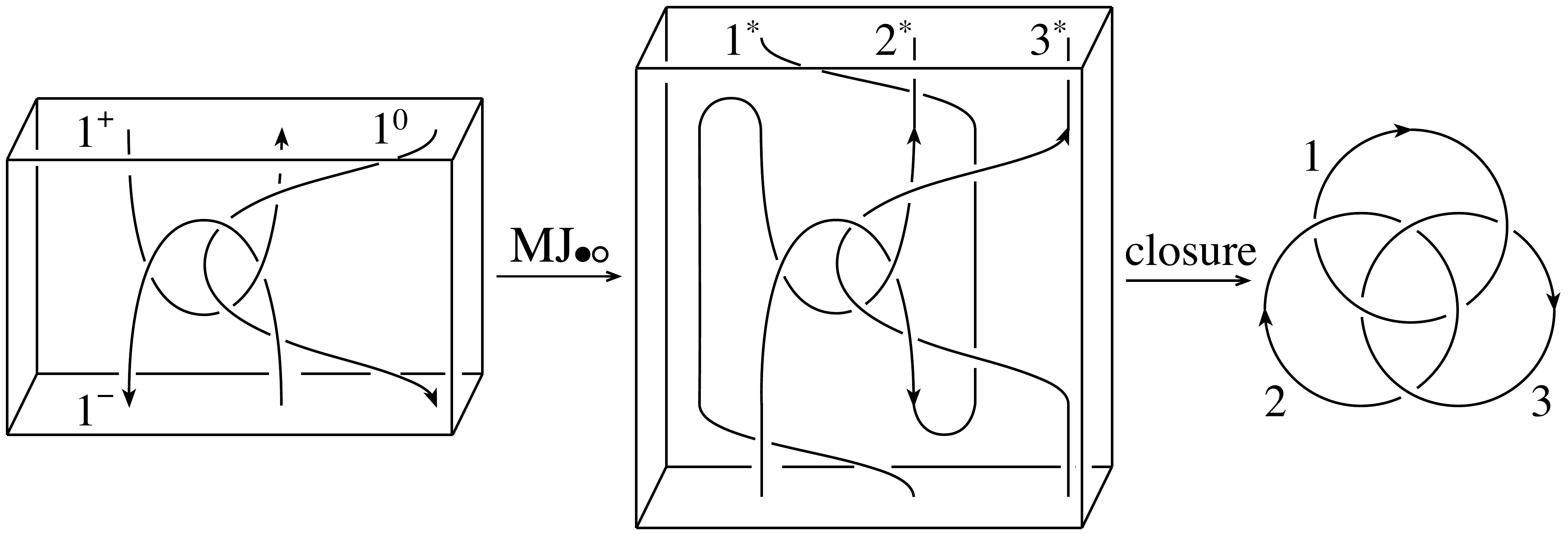}
 \caption{An examples of a bottom-top tangle, the corresponding string link and its closure}
 \label{fig:Borromean}
\end{figure}%
 
 First, one can check that $\MJ_{\bullet\circ}^{-1}(\sigma) = \psi_{\circ,\bullet}\circ\psi_{\bullet,\circ}$, where $\psi_{\circ,\bullet}$ and $\psi_{\bullet,\circ}$ are the cobordisms define in Table~\ref{tab:Gen}.
 It follows from the functoriality of $\Zt$, Lemma~\ref{lem:Composition} and Table~\ref{tab:Value} that
\begin{align*}
 \Zt^Y(\MJ_{\bullet\circ}^{-1}(\sigma)) &= \chi_{1^0}^{-1}\chi_{1^0,1^{0'}} \left( \varnothing +\frac{1}{2}\rTgraph{1^+}{1^-}{1^0} + \frac{1}{2}\rTgraph{1^+}{1^-}{1^{0'}} +(\ideg>1) \right) \\
 &= \varnothing + \rTgraph{1^+}{1^-}{1^0} + (\ideg>1).
\end{align*}
 Thus the previous theorem implies $\mu_2^{\A}(\sigma) = (-1)^2\rTgraph{2^\ast}{1^\ast}{3^\ast}$.
 Therefore,
 \[\mu_2(\sigma) = x_1\otimes[x_2,x_3] +x_2\otimes[x_3,x_1] +x_3\otimes[x_1,x_2] \in \varpi/\varpi_2 \otimes \varpi_2/\varpi_3.\]
 
 Moreover, by \cite[Remark in page~1160]{CHM08}, the Milnor $\overline{\mu}$-invariants of length 3 of the closure of $\sigma$, which is the Borromean rings, are as follows:
\[\overline{\mu}_{\widehat{\sigma}}(j_1,j_2;i)= \begin{cases}
 \sgn(j_1\;j_2\;i) & \text{if $\{j_1,j_2,i\}=\{1,2,3\}$,} \\
 0 & \text{otherwise.}
\end{cases}\]
\end{example}

\def\cprime{$'$} \def\cprime{$'$} \def\cprime{$'$}

\end{document}